\documentclass{amsart}
\usepackage{a4wide}
\usepackage{pgf,tikz,pgfplots}
\pgfplotsset{compat=1.15}
\usepackage{mathrsfs}
\usetikzlibrary{arrows}

\usepackage[T1]{fontenc}
\usepackage{microtype}
\usepackage{lmodern}
\usepackage[colorlinks=true,urlcolor=blue, citecolor=red,linkcolor=blue,linktocpage,pdfpagelabels, bookmarksnumbered,bookmarksopen]{hyperref}
\usepackage[hyperpageref]{backref}
\usepackage{amsthm} 
\usepackage{latexsym,amsmath,amssymb}

\usepackage{accents}
\usepackage{esint}
\usepackage[colorinlistoftodos,prependcaption,textsize=small]{todonotes}

\usepackage{soul}
\usepackage{mathtools} 
\usepackage{xparse} 
\usepackage[capitalize]{cleveref}

\usepackage[shortlabels]{enumitem}
\allowdisplaybreaks

\usepackage[textsize=small]{todonotes}
\title[On uniqueness for the hyperbolic Half-wave map equation]{O\lowercase{n uniqueness for hyperbolic half-wave maps in dimension} \texorpdfstring{$\lowercase{d} \geq 3$}{d >= 3}}

\renewcommand{\S}{{\mathbb S}}

\newtheorem{theorem}{Theorem}
\newtheorem{lemma}[theorem]{Lemma}

\theoremstyle{definition}
\newtheorem{definition}[theorem]{Definition}

\theoremstyle{remark}
\newtheorem{remark}[theorem]{Remark}

\newcommand{\HH}{\mathbb{H}}
\newcommand{\R}{\mathbb{R}}

\newcommand{\Z}{\mathbb{Z}}

\newcommand{\brac}[1]{\left (#1 \right )}

\newcommand{\scpr}[2]{\left \langle #1 , #2\right \rangle}
\newcommand{\abs}[1]{\left\lvert #1 \right \rvert}

\renewcommand{\vec}[1]{{\bf #1}}

\newcommand{\barint}{
\rule[.036in]{.12in}{.009in}\kern-.16in \displaystyle\int }

\newcommand{\barcal}{\text{$ \rule[.036in]{.11in}{.007in}\kern-.128in\int $}}

\def\mvint_#1{\mathchoice
          {\mathop{\vrule width 6pt height 3 pt depth -2.5pt
                  \kern -8pt \intop}\nolimits_{\kern -3pt #1}}%
          {\mathop{\vrule width 5pt height 3 pt depth -2.6pt
                  \kern -6pt \intop}\nolimits_{#1}}%
          {\mathop{\vrule width 5pt height 3 pt depth -2.6pt
                  \kern -6pt \intop}\nolimits_{#1}}%
          {\mathop{\vrule width 5pt height 3 pt depth -2.6pt
                  \kern -6pt \intop}\nolimits_{#1}}}

\numberwithin{theorem}{section} \numberwithin{equation}{section}

\newcommand{\aleq}{\lesssim}

\newcommand{\defeq}{:=}

\newcommand{\Ds}[1]{|\nabla|^{#1}}
\newcommand{\Dso}{|\nabla|}

\usepackage{scalerel}[2014/03/10]
\usepackage[usestackEOL]{stackengine}
\def\avint{\,\ThisStyle{\ensurestackMath{%
			\stackinset{c}{.2\LMpt}{c}{.5\LMpt}{\SavedStyle-}{\SavedStyle\phantom{\int}}}%
		\setbox0=\hbox{$\SavedStyle\int\,$}\kern-\wd0}\int}

\let\latexchi\chi
\makeatletter
\renewcommand\chi{\@ifnextchar_\sub@chi\latexchi}
\newcommand{\sub@chi}[2]{
  \@ifnextchar^{\subsup@chi{#2}}{\latexchi^{}_{#2}}%
}
\newcommand{\subsup@chi}[3]{
  \latexchi_{#1}^{#3}%
}
\makeatother

\author{Silvino Reyes Farina}
\email[Silvino Reyes Farina]{sir25@pitt.edu}

\begin{document}
\begin{abstract}
Half-wave maps appear in the physics literature as the continuum limit of Calogero-Moser spin systems. We obtain a uniqueness result for the Half-Wave Maps equation in dimension $d \ge 3$ in the natural energy class with $\HH^2$ target. In the proof, we differentiate in time to arrive at a wave-type equation and isometrically embed $\HH^2$ into some $\R^m$ using the Nash embedding theorem. Relying on geometric properties of $\HH^2$, combined with fractional Leibniz rules and commutator estimates, we then use a Gr\"{o}nwall inequality argument to obtain uniqueness.

\end{abstract}

.
\maketitle 
\tableofcontents

\section{Introduction and Main Result}
The Half-Wave Maps equation (HWM) arises in the physics literature as the continuum limit of a Calogero-Moser spin system \cite{lenzmann2020derivation}. Its homogeneous form is:
\begin{align*}
    \partial_t \vec{u} = \vec{u} \wedge (-\Delta)^{\frac{1}{2}}\vec{u},
\end{align*}
where solutions -- referred to as half-wave maps -- are functions $\vec{u}$ that map from Minkowski space-time $\R \times \R^d$ into a smooth manifold $M$ (typically the 2-sphere $\S^2$), the half-Laplacian operator is denoted by $(-\Delta)^{\frac{1}{2}}$, and $\wedge$ is the usual cross-product in $\R^3$. In the last few years, there has been much research done on the HWM. In \cite{lenzmann2020derivation}, Lenzmann and Sok present a rigorous derivation of the HWM and study convergence to global-in-time weak solutions in the energy class as well as short-time strong solutions of higher regularity. Schikorra and Lenzmann also give a complete classification of all traveling solitary waves
with finite energy for the HWM in \cite{lenzmann2018energy}. In \cite{gerard2018lax}, Gérard and Lenzmann prove that the HWM admits a Lax pair for both the $\S^2$ and $\HH^2$ target cases.

Proving well-posedness of the HWM in increasingly lower spatial dimensions is of interest because the most physically relevant case of the HWM is when $d=1$ (i.e., when the PDE models a Calogero-Moser spin system). Past work has used the reformulation of the HWM equation into a Wave Maps equation (WM) to prove existence in $d \ge 5$ for $M= \S^2$ \cite{krieger2017small}, with recent improvements to include $d=4$ with small $\Dot{B}^{2,1}_1$ data \cite{kiesenhofer2021small}. Existence of solutions for the HWM in $d \ge 4$ has also been shown with $M = \HH^2$ and small $\Dot{B}^{2,1}_1$ data \cite{liu2021global}.

Similar well-posedness results are difficult to achieve for $d=3$ as we lose access to the $L^2_tL^\infty_x$ Strichartz estimate. In a more recent paper by Marsen, progress on this case was made by proving that the equation is weakly globally well-posed for initial data that is not only small in $\Dot{B}_{2,1}^{3/2}$, but also possesses some angular regularity and weighted decay of derivatives \cite{marsden2024global}. Another result by Yang established the existence of weak global solutions of the HWM with target $\S^2$ for $d=1$ with large initial data in $\Dot{H}^1 \cap \Dot{H}^{1/2}(\R)$ \cite{yangR}. Uniqueness of solutions was shown for $d \ge 3$ in $\S^2$ using the reformulation of the HWM equation into a Wave Maps equation, orthogonality, and a Gr\"{o}nwall's inequality argument in \cite{eyeson2022uniqueness}. 

Using a similar strategy as in \cite{eyeson2022uniqueness}, the aim of this paper is to show uniqueness of solutions for the hyperbolic Half-Wave Maps equation (HHWM)
\begin{align}\label{eq:half-wave}
    \partial_t \vec{u} = \vec{u} \wedge_L |\nabla| \vec{u},
\end{align}
where $\vec{u}: \mathbb{R}^{d+1} \to \mathbb{H}^2 = \{\vec{u} \in \R^3: -\vec{u}_0^2 + \vec{u}_1^2 + \vec{u}_2^2 =-1, \vec{u}_0 > 0\}$, the Lorentzian cross-product $\wedge_L$ is given by $\vec{a}\wedge_L \vec{b} =\operatorname{diag}\brac{-1,1,1}\brac{\vec{a} \wedge \vec{b}}$, and $\Dso$ is defined by 
\begin{align*}
    \Dso^s f 
    &= \brac{-\Delta}^\frac{s}{2}f\\
    &= \mathcal{F}^{-1}\brac{|\xi|^s\mathcal{F}(f)},
\end{align*}
where $\mathcal{F}$ is the Fourier transform and $s \in (0,2)$ ($s=1$ in the case above).
The Lorentzian inner-product is given by $\langle \vec{a}, \vec{b} \rangle_L = \langle \operatorname{diag}\brac{-1,1,1}\vec{a}, \vec{b} \rangle_{\R^3}$. Therefore, we show the following theorem.

\begin{theorem} \label{thm:1}
   Let $d \ge 3$ and $\alpha \in (1,d+1/2)$, and $\mathbf{u}$,$\mathbf{v}: [0,T] \times \mathbf{R}^d \to \mathbb{H}^2 \hookrightarrow \mathbb{R}^3$ are smooth solutions to the half-wave map equation with the same initial data $\mathbf{u}(\cdot,0) = \mathbf{v}(0,\cdot) \in Q + C^\infty_c(\R^d,\R^3)$ for some $Q \in \mathbb{H}^2$ such that $|\mathbf{u}|$, $|\mathbf{v} | < \infty $. If
    \begin{align}
        \lVert \Dso^\alpha \mathbf{u}\rVert_{L^2_tL_x^{(\frac{2d}{2\alpha - 1},2)}(\mathbb{R}^d \times (0,T))} + \lVert \Dso^\alpha \mathbf{v}\rVert_{L^2_tL_x^{(\frac{2d}{2\alpha - 1},2)}(\mathbb{R}^d \times (0,T))} < \infty \label{first}, 
    \end{align}
    then $\mathbf{u}=\mathbf{v}$.
\end{theorem}

\textbf{Outline:} In \hyperref[sec2]{Section 2}, we introduce commutator estimates and Sobolev inequalities that will be useful in our estimates. In \hyperref[sec2]{Section 3}, we tackle the main issue in using a Gr\"{o}nwall's argument with the reformulation of the HWM. As we want to take advantage of the geometry of hyperbolic space (such as orthogonality and $|\vec{u}|_L^2 = -1$), we define the energy for the WM equation using the Lorentzian inner product. Since the Lorentzian inner product is not positive definite, we cannot assume the energy is non-negative. Thus, we use the Nash embedding to embed $\HH^2$ into some $\R^m$ and then define the conservation of energy using the Euclidean inner product. We conclude in \hyperref[sec2]{Section 4} by performing a Gr\"{o}nwall's argument to show uniqueness.

\subsection*{Acknowledgement} The research that has led to this result was funded by the National Science Foundation (NSF), Career DMS-2044898 (PI: Schikorra).

\section{Commutator Estimates and Nirenberg-Sobolev Inequalities} \label{sec2}
Most of the following lemmas and their proofs can be found in \cite{eyeson2022uniqueness} and are restated here for the reader's convenience.
\begin{lemma} \label{comp}
    For any $p \in (1,\infty)$,
    \begin{align*}
        \lVert \nabla f \rVert_{L^p(\R^d)} \approx \lVert \Dso f \rVert_{L^p(\R^d)}.
    \end{align*}
\end{lemma}
\begin{lemma}[Sobolev inequality] \label{lem:sob_in}
    Let $\alpha \in (0,d)$ and $p \in (1,\frac{1}{\alpha})$ then for any $ f \in C^\infty_c(\R^d)$,
    \begin{align*}
        \lVert f \rVert_{L^\frac{dp}{d-\alpha p}(\R^d)} 
        \lesssim \lVert \Dso^\alpha f \rVert_{L^p(\R^d)}.
    \end{align*}
\end{lemma}
\begin{lemma}[Gagliardo-Nirenberg-Sobolev inequality]\label{lem:gag_in}
    Assume
    \begin{enumerate}
        \item $\beta \in (0,\frac{1}{2}]$ and $p \in [\frac{2d}{\beta}, \infty)$, or
        \item $\beta \in (\frac{1}{2},1]$ and $p \in [\frac{2d}{\beta}, \frac{2d}{2\beta - 1}]$.
    \end{enumerate}
    Then for $\theta = 2\brac{\beta - \frac{d}{p}} \in [\beta, 1]$ we have 
    \begin{align*}
        \lVert \Dso^\beta f \rVert_{L^p(\R^d)} \lesssim \lVert f \rVert_{L^\infty(\R^d)}^{1-\theta} \lVert \Dso f \rVert_{L^{2d}(\R^d)}^\theta.
    \end{align*}
\end{lemma}
\begin{definition}\label{def_comm}
    The Leibniz rule operator for $\Dso^s$ will be denoted by 
    \begin{align*}
        H_{\Dso^s}\brac{f,g} \defeq \Dso^s(fg) - f\Dso^sg - \brac{\Dso^sf}g.
    \end{align*}
\end{definition}
\begin{lemma} \label{lem:comm_int}
    Let $s \in (0,2)$. Then for some $c = c(s,d)$,
    \begin{align*}
        H_{\Dso^s}\brac{f,g}(x) = c \int_{\R^d} \frac{\brac{f(x) - f(y)}\brac{g(x) - g(y)}}{\abs{x-y}^{d+s}} \ dy.
    \end{align*}
\end{lemma}
\begin{lemma}\label{comm_lem} For $\alpha \in (0,1]$
    \begin{align*}
    H_{\Dso^\alpha}\brac{f,gh} 
    =  H_{\Dso^\alpha}\brac{fg,h}
    +  H_{\Dso^\alpha}\brac{f,g} h
    - f  H_{\Dso^\alpha}\brac{g,h}.
    \end{align*}
\end{lemma}
\begin{proof}
    Using the definition of the commutator, we have
    \begin{align*}
        H_{\Dso^\alpha}\brac{f,gh}
        &= \Dso^\alpha\brac{fgh}
        - \Dso^\alpha\brac{f} g h
        - f \Dso^\alpha\brac{gh}\\
        &= \Dso^\alpha\brac{fgh}
        -\Dso^{\alpha}\brac{fg}h
        + f \Dso^\alpha g h\\
        & + H_{\Dso^\alpha}\brac{f,g} h
        - f \Dso^\alpha g h
        - f g \Dso^\alpha h
        - fH_{\Dso^\alpha}\brac{g,h}\\
        &=  H_{\Dso^\alpha}\brac{fg,h}
        +  H_{\Dso^\alpha}\brac{f,g} h
        - f  H_{\Dso^\alpha}\brac{g,h}.
    \end{align*}
\end{proof}
\begin{lemma} \label{lem:prod_rule}
    For $\alpha \in (0,1]$ and $\sigma \in (0,\alpha)$, we have that for any $p,p_1,p_2 \in (1,\infty)$ with $\frac{1}{p} = \frac{1}{p_1} + \frac{1}{p_2}$,
    \begin{align*}
        \lVert H_{\Dso^\alpha}\brac{f,g} \rVert_{L^{p}(\R^d)} 
        \lesssim \lVert \Dso^\sigma f \rVert_{L^{p_1}(\R^d)} \lVert \Dso^{\alpha - \sigma} g \rVert_{L^{p_2}(\R^d)}.
    \end{align*}
\end{lemma}
\begin{lemma}\label{lem:prod_rule_2}
     Let $\alpha \in (0,1)$ and $\beta \in (0,1]$. Pick any $\gamma \in (0,1)$ such that $\alpha + \beta - \gamma \in (0,1)$, and $p,p_1,p_2 \in (0,\infty)$ such that
     \begin{align*}
         \frac{1}{p} = \frac{1}{p_1} + \frac{1}{p_2}.
     \end{align*}
     Then 
     \begin{align*}
         \lVert\Dso^\alpha H_{\Dso^\beta}\brac{f,g} \rVert_{L^{p}(\R^d)}
         \lesssim\lVert\Dso^\gamma f\rVert_{L^{p_1}(\R^d)} \lVert\Dso^{\alpha+\beta-\gamma} g\rVert_{L^{p_2}(\R^d)}.
     \end{align*}
\end{lemma}
\begin{proof}
By duality, we only need to show that for any $h \in C^\infty_c(\R^n)$
\begin{align*}
    \left| \int_{\R^n} \Dso^\alpha H_{\Dso^\beta}\brac{f,g} h \right| \lesssim \lVert \Dso^{\gamma} f \rVert_{L^{p_1}(\R^d)} \ \lVert \Dso^{\alpha + \beta - \gamma} g \rVert_{L^{p_2}(\R^d)} \ \lVert h \rVert_{L^{p'}(\R^d)}
\end{align*}
where $\frac{1}{p} + \frac{1}{p'} = 1$.

Let $\Tilde{h} = \Dso^\alpha h \in L^1$. We use the same argument as in the proof of Theorem 3.4.1 in \cite{frac_estimates}.

By two applications of fractional integration by parts, we have
\begin{align*}
   \mathcal{I} = \left| \int_{\R^d} f g \Dso^\beta \Tilde{h} - \Dso^\beta f g \Tilde{h} - f \Dso^\beta g \Tilde{h}\right|.
\end{align*}
We let $F(x,s) = P^\beta_s f(x)$, $G(x,s) = P^\beta_s g(x)$, and $h(x,s) = P^\beta_s \Tilde{h}(x)$ be the $\beta$-harmonic extensions of $f,g$, and $\Tilde{h}$ (\cite{frac_estimates}, Definition 3.1). For example, the harmonic extension of $f$ is 
\begin{align*}
    P^\beta_sf(x) = c_{d,\beta} \int_{\R^d} \frac{s^\beta}{\brac{|x-y|^2 + t^2}^{\frac{d + \beta}{2}}} f(y) \ dy
\end{align*}
and satisfies
\begin{align}
\begin{cases}
    &\operatorname{div}_{\R^{d+1}}(s^{1-\beta} \nabla_{\R^{d+1}} F(x,s)) = 0\\
    &\lim_{s\to 0} F(x,s) = f(x)\\
    &\lim_{|(x,s)|\to \infty} F(x,s) = 0 \label{harmonic}
\end{cases}
\end{align}
and 
\begin{align*}
    \lim_{s \to 0} -s^{1-\beta} \partial_s F(x,s) = c\Dso^\beta f(x).
\end{align*}

By integration by parts in $s$ and \eqref{harmonic}
\begin{align*}
    &\int_{\R^{d+1}_+} \partial_s\Big(s^{1-\beta} F G \partial_s H - s^{1-\beta}\partial_s F G H - s^{1-\beta}F \partial_s G H\Big)\\
    =&\int_{\R^d} \left[s^{1-\beta} \Big(F G \partial_s H - \partial_s F G H - F \partial_s G H\Big)\right]^\infty_0
    -\int_{\R^{d+1}_+} s^{1-\beta} \Big(F G \partial_s H - \partial_s F G H - F \partial_s G H \Big) \partial_s (1) \\
    =&\lim_{s\to \infty}\int_{\R^d} \left[s^{1-\beta} \Big(F G \partial_s H - \partial_s F G H - F \partial_s G H\Big)\right]\\
    -& \lim_{s\to 0}\int_{\R^d} \left[s^{1-\beta} \Big(F G \partial_s H - \partial_s F G H - F \partial_s G H\Big)\right]\\
    =&\lim_{s\to \infty}\int_{\R^d} \left[s^{1-\beta} \Big(F G \partial_s H - \partial_s F G H - F \partial_s G H\Big)\right]
    - \int_{\R^d} f g \Dso^\beta \Tilde{h} - \Dso^\beta f g \Tilde{h} - f \Dso^\beta g \Tilde{h}.
\end{align*}
Now by the decay estimate in \cite{frac_estimates} (Lemma 3.1.3), we have that
\begin{align*}
    &\lim_{s\to \infty}\int_{\R^d} \left[s^{1-\beta} \Big(F G \partial_s H - \partial_s F G H - F \partial_s G H\Big)\right]\\
    &\lesssim \lim_{s\to \infty} s^{-\beta}\lVert f \rVert_{L^\infty(\R^d)} \ \lVert g \rVert_{L^\infty(\R^d)}\lVert h \rVert_{L^\infty(\R^d)}\\
    &=0.
\end{align*}
Therefore,
\begin{align*}
    \mathcal{I}
    &= \left| \int_{\R^d} f g \Dso^\beta \Tilde{h} - \Dso^\beta f g \Tilde{h} - f \Dso^\beta g \Tilde{h} \ \right|\\
    &= \left| \int_{\R^{d+1}_+} \partial_s\Big(s^{1-\beta} F G \partial_s H - s^{1-\beta}\partial_s F G H - s^{1-\beta}F \partial_s G H\Big) \right|\\
    & =\vcentcolon \left|\int_{\R^{d+1}_+} \mathcal{J} \ \right|.
\end{align*}
With the product rule, we have 
\begin{align*}
    \mathcal{J}
    &= \partial_s F G s^{1-\beta} \partial_s H +  F \partial_s G s^{1-\beta} \partial_s H + F G \partial_s\brac{s^{1-\beta} \partial_s H} \\
    &- \partial_s\brac{s^{1-\beta}\partial_s } G H - s^{1-\beta}\partial_s F \partial_sG H - s^{1-\beta}\partial_s F G \partial_s H \\
    &- \partial_s F s^{1-\beta}\partial_s G H - F \partial_s\brac{s^{1-\beta}\partial_s G} H - F s^{1-\beta}\partial_s G \partial_s H\\
    &= F G \partial_s\brac{s^{1-\beta} \partial_s H} - \partial_s\brac{s^{1-\beta}\partial_s F} G H - F \partial_s\brac{s^{1-\beta}\partial_s G} H - 2s^{1-\beta}\partial_s F \partial_sG H.
\end{align*}
Using \eqref{harmonic}, we have, for example, that $\partial_s\brac{s^{1-\beta} \partial_s F} = - s^{1-\beta}\Delta_x F$. Therefore,
\begin{align*}
    \mathcal{J} 
    &= - F G s^{1-\beta}\Delta_x H + s^{1-\beta}\Delta_x F G H + F s^{1-\beta}\Delta_x G H - 2s^{1-\beta}\partial_s F \partial_sG H \\
    &= s^{1-\beta} \brac{\Delta_x\brac{FG}H - FG\Delta_x H} - 2s^{1-\beta}\nabla_x F \cdot \nabla_x G H - 2s^{1-\beta}\partial_s F \partial_sG H.
\end{align*}
By integration by parts, we have
\begin{align*}
    \int_{\R^{d+1}_+} s^{1-\beta} \brac{\Delta_x\brac{FG}H - FG\Delta_x H} 
    = \int_{\R^{d+1}_+} s^{1-\beta} \brac{\nabla_x\brac{FG} \cdot \nabla_x H - \nabla_x\brac{FG} \nabla_x H} = 0.
\end{align*}
Thus, by \cite{frac_estimates}(Lemma 5.7.1, Lemma 5.7.1(a), and Lemma 5.7.1(b)) 
\begin{align*}
    \mathcal{I} 
    &\lesssim \left|\int_{\R^{d+1}} s^{1-\beta}\nabla_x F \cdot \nabla_x G H - s^{1-\beta}\partial_s F \partial_sG H \right|\\
    &= \left|\int_{\R^{d+1}} s^{1-\gamma -\alpha -\beta - \alpha}\nabla_x F \cdot \nabla_x G H - s^{1-\gamma -\alpha -\beta - \alpha}\partial_s F \partial_sG H \right|\\
    &\lesssim \lVert \Dso^{\gamma} f \rVert_{L^{p_1}(\R^d)} \ \lVert \Dso^{\alpha + \beta - \gamma} g \rVert_{L^{p_2}(\R^d)} \ \lVert I^{-\alpha}\Dso^\alpha h \rVert_{L^{p'}(\R^d)}\\
    &= \lVert \Dso^{\gamma} f \rVert_{L^{p_1}(\R^d)} \ \lVert \Dso^{\alpha + \beta - \gamma} g \rVert_{L^{p_2}(\R^d)} \ \lVert h \rVert_{L^{p'}(\R^d)}.
\end{align*}

\end{proof}
\begin{lemma} \label{lem:prod_rule_frac}
    Assume $d \ge 3$, $\sigma \in [0,1]$. Then 
    \begin{align*}
        \lVert \Dso^\sigma\brac{f\Dso g}\rVert_{L^{\frac{2d}{d+2\sigma-1}}(\R^d)}
        \lesssim \lVert \Dso f\rVert_{L^2(\R^d)}\lVert \Dso^{1+\sigma}g\rVert_{L^{\frac{2d}{2(1+\sigma)-1}}(\R^d)}.
    \end{align*}
\end{lemma}
\begin{lemma} \label{lem:comm_difference}
    For $\alpha \in (\frac{1}{2},1)$, $d \ge 2$
    \begin{align*}
        \lVert H_{\Dso}\brac{fh,g} - fH_{\Dso}\brac{h,g}\rVert_{L^2(\R^d)} 
        \lesssim \lVert \Dso^\alpha f\rVert_{L^{\frac{2d}{2\alpha - 1}}(\R^d)}\lVert \Dso^\alpha g\rVert_{L^{\frac{2d}{2\alpha - 1}}(\R^d)}\lVert h \rVert_{L^2(\R^d)}.
    \end{align*}
\end{lemma}
\begin{lemma} \label{infty_bound}
    Let $d\ge 2$. Then 
    \begin{align*}
        \lVert H_{\Dso}\brac{f,g} \rVert_{L^\infty(\R^d)} \lesssim \lVert \Dso f \rVert_{L^{2d,2}(\R^d)} \ \lVert \Dso g \rVert_{L^{2d,2}(\R^d)}.
    \end{align*}
\end{lemma}
\begin{lemma}\label{double_comm_int}
    Assume $s \in (0,1]$ and $\alpha_1,\alpha_2 \in (0,s)$ such that $\alpha_1 + \alpha_2 = s$. Let $p \in (1,\infty)$. Assume for $p_1,p_2 \in [2,\infty)$ such that 
    \begin{align*}
        \frac{1}{p_1} + \frac{1}{p_2} = \frac{1}{p}.
    \end{align*}
    Then we have 
    \begin{align*}
        \brac{\int_{\R^d}\brac{\int_{\R^d} \frac{\abs{f(x) - f(y)} \abs{g(x)-g(y)}}{|x-y|^{d+s}} \ dy}^p }^{\frac{1}{p}}
        \lesssim \lVert \Dso^{\alpha_1} f \rVert_{L^{p_1}(\R^d)} \ \lVert \Dso^{\alpha_2} g \rVert_{L^{p_2}(\R^d)}.
    \end{align*}
\end{lemma}
\begin{lemma}\label{triple_comm_int_3}
    Assume $\alpha_i \in (0,1)$ such that $\sum_{i=1}^3\alpha_i = 1$ and let $p \in (1,\infty)$. Assume $p \in (p,\infty)$ such that
    \begin{align*}
        \frac{1}{p_1} + \frac{1}{p_2} + \frac{1}{p_3}= \frac{1}{p}
    \end{align*}
    and 
    \begin{align*}
        \frac{1}{p_i} - \frac{\alpha_i}{d} < \frac{1}{2}, \ i = 1,2,3. 
    \end{align*}
    Then we have 
    \begin{align*}
        &\brac{\int_{\R^d}\brac{\int_{\R^d} \frac{\abs{f(x) - f(y)} \abs{g(x)-g(y)} \abs{h(x) - h(y)}}{|x-y|^{d+1}} \ dy}^p }^{\frac{1}{p}}\\
        \lesssim & \lVert \Dso^{\alpha_1} f \rVert_{L^{p_1}(\R^d)} \ \lVert \Dso^{\alpha_2} g \rVert_{L^{p_2}(\R^d)} \ \lVert \Dso^{\alpha_3} h \rVert_{L^{p_3}(\R^d)}.
    \end{align*}
\end{lemma}
\begin{lemma}\label{triple_comm_int}
    Assume $\alpha_i \in (0,1)$, $i=1,2$ such that $\alpha_1 + \alpha_2 > 1$. Assume $p_i \in (2,\infty)$ such that
    \begin{align*}
        \frac{1}{p_1} + \frac{1}{p_2} + \frac{1}{p_3}= \frac{1}{p}
    \end{align*}
    and 
    \begin{align*}
        \frac{dp_3}{d-(1-\alpha_1+\alpha_2)p_3} \in (1,\infty).
    \end{align*}
    Then we have
    \begin{align*}
        &\brac{\int_{\R^d}\brac{\int_{\R^d} \frac{\abs{f(x) - f(y)} \abs{g(x)-g(y)} \abs{h(y)}}{|x-y|^{d+1}} \ dy}^2 }^{\frac{1}{2}}\\
        \lesssim & \lVert \Dso^{\alpha_1} f \rVert_{L^{p_1}(\R^d)} \ \lVert \Dso^{\alpha_2} g \rVert_{L^{p_2}(\R^d)} \ \lVert h \rVert_{L^{\frac{dp_3}{d-(1-\alpha_1 - \alpha_2)p_3}}(\R^d)}.
    \end{align*}
\end{lemma}
\section{Decay Estimates}
As originally shown in \cite{krieger2017small} for HWM and then in \cite{liu2021global} for the HHWM (\cref{reform}), by differentiating the HHWM in time and then subtracting the Laplacian of $\vec{u}$, we get
\begin{equation}
\begin{aligned}\label{eq:wave map}
    \vec{u}_{tt} - \Delta \vec{u} = &(|\partial_t \vec{u}|_L^2 - |\nabla \vec{u}|_L^2) \vec{u} \\
    &+ \vec{u} \wedge_L|\nabla|(\vec{u} \wedge_L|\nabla|\vec{u}) - \vec{u} \wedge_L(\vec{u} \wedge_L(-\Delta \vec{u})) \\
    &- P_\vec{u}\brac{\langle\vec{u} , |\nabla| \vec{u}\rangle_L|\nabla|\vec{u}},
\end{aligned}
\end{equation}
where $P_{\vec{u}}: \mathbb{R}^3 \to T_{\vec{u}} \mathbb{H}^2 \subset \mathbb{R}^3$ is the projection onto the orthogonal complement of $\mathbf{u}$. So
$P_{\vec{u}} \vec{x} =  \vec{x} + \langle \vec{u},\vec{x} \rangle_L\vec{u}$ for all $\mathbf{x} \in \R^3$.
Recall that in \cite{eyeson2022uniqueness}, uniqueness of solutions was shown in $d\ge 3$ for $\S^2$. This method relied on the integral representing the energy conservation of the Wave Maps equation
\begin{align*}
    E(\mathbf{u} -\mathbf{v},t) = \frac{1}{2}\int_{\R^d} |\partial_t\brac{\mathbf{u} -\mathbf{v}}|^2 + |\nabla \brac{\mathbf{u} -\mathbf{v}}|^2 \ge 0
\end{align*}
being non-negative, where $\vec{u}$ and \vec{v} are two solutions to the HWM. 

The main issue when attempting to use the strategy used in \cite{eyeson2022uniqueness} for the HHWM is that the Lorentzian inner-product is not necessarily positive-definite on differences of tangent spaces, so 
\begin{align*}
    E(\mathbf{u} -\mathbf{v},t) = \frac{1}{2}\int_{\R^d} |\partial_t\brac{\mathbf{u} -\mathbf{v}}|^2_L + |\nabla \brac{\mathbf{u} -\mathbf{v}}|^2_L \overset{?}{\ge} 0.
\end{align*}
To circumvent this, we use the Nash embedding theorem \cite[Theorem 3]{nash} on $\mathbb{H}^2$ and isometrically embed it into some $\mathbb{R}^m$ in order to rewrite \eqref{eq:wave map} with Euclidean range. Therefore, there is a smooth embedding $\varphi: \mathbb{H}^2 \xhookrightarrow{} \mathbb{R}^m$ such that $\langle d\varphi_p \mathbf{x}, d \varphi_p \mathbf{y} \rangle_{\R^m}= \langle \mathbf{x}, \mathbf{y} \rangle_L$ for $\mathbf{x},\mathbf{y} \in T_p \mathbb{H}^2$, where $\langle \cdot, \cdot \rangle_{\R^m}$ is the Euclidean inner-product on $\mathbb{R}^m$.  Take an orthonormal frame $\mathbf{E}_p^1,\mathbf{E}_p^2$ for $T_{\Tilde{p}}\Tilde{\mathbb{H}}^2 = d\varphi_p(T_p\mathbb{H}^2) \subset T_{\varphi(p)} \mathbb{R}^m$ such that $d\varphi_p \mathbf{X}_p^r = \mathbf{E}^r_{\Tilde{p}}$ for $r \in\{1,2\}$, where $\mathbf{X}^1_p,\mathbf{X}^2_p$ is an orthonormal frame for $T_p\mathbb{H}^2$ (\cref{comp_ortho}), and extend this to an orthonormal frame $\mathbf{E}^1_{\Tilde{p}},\mathbf{E}^2_{\Tilde{p}}, \dots, \mathbf{E}^m_{\Tilde{p}}$ for $T_{\varphi(p)} \mathbb{R}^m$.

Now, letting $\Tilde{\vec{u}} = \varphi \circ \vec{u}$, we have
\begin{align*}
    \partial^\alpha \partial_\alpha \Tilde{\vec{u}} = \sum_{r=1}^2\lambda_r\mathbf{E}^r_{\Tilde{\vec{u}}} + \sum_{l=3}^m\lambda_l\mathbf{E}^l_{\Tilde{\vec{u}}},
\end{align*}
where $\partial^\alpha = \partial_\alpha$ if $\alpha = t$ and $\partial^\alpha = -\partial_\alpha$ otherwise.
The covariant derivative in the $\alpha$ direction will be denoted by $D^\alpha$ in $\HH^2 \subset \R^3$ and $\Tilde{D}^\alpha$ in $\Tilde{\HH}^2\subset \R^m$. If we fix all other directions except for the $\alpha$ direction, we have that $\vec{u}$ is a path from $\R$ to $\mathbb{H}^2$ and $\partial_\alpha \vec{u}$ is a vector field along $\vec{u}$, so
\begin{align*}
\tilde{D}^\alpha \partial_\alpha \Tilde{\vec{u}} = \langle \partial^\beta \partial_\alpha \tilde{\vec{u}}, \vec{E}_{\tilde{\vec{u}}}^1\rangle_{} \vec{E}_{\tilde{\vec{u}}}^1 +  \langle \partial^\beta \partial_\alpha {\tilde{\vec{u}}}, \vec{E}_{\Tilde{\vec{u}}}^2\rangle_{} \vec{E}_{\tilde{\vec{u}}}^2
\end{align*}
and
\begin{align*}
D^\alpha \partial_\alpha \vec{u} = \langle \partial^\beta \partial_\alpha \vec{u}, \vec{X}_{\vec{u}}^1\rangle_{L} \vec{X}_{\vec{u}}^1 +  \langle \partial^\beta \partial_\alpha \vec{u}, \vec{X}_{\vec{u}}^2\rangle_{L} \vec{X}_{\vec{u}}^2
\end{align*}
by \cref{cov_der_tang}. Also, 
\begin{align*}
    \Tilde{D}^\alpha \partial_\alpha \Tilde{\vec{u}}
    &= \Tilde{D}^\alpha \partial_\alpha\brac{\varphi \circ \vec{u}}\\
    &= \Tilde{D}^\alpha d\varphi_{\vec{u}}\partial_\alpha \vec{u}\\
    &= d\varphi_{\vec{u}}D^\alpha \partial_\alpha \Tilde{\vec{u}}
\end{align*}
by \cref{no_sec_der}. 
Then 
\begin{align*}
    \lambda_r &= \langle \partial^\alpha \partial_\alpha \Tilde{\vec{u}}, \mathbf{E}^r_{\Tilde{\vec{u}}}\rangle_{\R^m}\\
    &=\langle \Tilde{D}^\alpha \partial_\alpha \Tilde{\vec{u}},\mathbf{E}^r_{\Tilde{\vec{u}}}\rangle_{\R^m}\\
    &=\langle d\varphi_{\vec{u}}D^\alpha \partial_\alpha \Tilde{\vec{u}},d\varphi_{\vec{u}}\mathbf{X}^r_\vec{u}\rangle_{\R^m}\\
    &= \langle D^\alpha \partial_\alpha \vec{u},\mathbf{X}^r_\vec{u}\rangle_L\\
    &= \langle \partial^\alpha \partial_\alpha \vec{u},\mathbf{X}^r_\vec{u}\rangle_L,
\end{align*} 
and using the product rule and the fact that $\mathbf{E}^l_{\Tilde{\vec{u}}} \in \brac{T_{\Tilde{\vec{u}}} \Tilde{\mathbb{H}}^2}^\perp$,
\begin{align*}
    \lambda_l & = \langle \partial^\alpha\partial_\alpha \Tilde{\vec{u}}, \mathbf{E}^l_{\Tilde{\vec{u}}}\rangle_{\R^m}\\
    &=\partial^\alpha \langle \partial_\alpha \Tilde{\vec{u}}, \mathbf{E}^l_{\Tilde{\vec{u}}} \rangle_{\R^m}- \langle \partial_\alpha \Tilde{\vec{u}}, d\mathbf{E}^l_{\Tilde{\vec{u}}} \partial^\alpha \Tilde{\vec{u}}\rangle_{\R^m}\\
    &=- \langle \partial_\alpha \Tilde{\vec{u}}, d\mathbf{E}^l_{\Tilde{\vec{u}}} \partial^\alpha \Tilde{\vec{u}}\rangle_{\R^m}\\
    &= A^l(\Tilde{\vec{u}})(\partial_\alpha \Tilde{\vec{u}}, \partial^\alpha\Tilde{\vec{u}}),
\end{align*}
where $A(\Tilde{p})(\cdot,\cdot): T_{\Tilde{p}}\Tilde{\mathbb{H}}^2 \times T_{\Tilde{p}}\Tilde{\mathbb{H}}^2 \to \brac{T_{\Tilde{p}}\Tilde{\mathbb{H}}^2}^\perp$ is the second fundamental form and 
\begin{align*}
    A(\Tilde{p})(V,W) 
    &=  \sum_{l=3}^m A^l(\Tilde{\vec{u}})(V,W) \vec{E}^l_{\Tilde{p}}\\
    &= \sum_{l=3}^m - \langle V, d\mathbf{E}^l_{\Tilde{p}} W\rangle_{\R^m} \vec{E}^l_{\Tilde{p}}.
\end{align*}
Now using $\langle\vec{u} , \mathbf{X}^r_\vec{u}\rangle_L = 0$ given by \cref{tang_def}, we see that
\begin{equation}
    \begin{aligned}
        \sum_{\alpha}\langle \partial^\alpha\partial_\alpha \Tilde{\vec{u}},\mathbf{E}^r_{\Tilde{\vec{u}}}\rangle_L \mathbf{E}_{\Tilde{\vec{u}}}^r
        &=\sum_{\alpha}\langle \partial^\alpha\partial_\alpha \vec{u},\mathbf{X}^r_\vec{u}\rangle_L \mathbf{E}_{\Tilde{\vec{u}}}^r\\
        &=\langle \vec{u}_{tt} - \Delta \vec{u},\mathbf{X}^r_\vec{u}\rangle_L \mathbf{E}_{\Tilde{\vec{u}}}^r\\
        &= \Big\langle (|\partial_t \vec{u}|^2 - |\nabla \vec{u}|^2) \vec{u} , \mathbf{X}^r_\vec{u}\Big\rangle_L \mathbf{E}_{\Tilde{\vec{u}}}^r\\
        &+\Big\langle\vec{u} \wedge_L|\nabla|(\vec{u} \wedge_L|\nabla|\vec{u}) - \vec{u} \wedge_L(\vec{u} \wedge_L(-\Delta \vec{u})) - \langle \vec{u} , |\nabla| \vec{u}\rangle_L P_\vec{u}(|\nabla|\vec{u}), \mathbf{X}_\vec{u}^r\Big\rangle_L \mathbf{E}^r_{\Tilde{\vec{u}}}\\
        &=\Big\langle\vec{u} \wedge_L|\nabla|(\vec{u} \wedge_L|\nabla|\vec{u}) - \vec{u} \wedge_L(\vec{u} \wedge_L(-\Delta \vec{u})) - \langle \vec{u} , |\nabla| \vec{u}\rangle_L P_\vec{u}(|\nabla|\vec{u}), \mathbf{X}_\vec{u}^r\Big\rangle_L \mathbf{E}^r_{\Tilde{\vec{u}}}\\
        &=\Big\langle\vec{u} \wedge_LH_{|\nabla|}(\vec{u} \wedge_L,|\nabla|\vec{u}) - \langle \vec{u} , |\nabla| \vec{u}\rangle_L P_\vec{u}(|\nabla|\vec{u}), \mathbf{X}_\vec{u}^r\Big\rangle_L \mathbf{E}^r_{\Tilde{\vec{u}}}
    \end{aligned}
\end{equation}
for $r \in \{1,2\}$. Thus,
\begin{align*}
    \Tilde{\vec{u}}_{tt} - \Delta \Tilde{\vec{u}}
    &= \sum_{l=1}^m A^l(\Tilde{\vec{u}})\brac{\partial_{\alpha} \Tilde{\vec{u}}, \partial^{\alpha}\Tilde{\vec{u}} }\mathbf{E}_{\Tilde{\vec{u}}}^l \\
    &+\sum_{r=1}^2 \Big\langle\vec{u} \wedge_LH_{|\nabla|}(\vec{u} \wedge_L,|\nabla|\vec{u}) - \langle \vec{u} , |\nabla| \vec{u}\rangle_L P_\vec{u}(|\nabla|\vec{u}), \mathbf{X}_\vec{u}^r\Big\rangle_L \mathbf{E}^r_{\Tilde{\vec{u}}}.
\end{align*}
Suppose $\vec{u},\vec{v}$ are solutions to \eqref{eq:half-wave} such that $\vec{u}(\cdot,0) = \vec{v}(\cdot,0)$ and let $\Tilde{\vec{w}} = \Tilde{\vec{u}} - \Tilde{\vec{v}}$. Since the Euclidean inner product is positive definite, we conclude that the energy is non-negative. Therefore, 
\begin{equation}
    \begin{aligned}
        E(t) = \frac{1}{2}\brac{\lVert\nabla \Tilde{\vec{w}}(t)\rVert^2_{L^2(\mathbb{R}^d)} + \lVert \partial_t\Tilde{\vec{w}}(t) \rVert^2_{L^2(\mathbb{R}^d)}} \ge 0.
    \end{aligned}
\end{equation}
Differentiating in time yields,
\begin{align}
        \frac{d}{dt}E(t) =& \int_{\mathbb{R}^d} \langle \Tilde{\vec{w}}_{tt} - \Delta \Tilde{\vec{w}}, \Tilde{\vec{w}}_{t} \rangle_{\R^m} \nonumber
        \\
        =&\int_{\mathbb{R}^d} \langle A^l(\Tilde{\vec{u}})(\partial_{\alpha} \Tilde{\vec{u}}, \partial^{\alpha}\Tilde{\vec{u}})\mathbf{E}_{\Tilde{\vec{u}}}^l, \Tilde{\vec{w}}_{t} \rangle_{\R^m}- \langle A^l(\Tilde{\vec{v}})(\partial^{\alpha} \Tilde{\vec{v}}, \partial_{\alpha}\Tilde{\vec{v}})\mathbf{E}_{\Tilde{\vec{v}}}^l, \Tilde{\vec{w}}_{t} \rangle_{\R^m} \label{eq:diffenergy_1}\\
        -&\int_{\mathbb{R}^d} \langle\vec{u} , |\nabla| \vec{u}\rangle_L \langle P_\vec{u}( |\nabla|\vec{u}), \mathbf{X}_\vec{u}^r \rangle_L \langle \mathbf{E}_{\Tilde{\vec{u}}}^r, \Tilde{\vec{w}}_{t} \rangle_{\R^m}-  \langle\vec{v} , |\nabla| \vec{v}\rangle_L\langle P_\vec{v}(|\nabla|\vec{v}), \mathbf{X}_\vec{v}^r\rangle_L \langle \mathbf{E}_{\Tilde{\vec{v}}}^r, \Tilde{\vec{w}}_{t} \rangle_{\R^m} \label{eq:diffenergy_2}\\
        +&\int_{\mathbb{R}^d} \langle \vec{u} \wedge_LH_{|\nabla|}(\vec{u} \wedge_L,|\nabla|\vec{u}), \mathbf{X}_\vec{u}^r\rangle_L \langle \mathbf{E}_{\Tilde{\vec{u}}}^r, \Tilde{\vec{w}}_{t} \rangle_{\R^m}- \langle \vec{v} \wedge_LH_{|\nabla|}(\vec{v} \wedge_L,|\nabla|\vec{v}), \mathbf{X}_\vec{v}^r\rangle_L \langle \mathbf{E}_{\Tilde{\vec{v}}}^r, \Tilde{\vec{w}}_{t} \rangle_{\R^m} \label{eq:diffenergy_3}.
\end{align}
In the remainder of this section we prove the main estimate for \cref{thm:1}.
\begin{theorem} \label{big_thm}
    Let $\vec{u},\vec{v}$ be as in \cref{thm:1} such that $\lVert \vec{u} \rVert_\infty, \lVert \vec{v} \rVert_\infty < \Lambda$ for some $0< \Lambda < \infty$. Set 
    $$
    E(t) = \frac{1}{2}\brac{\lVert\nabla \Tilde{\vec{w}}(t)\rVert^2_{L^2(\mathbb{R}^d)} + \lVert \partial_t\Tilde{\vec{w}}(t) \rVert^2_{L^2(\mathbb{R}^d)}}.
    $$
    Then for any $\alpha > 1$,
    $$
    \frac{d}{dt}E(t) \lesssim_\Lambda \Sigma(t)E(t),
    $$
    where
    \begin{align*}
        \Sigma(t) &\lesssim_\Lambda \lVert |\nabla|^\alpha \vec{u}(t)\rVert_{L^{\frac{2d}{2\alpha - 1}}(\mathbb{R}^d)}^2 + \lVert |\nabla|^\alpha \vec{v}(t)\rVert_{L^{\frac{2d}{2\alpha - 1}}(\mathbb{R}^d)}^2\\
        &+ \lVert|\nabla| \vec{u}(t)\rVert_{L^{(2d,2)}(\mathbb{R}^d)}^2 + \lVert |\nabla|\vec{v}(t)\rVert_{L^{(2d,2)}(\mathbb{R}^d)}^2.
    \end{align*}
\end{theorem}
With the assumptions of \cref{thm:1}, we are left to prove \cref{big_thm}.
\subsection{Proof of \cref{big_thm}}
We will use the following throughout without further mention. Observe that for $\frac{1}{2} < \alpha_1 < \alpha_2 < d + \frac{1}{2}$, using \cref{lem:sob_in}, we have
\begin{align} \label{chain}
    \lVert \Dso^{\alpha_1} \vec{u} \rVert_{L^{\frac{2d}{2\alpha_1 - 1}}(\R^d)}
    \lesssim  \lVert \Dso^{\alpha_2} \vec{u} \rVert_{L^{\frac{2d}{2\alpha_2 - 1}}(\R^d)}.
\end{align}
We begin by estimating \eqref{eq:diffenergy_1} as in \cite{sha_stru}.
\begin{lemma}
For $d \ge 3$, we have that 
\begin{align*}
    &\int_{\mathbb{R}^d} \langle A(\Tilde{\vec{u}})(\partial_{\alpha} \Tilde{\vec{u}}, \partial^{\alpha}\Tilde{\vec{u}}), \Tilde{\vec{w}}_{t} \rangle_{\R^m}- \langle A(\Tilde{\vec{v}})(\partial^{\alpha} \Tilde{\vec{v}}, \partial_{\alpha}\Tilde{\vec{v}}), \Tilde{\vec{w}}_{t} \rangle_{\R^m}\ \\
    & \lesssim 
    \brac{\lVert d \vec{u} \rVert_{L^{2n}(\R^d)}^2 + \lVert d \vec{v} \rVert_{L^{2n}(\R^d)}^2} \lVert d \Tilde{\vec{w}} \rVert_{L^2(\R^d)}^2.
\end{align*}
\end{lemma}
\begin{proof}
Since the second fundamental form is bi-linear and symmetric, we have
\begin{align*}
    &A(\Tilde{\vec{u}})(\partial_{\alpha} \Tilde{\vec{u}}, \partial^{\alpha}\Tilde{\vec{u}}) - A(\Tilde{\vec{v}})(\partial_{\alpha} \Tilde{\vec{v}}, \partial^{\alpha}\Tilde{\vec{v}})\\
    &=  \sum_{l=3}^m\langle \partial_\alpha \Tilde{\vec{u}}, d\mathbf{E}^l_{\Tilde{\vec{u}}} \partial^\alpha \Tilde{\vec{u}}\rangle_{\R^m}\vec{E}^l_{\Tilde{\vec{u}}} -  \langle \partial_\alpha \Tilde{\vec{v}}, d\mathbf{E}^l_{\Tilde{\vec{v}}} \partial^\alpha \Tilde{\vec{v}}\rangle_{\R^m}\vec{E}^l_{\Tilde{\vec{v}}}\\
    &=  \sum_{l=3}^m\langle \partial_\alpha \Tilde{\vec{u}}, d\mathbf{E}^l_{\Tilde{\vec{u}}} \partial^\alpha \Tilde{\vec{u}}\rangle_{\R^m}\brac{\vec{E}^l_{\Tilde{\vec{u}}} - \vec{E}^l_{\Tilde{\vec{v}}}} + \langle \partial_\alpha \Tilde{\vec{u}}, d\mathbf{E}^l_{\Tilde{\vec{u}}} \partial^\alpha \Tilde{\vec{u}}\rangle_{\R^m}\vec{E}^l_{\Tilde{\vec{v}}}-  \langle \partial_\alpha \Tilde{\vec{v}}, d\mathbf{E}^l_{\Tilde{\vec{v}}} \partial^\alpha \Tilde{\vec{v}}\rangle_{\R^m}\vec{E}^l_{\Tilde{\vec{v}}}\\
    &= \brac{A(\Tilde{\vec{u}}) - A(\Tilde{\vec{v}})}(\partial_{\alpha} \Tilde{\vec{u}}, \partial^{\alpha}\Tilde{\vec{u}}) + A(\Tilde{\vec{v}})(\partial_{\alpha} \Tilde{\vec{u}}, \partial^{\alpha}\Tilde{\vec{u}}) - A(\Tilde{\vec{v}})(\partial_{\alpha} \Tilde{\vec{v}}, \partial^{\alpha}\Tilde{\vec{v}})\\
    &= \brac{A(\Tilde{\vec{u}}) - A(\Tilde{\vec{v}})}(\partial_{\alpha} \Tilde{\vec{u}}, \partial^{\alpha}\Tilde{\vec{u}}) + A(\Tilde{\vec{v}})(\partial_{\alpha} \Tilde{\vec{u}}, \partial^{\alpha}\Tilde{\vec{u}} - \partial^{\alpha}\Tilde{\vec{v}}) + A(\Tilde{\vec{v}})(\partial_{\alpha} \Tilde{\vec{u}} - \partial_{\alpha} \Tilde{\vec{v}}, \partial^{\alpha}\Tilde{\vec{v}})\\
    &= \brac{A(\Tilde{\vec{u}}) - A(\Tilde{\vec{v}})}(\partial_{\alpha} \Tilde{\vec{u}}, \partial^{\alpha}\Tilde{\vec{u}}) + A(\Tilde{\vec{v}})(\partial_{\alpha} \Tilde{\vec{u}}, \partial^{\alpha}\Tilde{\vec{u}} - \partial^{\alpha}\Tilde{\vec{v}}) + A(\Tilde{\vec{v}})(\partial_{\alpha} \Tilde{\vec{v}}, \partial^{\alpha}\Tilde{\vec{u}} - \partial^{\alpha}\Tilde{\vec{v}})\\
    &= \brac{A(\Tilde{\vec{u}}) - A(\Tilde{\vec{v}})}(\partial_{\alpha} \Tilde{\vec{u}}, \partial^{\alpha}\Tilde{\vec{u}}) + A(\Tilde{\vec{v}})(\partial_{\alpha} \Tilde{\vec{u}} + \partial_{\alpha} \Tilde{\vec{v}}, \partial^{\alpha}\Tilde{\vec{w}}).
\end{align*}
Hence,
    \begin{align*}
&\int_{\mathbb{R}^d} \langle A(\Tilde{\vec{u}})(\partial_{\alpha} \Tilde{\vec{u}}, \partial^{\alpha}\Tilde{\vec{u}}), \Tilde{\vec{w}}_{t} \rangle_{\R^m}- \langle A(\Tilde{\vec{v}})(\partial_{\alpha} \Tilde{\vec{v}}, \partial^{\alpha}\Tilde{\vec{v}}), \Tilde{\vec{w}}_{t} \rangle_{\R^m}\\
&=\int_{\mathbb{R}^d} \langle [A(\Tilde{\vec{u}}) -A(\Tilde{\vec{v}})](\partial_{\alpha} \Tilde{\vec{u}}, \partial^{\alpha}\Tilde{\vec{u}}), \Tilde{\vec{w}}_{t} \rangle_{\R^m}+ \langle A(\Tilde{\vec{v}})(\partial_{\alpha} \Tilde{\vec{u}} + \partial_{\alpha} \Tilde{\vec{v}}, \partial^{\alpha}\Tilde{\vec{w}}), \Tilde{\vec{w}}_{t} \rangle_{\R^m}\\
&= I(t) + II(t).
\end{align*}
Then 
\begin{align*}
    I(t) &= \int_{\mathbb{R}^d} \langle [A(\Tilde{\vec{u}}) -A(\Tilde{\vec{v}})](\partial_{\alpha} \Tilde{\vec{u}}, \partial^{\alpha}\Tilde{\vec{u}}), \Tilde{\vec{w}}_{t} \rangle_{\R^m}\\
    &= \int_{\mathbb{R}^d} -\langle \partial_\alpha \Tilde{\vec{u}}, d\mathbf{E}^l_{\Tilde{\vec{u}}} \partial^\alpha \Tilde{\vec{u}}\rangle_{\R^m} \vec{E}^l_{\Tilde{\vec{u}}} + \langle \partial_\alpha \Tilde{\vec{u}}, d\mathbf{E}^l_{\Tilde{\vec{v}}} \partial^\alpha \Tilde{\vec{u}}\rangle_{\R^m} \vec{E}^l_{\Tilde{\vec{v}}}\\
    &= \int_{\mathbb{R}^d} -\langle \partial_\alpha \Tilde{\vec{u}}, \brac{d\mathbf{E}^l_{\Tilde{\vec{u}}} - d\mathbf{E}^l_{\Tilde{\vec{v}}}} \partial^\alpha \Tilde{\vec{u}}\rangle_{\R^m} \vec{E}^l_{\Tilde{\vec{u}}} - \langle \partial_\alpha \Tilde{\vec{u}}, d\mathbf{E}^l_{\Tilde{\vec{v}}} \partial^\alpha \Tilde{\vec{u}}\rangle_{\R^m} \brac{\vec{E}^l_{\Tilde{\vec{u}}} - \vec{E}^l_{\Tilde{\vec{v}}}}\\
    &\lesssim_{\Lambda} \int_{\mathbb{R}^d} |d \Tilde{\vec{u}}|^2 |\Tilde{\vec{w}}| |d \Tilde{\vec{w}}| \\
    &\lesssim_{\Lambda} \lVert d\Tilde{\vec{u}}\rVert_{L^{2n}(\R^d)}^2\lVert \Tilde{\vec{w}}\rVert_{L^{\frac{2n}{n-2}}(\R^d)} \lVert d \Tilde{\vec{w}} \rVert_{L^2(\R^d)}\\
    &\lesssim_{\Lambda} \lVert d\vec{u}\rVert_{L^{2n}(\R^d)}^2 \lVert d \Tilde{\vec{w}} \rVert_{L^2(\R^d)}^2.
\end{align*}
For the next integral, since $\partial \Tilde{\vec{u}} \in T_{\Tilde{\vec{u}}}\Tilde{\mathbb{H}}^2, \partial \Tilde{\vec{v}} \in T_{\Tilde{\vec{v}}}\Tilde{\mathbb{H}}^2$ and $A(\Tilde{\vec{u}})(\cdot,\cdot) \in \brac{T_{\Tilde{\vec{u}}}\Tilde{\mathbb{H}}^2}^\perp, A(\Tilde{\vec{v}})(\cdot,\cdot) \in \brac{T_{\Tilde{\vec{v}}}\Tilde{\mathbb{H}}^2}^\perp$, we have $\langle A(\Tilde{\vec{u}})(\cdot,\cdot), \partial_t \Tilde{\vec{u}}\rangle_{\R^m}= 0 = \langle A(\Tilde{\vec{v}})(\cdot,\cdot),\partial_t \Tilde{\vec{v}}\rangle$, which yields

\begin{align*}
    \langle A(\Tilde{\vec{v}})(\partial_\alpha \Tilde{\vec{u}},\partial^\alpha \Tilde{\vec{w}}),\partial_\alpha \Tilde{\vec{w}}_t\rangle
    &= \langle A(\Tilde{\vec{v}})(\partial_\alpha \Tilde{\vec{u}},\partial^\alpha \Tilde{\vec{w}}),\partial_t \Tilde{\vec{u}}\rangle_{\R^m}\\
    &= \langle [A(\Tilde{\vec{v}})-A(\Tilde{\vec{u}})](\partial_\alpha \Tilde{\vec{u}},\partial^\alpha \Tilde{\vec{w}}),\partial_t \Tilde{\vec{u}}\rangle_{\R^m}\\
    &\lesssim_{\Lambda} |d\Tilde{\vec{u}}|^2|\Tilde{\vec{w}}||d\Tilde{\vec{w}}|,
\end{align*}
and similarly,

\begin{align*}
    \langle A(\Tilde{\vec{v}})(\partial_\alpha \Tilde{\vec{v}},\partial^\alpha \Tilde{\vec{w}}),\partial_\alpha \Tilde{\vec{w}}_t\rangle
    &= \langle A(\Tilde{\vec{v}})(\partial_\alpha \Tilde{\vec{v}},\partial^\alpha \Tilde{\vec{w}}),\partial_t \Tilde{\vec{u}}\rangle_{\R^m}\\
    &= \langle [A(\Tilde{\vec{v}})-A(\Tilde{\vec{u}})](\partial_\alpha \Tilde{\vec{v}},\partial^\alpha \Tilde{\vec{w}}),\partial_t \Tilde{\vec{u}}\rangle_{\R^m}\\
    &\lesssim_{\Lambda} |d\Tilde{\vec{u}}|^2|\Tilde{\vec{w}}||d\Tilde{\vec{w}}| + |d\Tilde{\vec{v}}|^2|\Tilde{\vec{w}}||d\Tilde{\vec{w}}|.
\end{align*}

Hence,
\begin{align*}
    II(t) \lesssim_{\Lambda} \brac{\lVert d \vec{u} \rVert_{L^{2n}(\R^d)}^2 + \lVert d \vec{v} \rVert_{L^{2n}(\R^d)}^2} \lVert d \Tilde{\vec{w}} \rVert_{L^2(\R^d)}^2.
\end{align*}
\end{proof}
Since $\lVert \vec{u}\lVert_\infty, \lVert \vec{v}\lVert_\infty < \infty$, we have $\operatorname{Im}\vec{u},\operatorname{Im}\vec{v} \subset N = \overline{B}_r \cap \mathbb{H}^2$, where $\overline{B}_r$ is the compact ball in $\R^3$ with radius $r > \sup\{\lVert \vec{u}\lVert_\infty, \lVert \vec{v}\lVert_\infty\}$. Then $\varphi: N \to \varphi(N)$ is a diffeomorphism, so there is a diffeomorphism $\psi: \varphi(N) \to N$ that is the inverse of $\varphi$. 

We extend $\psi$ to a smooth map $\psi \in C_c^\infty(\R^m,\R^3)$ to get that

\begin{align} \label{w_eq}
    \vec{u}(x) - \vec{v}(x) 
    &= \psi(\Tilde{\vec{u}}(x))-\psi(\Tilde{\vec{v}}(x)) 
    = \int_0^1 d\psi_{\Tilde{\gamma}(s)(x)} \Tilde{\vec{w}}(x) \ ds,
\end{align}
where $\Tilde{\gamma}(s) = s\Tilde{\vec{u}}+(1-s)\Tilde{\vec{v}}$.

With this we have the following lemma.
\begin{lemma} \label{isom_est}
For $p \in (1,\infty]$,
\begin{align*}
    \lVert \Dso \Tilde{\vec{u}}\rVert_{L^{p}(\R^d)} &\lesssim_{\Lambda} \lVert \Dso\vec{u} \rVert_{L^{p}(\R^d)}, \\
     \lVert \Dso \Tilde{\vec{v}}\rVert_{L^{p}(\R^d)} &\lesssim_{\Lambda} \lVert \Dso \vec{v} \rVert_{L^{p}(\R^d)},
\end{align*}
and for $\sigma \in (\frac{1}{2},1)$ and $k \in \mathbb{Z}_+$
\begin{align*}
    \lVert \Dso^\sigma \int_0^1 d^k\psi_{\Tilde{\gamma}(s)}\ ds \rVert_{L^{\frac{2d}{2\gamma - 1}}(\R^d)} 
    \lesssim_{\Lambda}  \lVert \Dso \vec{u}\rVert_{L^{2d}(\R^d)} +  \lVert \Dso \vec{v}\rVert_{L^{2d}(\R^d)}.
\end{align*}
\end{lemma}
\begin{proof}
Note that by \cref{comp}
\begin{align*}
    \lVert \Dso \Tilde{\vec{u}}\rVert_{L^{p}(\R^d)}
    \lesssim \lVert \nabla \Tilde{\vec{u}}\rVert_{L^{p}(\R^d)}
\end{align*}
and
\begin{align*}
    \abs{\nabla\Tilde{\vec{u}}}
    = &  \abs{d\varphi_{\vec{u}} \nabla\vec{u}} \\
    \lesssim_{\Lambda} & \abs{\nabla\vec{u}}.
\end{align*}
By again using \cref{comp},
\begin{align*}
    \lVert \Dso \Tilde{\vec{u}}\rVert_{L^{p}(\R^d)} &\lesssim_{\Lambda} \lVert \Dso\vec{u} \rVert_{L^{p}(\R^d)}.
\end{align*}
Similarly,
\begin{align*}
   \lVert \Dso \Tilde{\vec{v}}\rVert_{L^{p}(\R^d)} &\lesssim_{\Lambda} \lVert \Dso \vec{v} \rVert_{L^{p}(\R^d)}.
\end{align*}
Next, by the Sobolev inequality, \cref{comp} and for $k \in \Z_+$,
\begin{align*}
    \lVert \Dso^\sigma\int_0^1 d^k\psi_{\Tilde{\gamma}(s)}\ ds \rVert_{L^{\frac{2d}{2\gamma - 1}}(\R^d)} 
    &\lesssim \lVert \Dso \int_0^1 d^k\psi_{\Tilde{\gamma}(s)}\ ds \rVert_{L^{2d}(\R^d)}\\
    & \lesssim \lVert \nabla \int_0^1 d^k\psi_{\Tilde{\gamma}(s)}\ ds \rVert_{L^{2d}(\R^d)}.
\end{align*}
Since we are working in a compact subset of $\mathbb{H}^2$, $d\varphi, d^2\psi$ are uniformly bounded, so we have
\begin{align*}
    \abs{\nabla \brac{\int_0^1 d^k\psi_{\Tilde{\gamma}(s)}\ ds}}
    &= \abs{\int_0^1 d^{k+1}\psi_{\Tilde{\gamma}(s)} \nabla\Tilde{\gamma}(s) \ ds}\\
    &= \abs{\int_0^1 d^{k+1}\psi_{\Tilde{\gamma}(s)} \nabla\brac{s\Tilde{\vec{u}}+(1-s)\Tilde{\vec{v}}} \ ds}\\ 
    \lesssim_{\Lambda} & \abs{\nabla\Tilde{\vec{u}}} + \abs{\nabla\Tilde{\vec{v
}}}.
\end{align*}
Thus, using \cref{comp}
\begin{align*}
    \lVert \Dso^\sigma \int_0^1 d^k\psi_{\Tilde{\gamma}(s)}\ ds \rVert_{L^{\frac{2d}{2\gamma - 1}}(\R^d)}
    & \lesssim \lVert \nabla \int_0^1 d^k\psi_{\Tilde{\gamma}(s)}\ ds \rVert_{L^{2d}(\R^d)}\\
    &\lesssim_{\Lambda}  \lVert \Dso \vec{u}\rVert_{L^{2d}(\R^d)} +  \lVert \Dso \vec{v}\rVert_{L^{2d}(\R^d)}.
\end{align*}
\end{proof}
Furthermore, we can extend $\vec{X}^r$ as well so that
\begin{align}
    \vec{X}^r_{\vec{u}} - \vec{X}^r_{\vec{v}} 
    &= \int_0^1 d\vec{X}^r_{\gamma(s)}\vec{w} \ ds \label{vec_ident}
\end{align}
where $\gamma(s) = s\vec{u}+(1-s)\vec{v}$, and
\begin{align}
    \vec{E}^r_{\Tilde{\vec{u}}} - \vec{E}^r_{\Tilde{\vec{v}}} 
    =& \int_0^1 d\vec{E}^r_{\Tilde{\gamma}(s)}\Tilde{\vec{w}} \ ds \label{vec_indet_2},
\end{align}
where $d\vec{X}^r_{\gamma(s)}$ and $d\vec{E}^r_{\Tilde{\gamma}(s)}$ are bounded. 
We now tackle integral \eqref{eq:diffenergy_2}. We can rewrite it as such

\begin{align}
    &\int_{\mathbb{R}^n} \Big\langle \langle\vec{u} , |\nabla| \vec{u}\rangle_L P_\vec{u}( |\nabla|\vec{u}), \mathbf{X}_\vec{u}^r \Big\rangle_L \langle\mathbf{E}_{\Tilde{\vec{u}}}^r, \Tilde{\vec{w}}_{t} \rangle_{\R^m} - \Big\langle \langle\vec{v} , |\nabla| \vec{v}\rangle_L P_\vec{v}(|\nabla|\vec{v}), \mathbf{X}_\vec{v}^r\Big\rangle_L\langle\mathbf{E}_{\Tilde{\vec{v}}}^r, \Tilde{\vec{w}}_{t} \rangle_{\R^m} \nonumber\\
    =&\int_{\mathbb{R}^n} \Big\langle \langle\vec{u} , |\nabla| \vec{u}\rangle_L P_\vec{u}(|\nabla|\vec{u}) - \langle\vec{v} , |\nabla| \vec{v}\rangle_L P_\vec{v}(|\nabla|\vec{v}), \mathbf{X}_\vec{u}^r\Big\rangle_L\langle\mathbf{E}_{\Tilde{\vec{u}}}^r, \Tilde{\vec{w}}_{t} \rangle_{\R^m} \label{eq:diffenergy_2_1}\\
    +&\int_{\mathbb{R}^n} \Big\langle \langle\vec{v} , |\nabla| \vec{v}\rangle_L P_\vec{v}(|\nabla|\vec{v}), (\mathbf{X}_\vec{u}^r-\mathbf{X}_\vec{v}^r)\Big\rangle_L\langle\mathbf{E}_{\Tilde{\mathbf{ u}}}^r, \Tilde{\vec{w}}_{t} \rangle_{\R^m} \label{eq:diffenergy_2_2}\\
    +&\int_{\mathbb{R}^n} \Big\langle \langle\vec{v} , |\nabla| \vec{v}\rangle_L P_\vec{v}(|\nabla|\vec{v}), \mathbf{X}_\vec{v}^r\Big\rangle_L\langle(\mathbf{E}_{\Tilde{\vec{u}}}^r-\mathbf{E}_{\Tilde{\vec{v}}}^r), \Tilde{\vec{w}}_{t} \rangle_{\R^m}. \label{eq:diffenergy_2_3}
\end{align}

From here on out, $\sideset{}{'}{\sum}\limits_{j=1}^3$ will denote
\begin{align*}
    \sideset{}{'}{\sum}_{j=1}^3\vec{p}^j\vec{q}^j
    &= -\vec{p}^1\vec{q}^1 + \vec{p}^2\vec{q}^2 + \vec{p}^3\vec{q}^3\\
    & =\langle \vec{p}, \vec{q} \rangle_L.
\end{align*}
Here we define $T_L$ by $T_L(\vec{f})(\vec{g}) = \langle \vec{f}, \vec{g} \rangle_L$, so 
\begin{align*}
    H_{\Dso^s}(T_L(\vec{f}),\vec{g}) = \Dso^s (T_L(\vec{f})(\vec{g})) + T_L(\Dso^s\vec{f})(\vec{g}) + T_L(\vec{f}) \Dso^s\vec{g}.
\end{align*}
We estimate the above integral term by term, beginning by estimating \eqref{eq:diffenergy_2_1}.
\begin{lemma}
    For $d \ge 3$, we have that 
    \begin{align*}
        &\int_{\mathbb{R}^n}  \left\langle \langle\vec{u} , |\nabla| \vec{u}\rangle_LP_\vec{u}(\Dso\vec{u}) - \langle\vec{v} , \Dso \vec{v}\rangle_LP_\vec{v}(\Dso\vec{v}), \mathbf{X}_\vec{u}^r \right\rangle_L\langle\mathbf{E}_{\Tilde{\vec{u}}}^r, \Tilde{\vec{w}}_{t} \rangle_{\R^m}\\
        &\lesssim_{\Lambda} (\lVert |\nabla|\vec{u}\rVert_{L^{2d}(\R^d)}^2+ \lVert |\nabla|\vec{v}\rVert_{L^{2d}(\R^d)}^2)(\lVert  \nabla \Tilde{\vec{w}} \rVert_{L^2(\R^d)}^2 +\lVert \Tilde{\vec{w}}_t\rVert_{L^2(\R^d)}^2).
    \end{align*} 
\end{lemma}
\begin{proof}
Observe that
\begin{align*}
    H_{\Dso}\brac{T_L(\vec{u}),\vec{u}} 
    &= \Dso\brac{\langle \vec{u}, \vec{u}\rangle} - \langle \Dso \vec{u}, \vec{u} \rangle_L - \langle \vec{u}, \Dso\vec{u} \rangle_L \\
    &= -\Dso 1 - 2 \langle \vec{u}, \Dso \vec{u} \rangle_L\\
    &= -2 \langle \vec{u}, \Dso \vec{u} \rangle_L, 
\end{align*}
since $\langle \vec{u}, \vec{u} \rangle_L = -1$ and $\Dso$ applied to a constant is $0$. So 
\begin{align*}
    \langle \vec{u}, \Dso\vec{u} \rangle_L = -\frac{1}{2} H_{\Dso}\brac{T_L(\vec{u}),\vec{u}}.
\end{align*}
Now using multi-linearity, we have
\begin{align*}
    &\langle\vec{u} , |\nabla| \vec{u}\rangle_LP_\vec{u}(\Dso\vec{u}) - \langle\vec{v} , \Dso \vec{v}\rangle_LP_\vec{v}(\Dso\vec{v})\\
    =&-\frac{1}{2} H_{|\nabla|}(T_L(\vec{u}), \vec{u})P_\vec{u}(|\nabla|\vec{w})\\
    &-\frac{1}{2}H_{|\nabla|}(T_L(\vec{u}), \vec{u})(P_\vec{u}-P_\vec{v}) |\nabla| \vec{v}\\
    &-\frac{1}{2}H_{|\nabla|}(T_L(\vec{w}), \vec{u})P_\vec{v}(|\nabla|\vec{v})\\
    &-\frac{1}{2} H_{|\nabla|}(T_L(\vec{v}), \vec{w})P_\vec{v}(|\nabla|\vec{v}).
\end{align*}
Then by integrating,
\begin{align}
     &\int_{\mathbb{R}^n}  \left\langle \langle\vec{u} , |\nabla| \vec{u}\rangle_LP_\vec{u}(\Dso\vec{u}) - \langle\vec{v} , \Dso \vec{v}\rangle_LP_\vec{v}(\Dso\vec{v}), \mathbf{X}_\vec{u}^r \right\rangle_L\langle\mathbf{E}_{\Tilde{\vec{u}}}^r, \Tilde{\vec{w}}_{t} \rangle_{\R^m} \nonumber\\
    =&-\frac{1}{2}\int_{\mathbb{R}^n}  \big\langle H_{|\nabla|}(T_L(\vec{u}), \vec{u})P_\vec{u}(|\nabla|\vec{w}), \mathbf{X}_\vec{u}^r\big\rangle_L\langle\mathbf{E}_{\Tilde{\vec{u}}}^r, \Tilde{\vec{w}}_t \rangle_{\R^m}\label{eq:energydiff_2_1_1}\\
    &-\frac{1}{2}\int_{\mathbb{R}^n} \big\langle H_{|\nabla|}(T_L(\vec{u}), \vec{u})(P_\vec{u}-P_\vec{v}) |\nabla| \vec{v}, \mathbf{X}_\vec{u}^r\big\rangle_L\langle\mathbf{E}_{\Tilde{\vec{u}}}^r, \Tilde{\vec{w}}_t \rangle_{\R^m} \label{eq:energydiff_2_1_2}\\
    &-\frac{1}{2}\int_{\mathbb{R}^n} \big\langle H_{|\nabla|}(T_L(\vec{w}), \vec{u})P_\vec{v}(|\nabla|\vec{v}), \mathbf{X}_\vec{u}^r\big\rangle_L \langle \mathbf{E}_{\Tilde{\vec{u}}}^r, \Tilde{\vec{w}}_t \rangle_{\R^m} \label{eq:energydiff_2_1_3}\\
    &-\frac{1}{2}\int_{\mathbb{R}^n} \big\langle H_{|\nabla|}(T_L(\vec{v}), \vec{w})P_\vec{v}(|\nabla|\vec{v}), \mathbf{X}_\vec{u}^r\big\rangle_L \langle\mathbf{E}_{\Tilde{\vec{u}}}^r, \Tilde{\vec{w}}_t \rangle_{\R^m} \label{eq:energydiff_2_1_4}.
\end{align}

We estimate the above term by term. We start by estimating \eqref{eq:energydiff_2_1_1} using \cref{infty_bound}, we have
\begin{align*}
    &\int_{\mathbb{R}^n} \big\langle H_{|\nabla|}(T_L(\vec{u}), \vec{u})P_\vec{u}(\Dso\vec{w}), \mathbf{X}_\vec{u}^r\big\rangle_L\langle\mathbf{E}_{\Tilde{\vec{u}}}^r, \Tilde{\vec{w}}_t \rangle_{\R^m}\\
     \lesssim_{\Lambda} & \lVert H_{|\nabla|}(T_L(\vec{u}), \vec{u})\rVert_{L^\infty(\R^d)} \ \lVert \nabla|\vec{w} \rVert_{L^{2}(\R^d)} \ \lVert \Tilde{\vec{w}}_t \rVert_{L^{2}(\R^d)}\\
    &\overset{\cref{infty_bound}}{\lesssim_{\Lambda}} \lVert |\nabla|\vec{u} \rVert_{L^{(2d,2)}(\R^d)}^2 \brac{\lVert \nabla \vec{w}\rVert_{L^2(\R^d)}^2 + \lVert \Tilde{\vec{w}}_t\rVert_{L^2(\R^d)}^2}.
\end{align*}
It is evident by the definition of $P_{\vec{u}}$ that 
\begin{align*}
|(P_{\vec{u}}-P_\vec{v})(\vec{x})|
&=|\vec{x} + \langle \vec{u}, \vec{x} \rangle_L\vec{u} - \vec{x} - \langle \vec{v}, \vec{x} \rangle_L\vec{v}|\\
&=|\langle \vec{u} - \vec{v}, \vec{x} \rangle_L\vec{u} - \langle \vec{v}, \vec{x} \rangle_L(\vec{u} - \vec{v})|\\
&\aleq_{\Lambda} |\vec{u}-\vec{v}| |\vec{x}|,
\end{align*}
since $\vec{u}$ and $\vec{v}$ are bounded. So for \eqref{eq:energydiff_2_1_2} we use $\cref{lem:prod_rule}$ for $\sigma \in (1,\frac{1}{2})$ and \cref{lem:gag_in} to get
\begin{align*}
    &\int_{\mathbb{R}^n} \big\langle H_{|\nabla|}(T_L(\vec{u}), \vec{u})\textcolor{red}{(P_\vec{u}-P_\vec{v})}|\nabla| \vec{v}, \mathbf{X}_\vec{u}^r\big\rangle_L \langle\mathbf{E}_{\Tilde{\vec{u}}}^r, \Tilde{\vec{w}}_t \rangle_{\R^m}\\
    \lesssim_{\Lambda} & \int_{\mathbb{R}^n}|\langle H_{|\nabla|}(T_L(\vec{u}), \vec{u})| \ \left|\vec{u}-\vec{v}\right| \ \left| |\nabla|\vec{v} \right| \ |\Tilde{\vec{w}}_t| \\
    \lesssim_{\Lambda} & \int_{\mathbb{R}^n}|\langle H_{|\nabla|}(T_L(\vec{u}), \vec{u})| \ \left| |\nabla|\vec{v} \ \int_0^1 d\psi_{\Tilde{\gamma}(s)} \Tilde{\vec{w}} \ ds\right| \ |\Tilde{\vec{w}}_t| \\
    \lesssim_{\Lambda} & \lVert H_{|\nabla|}(T_L(\vec{u}), \vec{u}) \rVert_{L^{2d}(\R^d)} \ \lVert |\nabla|\vec{v}\rVert_{L^{2d}(\R^d)} \ \lVert \Tilde{\vec{w}} \rVert_{L^{\frac{2d}{d-2}}(\R^d)} \lVert \Tilde{\vec{w}}_t\rVert_{L^{2}(\R^d)}\\
    \overset{\cref{lem:prod_rule}}{\lesssim_{\Lambda}} & \lVert \Dso^{1-\sigma}\vec{u} \rVert_{L^{\frac{2d}{1-\sigma}}(\R^d)} \ \lVert \Dso^{\sigma} \vec{u} \rVert_{L^{\frac{2d}{\sigma}}(\R^d)} \ \lVert |\nabla|\vec{v}\rVert_{L^{2d}(\R^d)} \ \lVert \Tilde{\vec{w}} \rVert_{L^{\frac{2d}{d-2}}(\R^d)} \lVert \Tilde{\vec{w}}_t\rVert_{L^{2}(\R^d)}\\
    \overset{\cref{lem:gag_in}}{\lesssim_{\Lambda}} & \lVert \Dso\vec{u} \rVert_{L^{2d}(\R^d)}^{1-\sigma} \ \lVert \Dso \vec{u} \rVert_{L^{2d}(\R^d)}^\sigma \ \lVert \Dso \vec{v} \rVert_{L^{2d}(\R^d)} \ \lVert \Dso\Tilde{\vec{w}} \rVert_{L^{2}(\R^d)} \ \lVert \Tilde{\vec{w}}_t \rVert_{L^{2}(\R^d)}\\
    \lesssim_{\Lambda} & \brac{\lVert \Dso \vec{u} \rVert_{L^{2d}(\R^d)}^2 + \lVert \Dso \vec{v} \rVert_{L^{2d}(\R^d)}^2} \ \brac{\lVert \Dso\Tilde{\vec{w}} \rVert_{L^{2}(\R^d)}^2 + \lVert \Tilde{\vec{w}}_t \rVert_{L^{2}(\R^d)}^2}.
\end{align*}
Recall that $\vec{w} = \int_0^1 d\psi_{\Tilde{\gamma}(s)} \Tilde{\vec{w}} \ ds$, where $\Tilde{\gamma}(t) = s\Tilde{\vec{u}} +(1-s)\Tilde{\vec{v}}$.
So for \eqref{eq:energydiff_2_1_3}, using \cref{comm_lem}, we have
\begin{align}
    &\int_{\mathbb{R}^n} \big\langle H_{|\nabla|}(T_L(\vec{w}), \vec{u})P_\vec{v}(|\nabla|\vec{v}), \mathbf{X}_\vec{u}^r\big\rangle_L\langle \mathbf{E}_{\Tilde{\vec{u}}}^r, \Tilde{\vec{w}}_t \rangle\nonumber\\
    &=\int_{\mathbb{R}^n} \left\langle H_{|\nabla|}\brac{T_L\brac{\int_0^1 d\psi_{\Tilde{\gamma}(s)} \Tilde{\vec{w}} \ ds}, \vec{u}} P_\vec{v}(|\nabla|\vec{v}), \mathbf{X}_\vec{u}^r\right\rangle_L\langle \mathbf{E}_{\Tilde{\vec{u}}}^r, \Tilde{\vec{w}}_t \rangle\nonumber\\
    &=\sum_{\ell=1}^m \sideset{}{'}{\sum}_{j=1}^3 \int_{\mathbb{R}^n} \left\langle H_{|\nabla|}\brac{\int_0^1 \frac{\partial \psi^j_{\Tilde{\gamma}(s)}}{\partial p^\ell} \Tilde{\vec{w}}^\ell \ ds, \vec{u}^j}P_\vec{v}(|\nabla|\vec{v}), \mathbf{X}_\vec{u}^r\right\rangle_L\langle \mathbf{E}_{\Tilde{\vec{u}}}^r, \Tilde{\vec{w}}_t \rangle\nonumber\\
    \overset{\cref{comm_lem}}{=} &\sum_{\ell=1}^m \sideset{}{'}{\sum}_{j=1}^3 \int_{\mathbb{R}^n} \big\langle H_{|\nabla|}\brac{\int_0^1 \frac{\partial \psi^j_{\Tilde{\gamma}(s)}}{\partial p^\ell} \ ds \ \vec{u}^j, \Tilde{\vec{w}}^\ell} P_\vec{v}(|\nabla|\vec{v}), \mathbf{X}_\vec{u}^r\big\rangle_L\langle \mathbf{E}_{\Tilde{\vec{u}}}^r, \Tilde{\vec{w}}_t \rangle\label{beg_2_1}\\
    &+\sum_{\ell=1}^m\sideset{}{'}{\sum}_{j=1}^3\int_{\mathbb{R}^n} \big\langle H_{|\nabla|}\brac{\int_0^1 \frac{\partial \psi^j_{\Tilde{\gamma}(s)}}{\partial p^\ell} \ ds,\vec{u}^j} \Tilde{\vec{w}}^\ell P_\vec{v}(|\nabla|\vec{v}), \mathbf{X}_\vec{u}^r\big\rangle_L\langle \mathbf{E}_{\Tilde{\vec{u}}}^r, \Tilde{\vec{w}}_t \rangle\label{beg_2_2}\\
    &+\sideset{}{'}{\sum}_{j=1}^3\sum_{\ell=1}^m\int_{\mathbb{R}^n} \big\langle \int_0^1 \frac{\partial \psi^j_{\Tilde{\gamma}(s)}}{\partial p^\ell} \ ds H_{|\nabla|}\brac{\vec{u}^j, \Tilde{\vec{w}}^\ell} P_\vec{v}(|\nabla|\vec{v}), \mathbf{X}_\vec{u}^r\big\rangle_L\langle \mathbf{E}_{\Tilde{\vec{u}}}^r, \Tilde{\vec{w}}_t \rangle\label{beg_2_3}.
\end{align}
For \eqref{beg_2_1}, using \cref{lem:prod_rule} for $\sigma \in (0,\frac{1}{2})$, we have
\begin{align*}
    &\int_{\mathbb{R}^n} \big\langle H_{|\nabla|}\brac{\int_0^1 \frac{\partial \psi^j_{\Tilde{\gamma}(s)}}{\partial p^\ell} \ ds \ \vec{u}^j, \Tilde{\vec{w}}^\ell} P_\vec{v}(|\nabla|\vec{v}), \mathbf{X}_\vec{u}^r\big\rangle_L\langle \mathbf{E}_{\Tilde{\vec{u}}}^r, \Tilde{\vec{w}}_t \rangle\\
    \lesssim_{\Lambda} & \lVert \Dso^{\sigma}\brac{\int_0^1 \frac{\partial \psi^j_{\Tilde{\gamma}(s)}}{\partial p^\ell} \ ds \ \vec{u}^j} \rVert_{L^{\frac{2d}{2\sigma-1}}(\R^d)} \ \lVert \Dso^{1-\sigma} \Tilde{\vec{w}} \rVert_{L^{\frac{2d}{d-2\sigma}}(\R^d)} \ \lVert \Dso \vec{v} \rVert_{L^{2d}(\R^d)} \ \lVert \Tilde{\vec{w}}_t \rVert_{L^{2}(\R^d)}\\
    \overset{\cref{lem:prod_rule_frac}}{\lesssim_{\Lambda}} & \brac{\lVert \Dso \vec{u} \rVert_{L^{2d}(\R^d)}^2 + \lVert \Dso \vec{v} \rVert_{L^{2d}(\R^d)}^2} \ \lVert \Dso \Tilde{\vec{w}} \rVert_{L^{2d}(\R^d)} \ \lVert \Tilde{\vec{w}}_t \rVert_{L^{2}(\R^d)}\\
    \lesssim_{\Lambda} & \brac{\lVert \Dso \vec{u} \rVert_{L^{2d}(\R^d)}^2 + \lVert \Dso \vec{v} \rVert_{L^{2d}(\R^d)}^2} \ \brac{\lVert \Dso \Tilde{\vec{w}} \rVert_{L^{2}(\R^d)}^2 + \lVert \Tilde{\vec{w}}_t \rVert_{L^{2}(\R^d)}^2}.
\end{align*}
 For \eqref{beg_2_2}, using \cref{lem:prod_rule} for $\sigma \in (0,\frac{1}{2})$, we have
\begin{align*}
    &\int_{\mathbb{R}^n} \big\langle H_{|\nabla|}\brac{\int_0^1 \frac{\partial \psi^j_{\Tilde{\gamma}(s)}}{\partial p^\ell} \ ds,\vec{u}^j} \Tilde{\vec{w}}^\ell P_\vec{v}(|\nabla|\vec{v}), \mathbf{X}_\vec{u}^r\big\rangle_L\langle \mathbf{E}_{\Tilde{\vec{u}}}^r, \Tilde{\vec{w}}_t \rangle\\
    \lesssim_{\Lambda} & \lVert \Dso^{\sigma} \int_0^1 \frac{\partial \psi^j_{\Tilde{\gamma}(s)}}{\partial p^\ell} \ ds \rVert_{L^{\frac{2d}{\sigma}}(\R^d)} \ \lvert \Dso^{1-\sigma} \vec{u}^j \rVert_{L^{\frac{2d}{1-\sigma}}(\R^d)} \ \lVert \Tilde{\vec{w}} \rVert_{L^{\frac{2d}{d-2}}(\R^d)} \ \lVert \Dso \vec{v} \rVert_{L^{2d}(\R^d)} \ \lVert \Tilde{\vec{w}}_t \rVert_{L^{2}(\R^d)}\\
    \overset{\cref{isom_est}}{\lesssim_{\Lambda}} & \brac{\lVert \Dso \vec{u} \rVert_{L^{2d}(\R^d)}^2 + \lVert \Dso \vec{v} \rVert_{L^{2d}(\R^d)}^2} \ \lVert \Dso \Tilde{\vec{w}} \rVert_{L^{2d}(\R^d)} \ \lVert \Tilde{\vec{w}}_t \rVert_{L^{2}(\R^d)}\\
    \lesssim_{\Lambda} & \brac{\lVert \Dso \vec{u} \rVert_{L^{2d}(\R^d)}^2 + \lVert \Dso \vec{v} \rVert_{L^{2d}(\R^d)}^2} \ \brac{\lVert \Dso \Tilde{\vec{w}} \rVert_{L^{2}(\R^d)}^2 + \lVert \Tilde{\vec{w}}_t \rVert_{L^{2}(\R^d)}^2}.
\end{align*}
For \eqref{beg_2_3}, using \cref{lem:prod_rule} with $\sigma \in (0,\frac{1}{2})$
\begin{align*}
    &\int_{\mathbb{R}^n} \big\langle \int_0^1 \frac{\partial \psi^j_{\Tilde{\gamma}(s)}}{\partial p^\ell} \ ds \ H_{|\nabla|}\brac{\vec{u}^j, \Tilde{\vec{w}}^\ell} P_\vec{v}(|\nabla|\vec{v}), \mathbf{X}_\vec{u}^r\big\rangle_L\langle \mathbf{E}_{\Tilde{\vec{u}}}^r, \Tilde{\vec{w}}_t \rangle\\
    \lesssim_{\Lambda} & \lVert \Dso^{\sigma} \vec{u}^j \rVert_{L^{\frac{2d}{2\sigma-1}}(\R^d)} \ \lvert \Dso^{1-\sigma} \Tilde{\vec{w}}^\ell\rVert_{L^{\frac{2d}{d-2\sigma}}(\R^d)} \ \lVert \Dso \vec{v} \rVert_{L^{2d}(\R^d)} \ \lVert \Tilde{\vec{w}}_t \rVert_{L^{2}(\R^d)}\\
    \lesssim_{\Lambda} & \brac{\lVert \Dso \vec{u} \rVert_{L^{2d}(\R^d)}^2 + \lVert \Dso \vec{v} \rVert_{L^{2d}(\R^d)}^2} \ \lVert \Dso \Tilde{\vec{w}} \rVert_{L^{2d}(\R^d)} \ \lVert \Tilde{\vec{w}}_t \rVert_{L^{2}(\R^d)}\\
    \lesssim_{\Lambda} & \brac{\lVert \Dso \vec{u} \rVert_{L^{2d}(\R^d)}^2 + \lVert \Dso \vec{v} \rVert_{L^{2d}(\R^d)}^2} \ \brac{\lVert \Dso \Tilde{\vec{w}} \rVert_{L^{2}(\R^d)}^2 + \lVert \Tilde{\vec{w}}_t \rVert_{L^{2}(\R^d)}^2}.
\end{align*}
Similarly for \eqref{eq:energydiff_2_1_4}, 
\begin{align*}
    &
    \int_{\mathbb{R}^n} \big\langle H_{|\nabla|}(T_L(\vec{v}), \vec{w})P_\vec{v}(|\nabla|\vec{v}), \mathbf{X}_\vec{u}^r\big\rangle_L\langle\mathbf{E}_{\Tilde{\vec{u}}}^r, \Tilde{\vec{w}}_t \rangle_{\R^m}\\
    \lesssim_{\Lambda} & \brac{\lVert \Dso \vec{u} \rVert_{L^{2d}(\R^d)}^2 + \lVert \Dso \vec{v} \rVert_{L^{2d}(\R^d)}^2} \ \brac{\lVert \Dso \Tilde{\vec{w}} \rVert_{L^{2}(\R^d)}^2 + \lVert \Tilde{\vec{w}}_t \rVert_{L^{2}(\R^d)}^2}.
\end{align*}
\end{proof}
Now we continue estimating \eqref{eq:diffenergy_2}. For \eqref{eq:diffenergy_2_2}, we use \eqref{vec_ident},
\begin{align*}
    &\int_{\mathbb{R}^n} \big\langle P_\vec{v}(\langle\vec{v} , |\nabla| \vec{v}\rangle_L|\nabla|\vec{v}), (\mathbf{X}_\vec{u}^r-\mathbf{X}_\vec{v}^r)\big\rangle_L\langle\mathbf{E}_{\Tilde{\vec{v}}}^r, \Tilde{\vec{w}}_t \rangle_{\R^m}\\
    &\lesssim_{\Lambda}
    \lVert |\nabla| \vec{v} \rVert_{L^{2d}(\R^d)}^2\lVert |\nabla| \Tilde{\vec{w}}_t \rVert_{L^2(\R^d)}^2,
\end{align*}
and similarly for \eqref{eq:diffenergy_2_3}, we use \eqref{vec_indet_2}
\begin{align*}
    &\int_{\mathbb{R}^n} \big\langle P_\vec{v}(\langle\vec{v} , |\nabla| \vec{v}\rangle_L|\nabla|\vec{v}), \mathbf{X}_\vec{v}^r\big\rangle_L\langle\brac{\mathbf{E}_{\Tilde{\vec{u}}}^r-\mathbf{E}_{\Tilde{\vec{v}}}^r}, \Tilde{\vec{w}}_t \rangle_{\R^m}\\
    & \lesssim_{\Lambda} \lVert |\nabla| \vec{v} \rVert_{L^{2d}(\R^2)}^2\lVert |\nabla| \Tilde{\vec{w}}_t \rVert_{L^2(\R^d)}^2.
\end{align*}
We next rewrite \eqref{eq:diffenergy_3} using multilinearity
\begin{align}
    &\int_{\mathbb{R}^n} \langle \vec{u} \wedge_LH_{|\nabla|}(\vec{u} \wedge_L,|\nabla|\vec{u}), \mathbf{X}_\vec{u}^r\rangle_L\langle\mathbf{E}_{\Tilde{\vec{u}}}^r, \Tilde{\vec{w}}_t \rangle_{\R^m}- \langle \vec{v} \wedge_LH_{|\nabla|}(\vec{v} \wedge_L,|\nabla|\vec{v}), \mathbf{X}_\vec{v}^r\rangle_L\langle\mathbf{E}_{\Tilde{\vec{v}}}^r, \Tilde{\vec{w}}_t \rangle_{\R^m}\nonumber\\
    &=\int_{\mathbb{R}^n} \langle \vec{w} \wedge_LH_{|\nabla|}(\vec{u} \wedge_L,|\nabla|\vec{u}), \mathbf{X}_\vec{u}^r\rangle_L\langle\mathbf{E}_{\Tilde{\vec{u}}}^r, \Tilde{\vec{w}}_t \rangle_{\R^m}dx \label{eq:diffenergy_3_1}\\
    &+\int_{\mathbb{R}^n} \langle \vec{v} \wedge_LH_{|\nabla|}(\vec{w} \wedge_L,|\nabla|\vec{u}), \mathbf{X}_\vec{u}^r\rangle_L\langle \mathbf{E}_{\Tilde{\vec{u}}}^r, \Tilde{\vec{w}}_t \rangle_{\R^m}dx\label{eq:diffenergy_3_2}\\
    &+\int_{\mathbb{R}^n} \langle \vec{v} \wedge_LH_{|\nabla|}(\vec{v} \wedge_L,|\nabla|\vec{w}), \mathbf{X}_\vec{u}^r\rangle_L\langle\mathbf{E}_{\Tilde{\vec{u}}}^r, \Tilde{\vec{w}}_t \rangle_{\R^m}dx\label{eq:diffenergy_3_3}\\
    &+\int_{\mathbb{R}^n} \langle \vec{v} \wedge_LH_{|\nabla|}(\vec{v} \wedge_L,|\nabla|\vec{v}), \mathbf{X}_\vec{u}^r-\mathbf{X}_\vec{v}^r\rangle_L\langle\mathbf{E}_{\Tilde{\vec{u}}}^r, \Tilde{\vec{w}}_t \rangle_{\R^m}dx\label{eq:diffenergy_3_4}\\
    &+\int_{\mathbb{R}^n} \langle \vec{v} \wedge_LH_{|\nabla|}(\vec{v} \wedge_L,|\nabla|\vec{v}), \mathbf{X}_\vec{u}^r\rangle_L\langle\brac{\mathbf{E}_{\Tilde{\vec{u}}}^r-\mathbf{E}_{\Tilde{\vec{v}}}^r}, \Tilde{\vec{w}}_t \rangle_{\R^m}dx.\label{eq:diffenergy_3_5}
\end{align}
The last two integrals \eqref{eq:diffenergy_3_4} and \eqref{eq:diffenergy_3_5} can be estimated using \cref{lem:prod_rule} for $\sigma \in (0,\frac{1}{2})$
\begin{align*}
    &\int_{\mathbb{R}^n} \langle \vec{v} \wedge_LH_{|\nabla|}(\vec{v} \wedge_L,|\nabla|\vec{v}), \mathbf{X}_\vec{u}^r-\mathbf{X}_\vec{v}^r\rangle_L\langle \mathbf{E}_{\Tilde{\vec{u}}}^r, \Tilde{\vec{w}}_t \rangle_{\R^m}dx\\
    &\lesssim_{\Lambda} \lVert H_{|\nabla|}(\vec{v} \wedge_L,|\nabla|\vec{v}) \rVert_{L^d(\R^d)} \lVert \vec{w} \rVert_{L^{\frac{2d}{d-2}}(\R^d)}\lVert \Tilde{\vec{w}}_t \rVert_{L^2(\R^d)}\\
    &\lesssim_{\Lambda} \lVert |\nabla|^{1-\sigma}\vec{v} \rVert_{L^{\frac{2d}{2(1-\sigma) -1}}(\R^d)} \lVert |\nabla|^{1+\sigma}\vec{v} \rVert_{L^{\frac{2d}{2(1+\sigma) -1}}(\R^d)} \ \brac{\lVert \Dso \Tilde{\vec{w}}\rVert_{L^2(\R^d)}^2 + \lVert \Tilde{\vec{w}}_t \rVert_{L^2(\R^d)}^2}\\
    &\lesssim_{\Lambda} \lVert |\nabla|^{1+\sigma}\vec{v} \rVert_{L^{\frac{2d}{2(1+\sigma) -1}}(\R^d)}^2 \ \brac{\lVert \Dso \Tilde{\vec{w}} \rVert_{L^2(\R^d)}^2 + \lVert \Tilde{\vec{w}}_t \rVert_{L^2(\R^d)}^2}
\end{align*}
and similarly
\begin{align*}
    &\int_{\mathbb{R}^n} \langle \vec{v} \wedge_LH_{|\nabla|}(\vec{v} \wedge_L,|\nabla|\vec{v}), \mathbf{X}_\vec{u}^r\rangle_L\langle\brac{\mathbf{E}_{\Tilde{\vec{u}}}^r-\mathbf{E}_{\Tilde{\vec{v}}}^r}, \Tilde{\vec{w}}_t \rangle_{\R^m}dx\\
    &\lesssim_{\Lambda}\lVert |\nabla|^{1+\sigma}\vec{v} \rVert_{L^{\frac{2d}{2(1+\sigma) -1}}(\R^d)}^2 \ \brac{\lVert \Dso \Tilde{\vec{w}} \rVert_{L^2(\R^d)}^2 + \lVert \Tilde{\vec{w}}_t \rVert_{L^2(\R^d)}^2}.
\end{align*}
We continue estimating \eqref{eq:diffenergy_3} by next estimating \eqref{eq:diffenergy_3_1}.
\begin{lemma}
For $d\geq 3$, and any $\alpha \in (1,d+\frac{1}{2})$ we have
\begin{align*}
    &\abs{\int_{\mathbb{R}^n} \langle \vec{w} \wedge_LH_{|\nabla|}(\vec{u} \wedge_L,|\nabla|\vec{u}), \mathbf{X}_\vec{u}^r\rangle_L\langle\mathbf{E}_{\Tilde{\vec{u}}}^r, \Tilde{\vec{w}}_t \rangle_{\R^m}} \\
    \lesssim_{\Lambda} & \brac{\|\nabla\Tilde{\vec{w}}\|_{L^2(\R^d)}^2+\|\partial_t\Tilde{\vec{w}}\|_{L^2(\R^d)}^2}\, \|\Ds{\alpha} \vec{u}\|_{L^{\frac{2d}{2\alpha -1}}(\R^d)}^2.
\end{align*}
\end{lemma}
\begin{proof}
It suffices to prove the estimate for $\alpha\in (1, \frac32)$. As before, by Cauchy-Schwarz and H\"older's inequality

\begin{align*}
&\abs{\int_{\mathbb{R}^n}\langle \vec{w} \wedge_LH_{|\nabla|}(\vec{u} \wedge_L,|\nabla|\vec{u}), \mathbf{X}_\vec{u}^r\rangle_L\langle\mathbf{E}_{\Tilde{\vec{u}}}^r, \Tilde{\vec{w}}_t \rangle_{\R^m}} \\
\lesssim_{\Lambda} &\|\vec{w}\|_{L^{\frac{2d}{d-2}}(\R^d)}\, \|\partial_t\Tilde{\vec{w}}\|_{L^2(\R^d)}\,
\left \|H_{\abs{\nabla}}(\vec{u} \wedge_L,\abs{\nabla} \vec{u}) \right \|_{L^{d}(\R^d)}\\
\lesssim_{\Lambda} &\brac{\|\nabla\Tilde{\vec{w}}\|_{L^2(\R^d)}^2+\|\partial_t\Tilde{\vec{w}}\|_{L^2(\R^d)}^2}\, \left \|H_{\abs{\nabla}}(\vec{u} \wedge_L,\abs{\nabla} \vec{u}) \right \|_{L^{d}(\R^d)}.
\end{align*}
In the last line we used the Sobolev embedding, \Cref{lem:sob_in}. What is left to estimate is the Leibniz rule operator term.\\
We apply \eqref{lem:prod_rule} and find that for any $\sigma \in (\frac12,1)$
\begin{align*}
\|H_{\Dso{}}(\vec{u} \wedge_L,\Dso{} \vec{u})\|_{L^{d}(\R^d)} \aleq& \|\Ds{\sigma} \vec{u}\|_{L^{\frac{2d}{2\sigma-1}}(\R^d)}\, \|\Ds{2-\sigma} \vec{u}\|_{L^{\frac{2d}{2(2-\sigma)-1}}(\R^d)}\\
\aleq& \|\Ds{2-\sigma} \vec{u}\|_{L^{\frac{2d}{2(2-\sigma)-1}}(\R^d)}^2.
\end{align*}
The last line is the Sobolev embedding. Any $\alpha\in (1,\frac32)$ is of the form $2-\sigma$ for some $\sigma \in (\frac12,1)$, so the result follows.
\end{proof}
Now we estimate \eqref{eq:diffenergy_3_2}.
\begin{lemma}\label{la:weirdHest1}
 For any $\alpha \in (1,d+\frac{1}{2})$, and any $d \geq 3$,
\begin{align*}
\|\vec{v} \wedge_L \brac{H_{\abs{\nabla}}\brac{\vec{w} \wedge_L,\abs{\nabla} \vec{u}}} \|_{L^{2}(\R^d)} \aleq \|\nabla \Tilde{\vec{w}}\|_{L^{2}(\R^d)}\, \brac{\|\Ds{\alpha} \vec{u}\|_{L^{\frac{2d}{2\alpha-1}}}^2+\|\Ds{\alpha} \vec{v}\|_{L^{\frac{2d}{2\alpha-1}}}^2}
\end{align*}
in particular
\begin{align*}
&\int_{\mathbb{R}^n} \langle \vec{v} \wedge_LH_{|\nabla|}(\vec{w} \wedge_L,|\nabla|\vec{u}), \mathbf{X}_\vec{u}^r\rangle_L\mathbf{E}_{\Tilde{\vec{u}}}^r, \Tilde{\vec{w}}_t \rangle dx \\
\lesssim_{\Lambda} & \brac{\|\Ds{\alpha} \vec{u}\|_{L^{\frac{2d}{2\alpha-1}}}^2+\|\Ds{\alpha} \vec{v}\|_{L^{\frac{2d}{2\alpha-1}}}^2}\, \brac{\|\abs{\nabla} \Tilde{\vec{w}}\|_{L^{2}(\R^d)}^2 + \|\Tilde{\vec{w}}_t\|_{L^2(\R^d)}^2}
\end{align*}
\end{lemma}
\begin{proof} We recall the formula
\begin{align*}
    \vec{a} \wedge_L(\vec{b} \wedge_L\vec{c})= \langle\vec{b}, \vec{a}\rangle_L\vec{c}  - \langle\vec{c}, \vec{a}\rangle_L\vec{b} 
\end{align*}
so 
\begin{align}
 &\vec{v} \wedge_L \brac{H_{\abs{\nabla}}\brac{\vec{w} \wedge_L,\abs{\nabla} \vec{u}}} \nonumber\\
 =& \sideset{}{'}{\sum}_{j=1}^3 \brac{\vec{v}^j H_{\abs{\nabla}}\brac{(u-v)^j ,\abs{\nabla} \vec{u}}}-\brac{\vec{v}^j H_{\abs{\nabla}}\brac{\vec{u-v} ,\abs{\nabla} \vec{u}^j}} \nonumber\\
 = & -\sideset{}{'}{\sum}_{j=1}^3 \,(v-u)^{{j}} H_{\abs{\nabla}}\brac{\vec{u-v} ,\abs{\nabla} \vec{u}^j}\label{secondterm_1}\\
 &+\frac{1}{2}{\sum_{j = 1}^3}' \,{(u-v)^j H_{\abs{\nabla}}\brac{(u-v)^j ,\abs{\nabla} \vec{u}}}\label{secondterm_2}\\
  &-\frac{1}{2}\sideset{}{'}{\sum}_{j=1}^3 \, (u+v)^j H_{\abs{\nabla}}\brac{(u-v)^j ,\abs{\nabla} \vec{u}}\label{secondterm_3}\\
  &-\sideset{}{'}{\sum}_{j=1}^3 \, u^{{j}} H_{\abs{\nabla}}\brac{\vec{u-v},\abs{\nabla} u^{j}} \label{secondterm_4}.
 \end{align}
Without loss of generality, we may assume $\alpha \in (1,\frac{3}{2})$ and set $\sigma := 2-\alpha$. Then $\sigma \in (\frac12,1)$ and by \cref{lem:prod_rule} we can estimate \eqref{secondterm_1} and \eqref{secondterm_2},

\begin{align*}
 &\|(\vec{v}-\vec{u})^j H_{\abs{\nabla}} \brac{\vec{u-v} ,\abs{\nabla} \vec{u}^j}\|_{L^{2}(\R^d)}\\
 =&\|\int_0^1 \sum_{\ell=1}^m \frac{\partial \psi^j_{\Tilde{\gamma}(s)}}{\partial \Tilde{p}^\ell}\Tilde{\vec{w}}^\ell \ dsH_{\abs{\nabla}}\brac{\vec{u-v},\abs{\nabla} \vec{u}^j}\|_{L^{2}(\R^d)}\\
 \lesssim_{\Lambda}& \|\int_0^1 \sum_{\ell=1}^m \frac{\partial \psi_{\Tilde{\gamma}(s)}}{\partial \Tilde{p}^\ell}\Tilde{\vec{w}}^\ell\|_{L^{\frac{2d}{d-2}}(\R^d)}\, \|\Ds{\sigma} \vec{w} \|_{L^{\frac{2d}{2\sigma-1}}(\R^d)} \|\Ds{2-\sigma} \vec{u}^j\|_{L^{\frac{2d}{2(2-\sigma)-1}}(\R^d)}\\
 \lesssim_{\Lambda}& \|\Tilde{\vec{w}}\|_{L^{\frac{2d}{d-2}}(\R^d)}\, \|\Ds{\sigma} \vec{w} \|_{L^{\frac{2d}{2\sigma-1}}(\R^d)} \|\Ds{2-\sigma} \vec{u}^j\|_{L^{\frac{2d}{2(2-\sigma)-1}}(\R^d)}\\
 \lesssim_{\Lambda}& \|\abs{\nabla} \Tilde{\vec{w}}\|_{L^{2}(\R^d)}\, \brac{\|\Ds{\alpha} \vec{u}\|_{L^{\frac{2d}{2\alpha-1}}}^2+\|\Ds{\alpha} \vec{v}\|_{L^{\frac{2d}{2\alpha-1}}}^2}.
 \end{align*}
For \eqref{secondterm_3}, observe that for fixed $i$, using \cref{lem:comm_int},
\begin{align*}
 &\sideset{}{'}{\sum}_{j=1}^3  \frac{1}{2}\brac{(u+v)^j H_{\abs{\nabla}}\brac{(u-v)^j ,\abs{\nabla} \vec{u}}}(x)\\
 =& c \int_{\R^d} \frac{\brac{\abs{\nabla} \vec{u}(x)-\abs{\nabla} \vec{u}(y)}\, \langle (\vec{u+v})(x),\vec{w}(x)-\vec{w}(y)\rangle_L } {|x-y|^{d+1}}dy.
\end{align*}
Setting $\vec{a} := \vec{u+v}$ and $\vec{b} := \vec{u} - \vec{v}$ we recall that $\vec{a}(x) \cdot_L \vec{b}(x) = |\vec{u}(x)|^2-|\vec{v}(x)|^2 = 0$,
\begin{align*}
 \vec{a}(x) \cdot_L (\vec{b}(x)-\vec{b}(y)) 
=&(\vec{a}(x)-\vec{a}(y))\cdot_L (\vec{b}(x)-\vec{b}(y)) -\brac{\vec{a}(x)-\vec{a}(y)}\cdot_L  \vec{b}(x).
\end{align*}
This implies that 
\begin{align*}
    &\sideset{}{'}{\sum}_{j=1}^3  \frac{1}{2}(u+v)^j H_{\Dso}\brac{(u-v)^j ,\Dso \vec{u}}(x) \\
    =&-\sideset{}{'}{\sum}_{j=1}^3  \frac{1}{2}\brac{(u-v)^j H_{\abs{\nabla}}\brac{(u+v)^j ,\Dso \vec{u}}}(x) \\
    &+ c \int_{\R^d} \frac{\brac{\Dso \vec{u}(x)-\Dso \vec{u}(y)}\, \scpr{(\vec{u+v})(x)-(\vec{u+v})(y)}{{\vec{w}(y)-\vec{w}(y)}}_L} {|x-y|^{d+1}}dy\\
    =&-\sideset{}{'}{\sum}_{j=1}^3  \frac{1}{2}\brac{(u-v)^j H_{\abs{\nabla}}\brac{(u+v)^j ,\Dso \vec{u}}}(x) \\
    &+ c \int_{\R^d} \frac{\brac{\Dso \vec{u}(x)-\Dso \vec{u}(y)}\, \scpr{(\vec{u+v})(x)-(\vec{u+v})(y)}{{\vec{w}(x)}}_L} {|x-y|^{d+1}}dy\\
    &- c \int_{\R^d} \frac{\brac{\Dso \vec{u}(x)-\Dso \vec{u}(y)}\, \scpr{(\vec{u+v})(x)-(\vec{u+v})(y)}{{\vec{w}(y)}}_L} {|x-y|^{d+1}}dy.
\end{align*}
We estimate the first term above exactly as we estimated \eqref{secondterm_1},
\begin{align*}
    &\left \| \frac{1}{2}\int_0^1 \sum_{\ell=1}^m \frac{\partial \psi^j_{\Tilde{\gamma}(s)}}{\partial \Tilde{p}^\ell} \ ds \ \Tilde{\vec{w}}^\ell H_{\abs{\nabla}}\brac{(u{+}v)^j ,\abs{\nabla} \vec{u}} \right \|_{L^{2}(\R^d)} \\
    &\lesssim_{\Lambda} \|\Dso \Tilde{\vec{w}}\|_{L^{2}(\R^d)}\, \brac{\|\Ds{\alpha} \vec{u}\|_{L^{\frac{2d}{2\alpha-1}}}^2+\|\Ds{\alpha} \vec{v}\|_{L^{\frac{2d}{2\alpha-1}}}^2}.
\end{align*}
 
For the second term we use \cref{double_comm_int}. Take a small $\sigma \in (0,1)$. Then,
\begin{align}
    &\left \|\int_{\R^d} \frac{\brac{\abs{\nabla} \vec{u}(x)-\abs{\nabla} \vec{u}(y)}\, \scpr{(\vec{u}+\vec{v})(x)-(\vec{u}+\vec{v})(y)}{\vec{w}(x)}_L} {|x-y|^{d+1}}dy\right \|_{L^2(\R^d)}\label{eq:triple_1}\\
    \aleq &\|\Ds{1+\sigma} \vec{u}\|_{L^{\frac{2d}{2(1+\sigma)-1}}}\, \|\Ds{1-\sigma} (\vec{u}+\vec{v})\|_{\frac{2d}{2(1-\sigma)-1}}\, \|\int_0^1 \sum_{\ell=1}^m \frac{\partial \psi^j_{\Tilde{\gamma}(s)}}{\partial \Tilde{p}^\ell} \ ds \ \Tilde{\vec{w}}^\ell(x)\|_{L^{\frac{2d}{d-2}}(\R^d)}\nonumber\\
    \aleq &\|\Ds{1+\sigma} \vec{u}\|_{L^{\frac{2d}{2(1+\sigma)-1}}}\, \|\Ds{1-\sigma} (\vec{u}+\vec{v})\|_{\frac{2d}{2(1-\sigma) - 1}}\, \|\Tilde{\vec{w}}^\ell\|_{L^{\frac{2d}{d-2}}(\R^d)}\nonumber\\
    \aleq&\|\Ds{1 + \sigma} \vec{u}\|_{L^{\frac{2d}{2\alpha-1}}}\, \brac{\|\abs{\nabla} \vec{u}\|_{L^{2d}(\R^d)}+\|\abs{\nabla} \vec{v}\|_{L^{2d}(\R^d)}}\, \|\Dso \Tilde{\vec{w}}\|_{L^{2}(\R^d)}\nonumber.
\end{align}

For the third term we use \cref{triple_comm_int}. Take a small $\sigma \in (0,1)$. Then,
\begin{align}
    &\left \|\int_{\R^d} \frac{\brac{\abs{\nabla} \vec{u}(x)-\abs{\nabla} \vec{u}(y)}\, \scpr{(\vec{u}+\vec{v})(x)-(\vec{u}+\vec{v})(y)}{\vec{w}(y)}_L} {|x-y|^{d+1}}dy\right \|_{L^2(\R^d)}\label{eq:triple_2}\\
    \aleq &\|\Ds{1+\sigma} \vec{u}\|_{L^{\frac{2d}{2(1+\sigma)-1}}}\, \|\Ds{1+\sigma} (\vec{u}+\vec{v})\|_{\frac{2d}{2(1+\sigma)-1}}\, \|\int_0^1 \sum_{\ell=1}^m \frac{\partial \psi_{\Tilde{\gamma}(s)}}{\partial \Tilde{p}^\ell} \ ds \ \Tilde{\vec{w}}^\ell(x)\|_{L^{\frac{2d}{d-2}}(\R^d)}\nonumber\\
    \aleq &\|\Ds{1+\sigma} \vec{u}\|_{L^{\frac{2d}{2(1+\sigma)-1}}}\, \|\Ds{1+\sigma} (\vec{u}+\vec{v})\|_{\frac{2d}{2(1+\sigma) - 1}}\, \|\Tilde{\vec{w}}^\ell\|_{L^{\frac{2d}{d-2}}(\R^d)}\nonumber\\
    \aleq & \brac{\Ds{1 + \sigma} \vec{u}\|_{L^{\frac{2d}{2(1+\sigma)-1}}}^2+\|\Ds{1 + \sigma} \vec{v}\|_{L^{\frac{2d}{2(1+\sigma)-1}}}^2}\, \|\Dso \Tilde{\vec{w}}\|_{L^{2}(\R^d)}\nonumber.
 \end{align}

To estimate \eqref{secondterm_4}, we again use \cref{lem:comm_int}
\begin{align*}
&\sideset{}{'}{\sum}_{j=1}^3 \brac{\vec{u}^{j} H_{\abs{\nabla}}\brac{\vec{u-v} ,\abs{\nabla} \vec{u}^j}}(x)\\
&=c\int_{\R^n} \frac{\brac{(\vec{u-v})(x)-(\vec{u-v})(y)}\, \left \langle \abs{\nabla} \vec{u}(x)-\abs{\nabla} \vec{u}(y), \vec{u}(x) \right \rangle_L }{|x-y|^{d+1}}\, dy.
\end{align*}
Now we write 
\begin{align*}
&\scpr{\abs{\nabla} \vec{u}(x)-\abs{\nabla} \vec{u}(y)}{\vec{u}(x)}_L\\
=&\scpr{\vec{u}(x) }{\abs{\nabla} \vec{u}(x)}_L-\scpr{\vec{u}(y) }{ \abs{\nabla} \vec{u}(y)}_L +\scpr{\vec{u}(y) -\vec{u}(x)}{\abs{\nabla} \vec{u}(y)}_L\\
=&\scpr{\vec{u}(x) }{ \abs{\nabla} \vec{u}(x)}_L-\scpr{\vec{u}(y)}{\abs{\nabla} \vec{u}(y)}_L+\scpr{\vec{u}(y) -\vec{u}(x)}{\abs{\nabla}\vec{u}(y)-\abs{\nabla}\vec{u}(x)}_L+\scpr{ \vec{u}(y)-\vec{u}(x)}{\abs{\nabla}\vec{u}(x)}_L\\
=&\scpr{\vec{u}(x) }{\abs{\nabla} \vec{u}(x)}_L-\scpr{\vec{u}(y) }{\abs{\nabla} \vec{u}(y)}_L+\scpr{\vec{u}(x) -\vec{u}(y)}{\abs{\nabla}\vec{u}(x)-\abs{\nabla}\vec{u}(y)}_L-\scpr{\vec{u}(x)-\vec{u}(y)}{\abs{\nabla}\vec{u}(x)}_L\\
\end{align*}
Thus we have for $i=1,2,3$,
\begin{equation}\label{eq:L2est:325363}
\begin{split}
&\sideset{}{'}{\sum}_{j=1}^3 {\vec{u}^{j} H_{\abs{\nabla}}\brac{\vec{w} ,\abs{\nabla} \vec{u}^j}}(x)\\
=&\,H_{\abs{\nabla}}\brac{\vec{w} , \scpr{\vec{u} }{\abs{\nabla} \vec{u}}}(x)\\
&-\sideset{}{'}{\sum}_{j=1}^3  \abs{\nabla} \vec{u}^j(x)\, H_{\abs{\nabla}}\brac{\vec{w} , \vec{u}^j}(x)\\
&+c\int_{\R^n} \frac{\brac{\vec{w}(x)-\vec{w}(y)}\, \scpr{\abs{\nabla} \vec{u}(x)-\abs{\nabla} \vec{u}(y)}{\vec{u}(x)-\vec{u}(y)}_L} {|x-y|^{d+1}}\, dy\\
=&\,H_{\abs{\nabla}}\brac{\vec{w} , \scpr{\vec{u} }{\abs{\nabla} \vec{u}}}(x)\\
&-\sideset{}{'}{\sum}_{j=1}^3  \abs{\nabla} \vec{u}^j(x)\, H_{\abs{\nabla}}\brac{\vec{w} , \vec{u}^j}(x)\\
&+c\int_{\R^n} \frac{\vec{w}(x)\, \scpr{\abs{\nabla} \vec{u}(x)-\abs{\nabla} \vec{u}(y)}{\vec{u}(x)-\vec{u}(y)}_L} {|x-y|^{d+1}}\, dy\\
&-c\int_{\R^n} \frac{\vec{w}(y)\, \scpr{\Dso \vec{u}(x)-\abs{\nabla} \vec{u}(y)}{\vec{u}(x)-\vec{u}(y)}_L}{|x-y|^{d+1}}\, dy.
\end{split}
\end{equation}
We estimate the first term in \eqref{eq:L2est:325363} by expanding using the definition of the commutator and then recombine to get
\begin{align}
    &H_{\abs{\nabla}}\brac{\vec{w} , \scpr{\vec{u} }{\abs{\nabla} \vec{u}}_L}\nonumber\\
    =& \Dso\brac{\int_0^1 \sum_{\ell=1}^m \frac{\partial \psi_{\Tilde{\gamma}(s)}}{\partial \Tilde{p}^\ell} \ ds \ \Tilde{\vec{w}}^\ell, \scpr{\vec{u} }{\abs{\nabla} \vec{u}}_L} \nonumber\\
    -& \Dso\brac{\int_0^1 \sum_{\ell=1}^m \frac{\partial \psi_{\Tilde{\gamma}(s)}}{\partial \Tilde{p}^\ell} \ ds \ \Tilde{\vec{w}}^\ell}\langle \vec{u},\Dso\vec{u}\rangle_L\nonumber\\
    -& \int_0^1 \sum_{\ell=1}^m \frac{\partial \psi_{\Tilde{\gamma}(s)}}{\partial \Tilde{p}^\ell} \ ds \ \Tilde{\vec{w}}^\ell \ \Dso{\langle \vec{u},\Dso\vec{u}\rangle_L}\nonumber\\
    =& \Dso\brac{\int_0^1 \sum_{\ell=1}^m \frac{\partial \psi_{\Tilde{\gamma}(s)}}{\partial \Tilde{p}^\ell} \ ds \ \Tilde{\vec{w}}^\ell, \scpr{\vec{u} }{\abs{\nabla} \vec{u}}_L}\nonumber\\
    -& \Dso\brac{\int_0^1 \sum_{\ell=1}^m \frac{\partial \psi_{\Tilde{\gamma}(s)}}{\partial \Tilde{p}^\ell} \ ds} \ \Tilde{\vec{w}}^\ell\langle \vec{u},\Dso\vec{u}\rangle_L\nonumber\\
    -& \int_0^1 \sum_{\ell=1}^m \frac{\partial \psi_{\Tilde{\gamma}(s)}}{\partial \Tilde{p}^\ell} \ ds \ \Dso\Tilde{\vec{w}}^\ell \ \langle \vec{u},\Dso\vec{u}\rangle_L\nonumber\\
    -& H_{\Dso}\brac{\int_0^1 \sum_{\ell=1}^m \frac{\partial \psi_{\Tilde{\gamma}(s)}}{\partial \Tilde{p}^\ell} \ ds, \ \Dso\Tilde{\vec{w}}^\ell} \ \langle \vec{u},\Dso\vec{u}\rangle_L\nonumber\\
    -& \Dso\brac{\int_0^1 \sum_{\ell=1}^m \frac{\partial \psi_{\Tilde{\gamma}(s)}}{\partial \Tilde{p}^\ell} \ ds \ \langle \vec{u},\Dso\vec{u}\rangle_L} \ \Tilde{\vec{w}}^\ell\nonumber\\
    +& \Dso\brac{\int_0^1 \sum_{\ell=1}^m \frac{\partial \psi_{\Tilde{\gamma}(s)}}{\partial \Tilde{p}^\ell} \ ds} \ \Tilde{\vec{w}}^\ell\langle \vec{u},\Dso\vec{u}\rangle_L\nonumber\\
    +& H_{\Dso}\brac{\int_0^1 \sum_{\ell=1}^m \frac{\partial \psi_{\Tilde{\gamma}(s)}}{\partial \Tilde{p}^\ell} \ ds, \ \langle \vec{u},\Dso\vec{u}\rangle_L} \ \Tilde{\vec{w}}^\ell\nonumber\\
    =& H_{\Dso}\brac{\Tilde{\vec{w}}^\ell, \int_0^1 \sum_{\ell=1}^m \frac{\partial \psi_{\Tilde{\gamma}(s)}}{\partial \Tilde{p}^\ell} \ ds \ \scpr{\vec{u} }{\abs{\nabla \vec{u}}}_L}\label{eq:comm_w_1}\\
    -& H_{\Dso}\brac{\int_0^1 \sum_{\ell=1}^m \frac{\partial \psi_{\Tilde{\gamma}(s)}}{\partial \Tilde{p}^\ell} \ ds, \ \Dso\Tilde{\vec{w}}^\ell} \ \langle \vec{u},\Dso\vec{u}\rangle_L\label{eq:comm_w_2}\\
    +& H_{\Dso}\brac{\int_0^1 \sum_{\ell=1}^m \frac{\partial \psi_{\Tilde{\gamma}(s)}}{\partial \Tilde{p}^\ell} \ ds, \ \langle \vec{u},\Dso\vec{u}\rangle_L} \ \Tilde{\vec{w}}^\ell\label{eq:comm_w_3}.
\end{align}
To estimate $\eqref{eq:comm_w_1}$ we apply \cref{lem:prod_rule}, the Sobolev embedding \cref{lem:sob_in}, and \cref{lem:prod_rule_2}.
\begin{align*}
    &\|H_{\abs{\nabla}}\brac{\Tilde{\vec{w}}^\ell , \int_0^1 \sum_{\ell=1}^m \frac{\partial \psi_{\Tilde{\gamma}(s)}}{\partial \Tilde{p}^\ell} \ ds \scpr{\vec{u}}{\abs{\nabla} \vec{u}}_L}\|_{L^2(\R^d)} \\
    \aleq& \|\Ds{1-\sigma} \Tilde{\vec{w}}\|_{L^{\frac{2d}{d-2\sigma}}(\R^d)}\, \|\Ds{\sigma} \brac{\int_0^1 \sum_{\ell=1}^m \frac{\partial \psi_{\Tilde{\gamma}(s)}}{\partial \Tilde{p}^\ell} \ ds \ \langle\vec{u},\Dso\vec{u}\rangle_L}\|_{L^{\frac{d}{\sigma}}(\R^d)}\\
    \aleq& \|\Dso \Tilde{\vec{w}}\|_{L^2(\R^d)}\, \|\Ds{\sigma} \brac{\int_0^1 \sum_{\ell=1}^m \frac{\partial \psi_{\Tilde{\gamma}(s)}}{\partial \Tilde{p}^\ell} \ ds \ \langle\vec{u},\Dso\vec{u}\rangle_L}\|_{L^{\frac{d}{\sigma}}(\R^d)}\\
    \aleq& \|\Dso \Tilde{\vec{w}}\|_{L^2(\R^d)}\, \|\Ds{\sigma} \brac{\int_0^1 \sum_{\ell=1}^m \frac{\partial \psi_{\Tilde{\gamma}(s)}}{\partial \Tilde{p}^\ell} \ ds} \ \langle\vec{u},\Dso\vec{u}\rangle_L\|_{L^{\frac{d}{\sigma}}(\R^d)}\\
    +& \|\Dso \Tilde{\vec{w}}\|_{L^2(\R^d)}\, \|\int_0^1 \sum_{\ell=1}^m \frac{\partial \psi_{\Tilde{\gamma}(s)}}{\partial \Tilde{p}^\ell} \ ds \ \Ds{\sigma}\underbrace{(\scpr{\vec{u}}{\Dso \vec{u}}_L)}_{=H_{\abs{\nabla}}(T_L(\vec{u}),\vec{u})}\|_{L^{\frac{d}{\sigma}}(\R^d)}\\
    +& \|\Dso \Tilde{\vec{w}}\|_{L^2(\R^d)}\, \|H_{\Dso^\sigma}\brac{\int_0^1 \sum_{\ell=1}^m \frac{\partial \psi_{\Tilde{\gamma}(s)}}{\partial \Tilde{p}^\ell} \ ds, \ \underbrace{(\scpr{\vec{u}}{\Dso \vec{u}}_L)}_{=H_{\abs{\nabla}}(T_L(\vec{u}),\vec{u})}}\|_{L^{\frac{d}{\sigma}}(\R^d)}\\
    \aleq&\|\Dso \Tilde{\vec{w}}\|_{L^2(\R^d)}\, \|\int_0^1\sum_{\ell=1}^m \Dso^\sigma\frac{\partial \psi_{\Tilde{\gamma}(s)}}{\partial \Tilde{p}^\ell} \ ds\|_{L^{\frac{2d}{2\sigma - 1}}(\R^d)} \ \|\Dso \vec{u}\|_{L^{2d}(\R^d)}\\
    +&\|\Dso \Tilde{\vec{w}}\|_{L^2(\R^d)}\, \|\Dso^\gamma\vec{u}\|_{L^\frac{2d}{2\gamma-1}(\R^d)} \|\Dso^{1 + \sigma - \gamma}\vec{u}\|_{L^\frac{2d}{2(1+\sigma-\gamma)-1}(\R^d)}\\
    +&\|\Dso \Tilde{\vec{w}}\|_{L^2(\R^d)}\, \|\Dso^{\sigma -\gamma_1}\vec{u}\|_{L^\frac{2d}{2(\sigma-\gamma_1)-\theta}(\R^d)} \|\Dso^{\gamma_1}\vec{u}\|_{L^\frac{2d}{2\gamma_1-1+\theta}(\R^d)} \ \|\Dso^{1 + \gamma_1 - \gamma_2}\vec{u}\|_{L^\frac{2d}{2(1+\gamma_1-\gamma_2)-1}(\R^d)}\\
    \lesssim_{\Lambda} & \|\Dso \Tilde{\vec{w}}\|_{L^2(\R^d)}\,\brac{ \|\Dso\vec{u}\|_{L^{2d}(\R^d)}^2 + \|\Dso\vec{v}\|_{L^{2d}(\R^d)}^2}.
\end{align*}
For the second term in \eqref{eq:L2est:325363}, using a  similar decomposition as above and by  \cref{lem:prod_rule}, we have
\begin{align*}
 &\|\Dso \vec{u}^j H_{\Dso}\brac{\vec{w} , \vec{u}^j}\|_{L^2(\R^d)}\\
 \aleq& \|\Dso \vec{u}\|_{L^{2d}(\R^d)} \|H_{\Dso}\brac{\int_0^1 \sum_{\ell=1}^m \frac{\partial \psi_{\Tilde{\gamma}(s)}}{\partial \Tilde{p}^\ell} \ ds \ \Tilde{\vec{w}}^\ell , \vec{u}^j}  \|_{L^{\frac{2d}{d-1}}(\R^d)}\\
 \aleq& \|\Dso \vec{u}\|_{L^{2d}(\R^d)} \|H_{\Dso}\brac{\Tilde{\vec{w}}^\ell , \int_0^1 \sum_{\ell=1}^m \frac{\partial \psi_{\Tilde{\gamma}(s)}}{\partial \Tilde{p}^\ell} \ ds \ \vec{u}^j}  \|_{L^{\frac{2d}{d-1}}(\R^d)}\\
 +& \|\Dso \vec{u}\|_{L^{2d}(\R^d)} \|H_{\Dso}\brac{\int_0^1 \sum_{\ell=1}^m \frac{\partial \psi_{\Tilde{\gamma}(s)}}{\partial \Tilde{p}^\ell} \ ds, \ \Tilde{\vec{w}}^\ell} \vec{u}^j  \|_{L^{\frac{2d}{d-1}}(\R^d)}\\
 +& \|\Dso \vec{u}\|_{L^{2d}(\R^d)} \|H_{\Dso}\brac{\int_0^1 \sum_{\ell=1}^m \frac{\partial \psi_{\Tilde{\gamma}(s)}}{\partial \Tilde{p}^\ell} \ ds,\ \vec{u}^j}\Tilde{\vec{w}}^\ell  \|_{L^{\frac{2d}{d-1}}(\R^d)}\\
 \lesssim_{\Lambda}& \|\Dso \vec{u}\|_{L^{2d}(\R^d)} \|\Ds{1-\sigma}\Tilde{\vec{w}}\|_{L^{\frac{2d}{d-2\sigma}}(\R^d)} \|\Ds{\sigma} \vec{u}\|_{L^{\frac{2d}{2\sigma - 1}(\R^d)}}\\
 +& \|\Dso \vec{u}\|_{L^{2d}(\R^d)} \|\Tilde{\vec{w}}\|_{L^{\frac{2d}{d-2}}(\R^d)} \|\Ds{\sigma} \vec{u}\|_{L^{\frac{2d}{\sigma }(\R^d)}} \ \|\Ds{1-\sigma} \vec{u}\|_{L^{\frac{2d}{1-\sigma }(\R^d)}}\\
 \lesssim_{\Lambda} & |\Dso \vec{u}\|_{L^{2d}(\R^d)}^2 \|\Dso\Tilde{\vec{w}}\|_{L^2(\R^d)}.
\end{align*}
For the third and forth term in \eqref{eq:L2est:325363}, we estimate as we did for $\eqref{eq:triple_1}$ and $\eqref{eq:triple_2}$,
\begin{align*}
  &\left \|\int_{\R^n} \frac{\vec{w}(x)\, \brac{\abs{\nabla} \vec{u}(x)-\abs{\nabla} \vec{u}(y)} \cdot_L \brac{\vec{u}(x)-\vec{u}(y)}}{|x-y|^{d+1}}\, dy\right \|_{L^2(\R^d,dx)}\\
  \aleq& \|\Ds{1 + \sigma} \vec{u}\|_{L^{\frac{2d}{2\alpha-1}}}^2\, \|\Dso \Tilde{\vec{w}}\|_{L^{2}(\R^d)},
\end{align*}
and 
\begin{align*}
  &\left \|\int_{\R^n} \frac{\vec{w}(y)\, \brac{\abs{\nabla} \vec{u}(x)-\abs{\nabla} \vec{u}(y)} \cdot_L \brac{\vec{u}(x)-\vec{u}(y)}}{|x-y|^{d+1}}\, dy\right \|_{L^2(\R^d,dx)}\\
  \aleq& \|\Ds{1 + \sigma} \vec{u}\|_{L^{\frac{2d}{2\alpha-1}}}^2\, \|\Dso \Tilde{\vec{w}}\|_{L^{2}(\R^d)},
\end{align*}
This concludes the estimate of \eqref{secondterm_4}.
\end{proof}
To estimate \eqref{eq:diffenergy_3_3}, we need the next two following results.
\begin{lemma}\label{prod_frac_vec}
    Let $d \ge 3$, $\vec{u},\vec{v} \in \mathbb{H}^2$. Then for $\sigma \in (0,1/2]$, we have
\begin{align*}
    &\lVert H_{\Dso}\brac{\vec{v}, \langle \vec{v},\Dso\vec{w} \rangle_L} \rVert_{L^{2}(\R^d)}\\
    \lesssim_{\Lambda} & \brac{\Dso^{1+\sigma}\vec{u} \rVert_{L^{\frac{2d}{2(1+\sigma) - 1}}(\R^d)}^2 + \lVert \Dso^{1+\sigma}\vec{v} \rVert_{L^{\frac{2d}{2(1+\sigma) - 1}}(\R^d)}^2}  \ \lVert \Dso \Tilde{\vec{w}} \rVert_{L^{2}(\R^d)}.
\end{align*}
\end{lemma}
\begin{proof}
Since  $|\vec{u}|^2_L = |\vec{v}|^2_L = -1$,
\begin{align*}
 \langle\vec{u}+\vec{v},\vec{u}-\vec{v}\rangle_L= |\vec{u}|_L^2-|\vec{v}|_L^2 = 0.
\end{align*}
Consequently,
\begin{align*}
 (\vec{u}+\vec{v}) \cdot_L \abs{\nabla} (\vec{u}-\vec{v}) = -(\abs{\nabla} (\vec{u}+\vec{v}) )\cdot_L (\vec{u}-\vec{v}) -H_{\abs{\nabla}} \brac{(\vec{u}+\vec{v})\cdot_L,\vec{u}-\vec{v}},
\end{align*}
That is
\begin{align*}
 \vec{v} \cdot_L \abs{\nabla} (\vec{u}-\vec{v}) =& -\frac{1}{2} (\vec{u}-\vec{v}) \cdot_L \abs{\nabla}(\vec{u}-\vec{v}) + \frac{1}{2} (\vec{u}+\vec{v}) \cdot_L \abs{\nabla}(\vec{u}-\vec{v}) \\
 =&-\frac{1}{2} (\vec{u}-\vec{v}) \cdot_L \abs{\nabla}(\vec{u}-\vec{v}) - \frac{1}{2} \brac{\abs{\nabla} (\vec{u}+\vec{v}) \cdot_L (\vec{u}-\vec{v}) +H_{\abs{\nabla}} ((\vec{u}+\vec{v})\cdot_L,\vec{u}-\vec{v})} \\
\end{align*}
Using \cref{comm_lem}, we get
\begin{align}
    &\lVert H_{\Dso}\brac{\vec{v}, \langle \vec{w},\Dso\vec{u} \rangle_L} \rVert_{L^{2}(\R^d)}\nonumber\\
    &= \sideset{}{'}{\sum}_{j=1}^3 \lVert \sum_{\ell=1}^m H_{\Dso}\brac{\vec{v}^i,\int_0^1 \frac{\partial \psi^j_{\Tilde{\gamma}(s)}}{\partial p^\ell} \ ds \tilde{\vec{w}}^\ell\Dso\vec{u}^j} \rVert_{L^{2}(\R^d)}\nonumber\\
    \lesssim & \sideset{}{'}{\sum}_{j=1}^3 \sum_{\ell=1}^m \lVert  H_{\Dso}\brac{\vec{v}^i\int_0^1 \frac{\partial \psi^j_{\Tilde{\gamma}(s)}}{\partial p^\ell} \ ds, \Tilde{\vec{w}}^\ell\Dso\vec{u}^j} \rVert_{L^{2}(\R^d)}\label{prod_frac_vec_1}\\
    &+ \sideset{}{'}{\sum}_{j=1}^3 \sum_{\ell=1}^m \lVert  H_{\Dso}\brac{\vec{v}^i,\int_0^1 \frac{\partial \psi^j_{\Tilde{\gamma}(s)}}{\partial p^\ell} \ ds} \tilde{\vec{w}}^\ell\Dso\vec{u}^j \rVert_{L^{2}(\R^d)}\label{prod_frac_vec_2}\\
    &+ \sideset{}{'}{\sum}_{j=1}^3 \sum_{\ell=1}^m \lVert \vec{v}^i H_{\Dso}\brac{\int_0^1 \frac{\partial \psi^j_{\Tilde{\gamma}(s)}}{\partial p^\ell} \ ds, \Tilde{\vec{w}}^\ell\Dso\vec{u}^j} \rVert_{L^{2}(\R^d)}\label{prod_frac_vec_3}.
\end{align}
For \eqref{prod_frac_vec_1}, we use \cref{lem:prod_rule} with small $\sigma \in (0,\frac{1}{2}]$ and \cref{lem:prod_rule_frac} to get
\begin{align*}
    &\lVert  H_{\Dso}\brac{\vec{v}^i\int_0^1 \frac{\partial \psi^j_{\Tilde{\gamma}(s)}}{\partial p^\ell} \ ds, \tilde{\vec{w}}^\ell\Dso\vec{u}^j} \rVert_{L^{2}(\R^d)}\\
    \lesssim & \lVert \Dso^{1-\sigma}\brac{\vec{v}^i\int_0^1 \frac{\partial \psi^j_{\Tilde{\gamma}(s)}}{\partial p^\ell} \ ds} \rVert_{L^{\frac{2d}{2(1-\sigma) - 1}}(\R^d)} \ \lVert \Dso^{\sigma}\brac{\tilde{\vec{w}}^\ell\Dso\vec{u}^j} \rVert_{L^{\frac{2d}{d + 2\sigma - 1}}(\R^d)}\\
    \lesssim_{\Lambda} & \brac{\Dso^{1+\sigma}\vec{u} \rVert_{L^{\frac{2d}{2(1+\sigma) - 1}}(\R^d)}^2 + \lVert \Dso^{1+\sigma}\vec{v} \rVert_{L^{\frac{2d}{2(1+\sigma) - 1}}(\R^d)}^2}  \ \lVert \Dso\Tilde{\vec{w}} \rVert_{L^{2}(\R^d)}.
\end{align*}

For \eqref{prod_frac_vec_2} we use \cref{lem:prod_rule} and \cref{lem:gag_in}
\begin{align*}
    &\lVert  H_{\Dso}\brac{\vec{v}^i,\int_0^1 \frac{\partial \psi^j_{\Tilde{\gamma}(s)}}{\partial p^\ell} \ ds} \tilde{\vec{w}}^\ell\Dso\vec{u}^j \rVert_{L^{2}(\R^d)}\\
    \lesssim &\lVert  H_{\Dso}\brac{\vec{v}^i,\int_0^1 \frac{\partial \psi^j_{\Tilde{\gamma}(s)}}{\partial p^\ell} \ ds} \rVert_{L^{2d}(\R^d)} \ \lVert \Tilde{\vec{w}} \rVert_{L^{\frac{2d}{d-2}}(\R^d)} \ \lVert \Dso\vec{u} \rVert_{L^{2d}(\R^d)}\\
    \lesssim &\lVert \Dso^\sigma\vec{v} \rVert_{L^{\frac{2d}{\sigma}}(\R^d)} \lVert \int_0^1\Dso^{1-\sigma} \frac{\partial \psi^j_{\Tilde{\gamma}(s)}}{\partial p^\ell} \ ds \rVert_{L^{\frac{2d}{1-\sigma}}(\R^d)} \ \lVert \Dso\Tilde{\vec{w}} \rVert_{L^{2}(\R^d)} \ \lVert \Dso\vec{u} \rVert_{L^{2d}(\R^d)}\\
    \lesssim &\lVert \Dso\vec{v} \rVert_{L^{2d}(\R^d)}^\sigma \lVert  \int_0^1 \Dso d\psi_{\Tilde{\gamma}(s)} \ ds \rVert_{L^{2d}(\R^d)}^{1-\sigma} \ \lVert \Dso\Tilde{\vec{w}} \rVert_{L^{2}(\R^d)} \ \lVert \Dso\vec{u} \rVert_{L^{2d}(\R^d)}\\
    \lesssim_{\Lambda} & \brac{\lVert \Dso\vec{u} \rVert_{L^{2d}(\R^d)}^2 + \lVert  \Dso \vec{v} \rVert_{L^{2d}(\R^d)}^2} \ \lVert \Dso\Tilde{\vec{w}} \rVert_{L^{2}(\R^d)}.
\end{align*}
Finally, for \eqref{prod_frac_vec_3}, taking small $\sigma\in (0,\frac{1}{2}]$, we use \cref{lem:prod_rule} and \cref{lem:prod_rule_frac} to get
\begin{align*}
    &\lVert \vec{v}^i H_{\Dso}\brac{\int_0^1 \frac{\partial \psi^j_{\Tilde{\gamma}(s)}}{\partial p^\ell} \ ds, \tilde{\vec{w}}^\ell\Dso\vec{u}^j} \rVert_{L^{2}(\R^d)}\\
    \lesssim & \lVert \Dso^{1-\sigma} \int_0^1 \frac{\partial \psi^j_{\Tilde{\gamma}(s)}}{\partial p^\ell} \ ds \rVert_{L^{\frac{2d}{2(1-\sigma) -1}}(\R^d)} \ \lVert \Dso^\sigma \brac{\Tilde{\vec{w}}^\ell \Dso\vec{u}^j }\rVert_{L^{\frac{2d}{d + 2\sigma - 1}}(\R^d)}\\
    \lesssim_{\Lambda} & \brac{\Dso^{1+\sigma}\vec{u} \rVert_{L^{\frac{2d}{2(1+\sigma) - 1}}(\R^d)}^2 + \lVert \Dso^{1+\sigma}\vec{v} \rVert_{L^{\frac{2d}{2(1+\sigma) - 1}}(\R^d)}^2}  \ \lVert \Dso\Tilde{\vec{w}} \rVert_{L^{2}(\R^d)}.
\end{align*}
Hence,
\begin{align*}
    &\lVert H_{\Dso}\brac{\vec{v}, \langle \vec{w},\Dso\vec{u} \rangle_L} \rVert_{L^{2}(\R^d)}\\
    \lesssim_{\Lambda} & \brac{\Dso^{1+\sigma}\vec{u} \rVert_{L^{\frac{2d}{2(1+\sigma) - 1}}(\R^d)}^2 + \lVert \Dso^{1+\sigma}\vec{v} \rVert_{L^{\frac{2d}{2(1+\sigma) - 1}}(\R^d)}^2}  \ \lVert \Dso\Tilde{\vec{w}} \rVert_{L^{2}(\R^d)}.
\end{align*}
Similarly,
\begin{align*}
    &\lVert H_{\Dso}\brac{\vec{v}, \langle \vec{w},\Dso\vec{v} \rangle_L} \rVert_{L^{2}(\R^d)}\\
    \lesssim_{\Lambda} & \brac{\Dso^{1+\sigma}\vec{u} \rVert_{L^{\frac{2d}{2(1+\sigma) - 1}}(\R^d)}^2 + \lVert \Dso^{1+\sigma}\vec{v} \rVert_{L^{\frac{2d}{2(1+\sigma) - 1}}(\R^d)}^2}  \ \lVert \Dso\Tilde{\vec{w}} \rVert_{L^{2}(\R^d)}.
\end{align*}
and
\begin{align*}
    &\lVert H_{\Dso}\brac{\vec{v}, \langle \Dso\brac{\vec{u} + \vec{v}},\vec{w} \rangle_L} \rVert_{L^{2}(\R^d)}\\
    \lesssim_{\Lambda} & \brac{\Dso^{1+\sigma}\vec{u} \rVert_{L^{\frac{2d}{2(1+\sigma) - 1}}(\R^d)}^2 + \lVert \Dso^{1+\sigma}\vec{v} \rVert_{L^{\frac{2d}{2(1+\sigma) - 1}}(\R^d)}^2}  \ \lVert \Dso\Tilde{\vec{w}} \rVert_{L^{2}(\R^d)}.
\end{align*}
Let $\sigma \in (0, \frac{1}{2}]$ be small, $\gamma \in (\sigma,1)$, and $\alpha \in (\gamma - \sigma, \gamma)$. Then for the last term, we use \cref{def_comm} to get
\begin{align}
    &\lVert H_{\Dso}\brac{\vec{v}, H_{\abs{\nabla}} ((\vec{u}+\vec{v})\cdot_L,\vec{u}-\vec{v})} \rVert_{L^{2}(\R^d)}\nonumber\\
    =& \lVert H_{\Dso}\brac{\vec{v}, H_{\abs{\nabla}} \left((\vec{u}+\vec{v})\cdot_L,\Dso^{-\gamma}\Dso^\gamma\brac{\int_0^1 d\psi_{\gamma(s)}\Tilde{\vec{w}} \ ds}\right)} \rVert_{L^{2}(\R^d)}\nonumber\\
    \lesssim & \lVert H_{\Dso}\brac{\vec{v}, H_{\abs{\nabla}} \left((\vec{u}+\vec{v})\cdot_L,\Dso^{-\gamma}\brac{\int_0^1 \Dso^\gamma d\psi_{\gamma(s)}\Tilde{\vec{w}} \ ds}\right)} \rVert_{L^{2}(\R^d)} \label{main_1}\\
    + & \lVert H_{\Dso}\brac{\vec{v}, H_{\abs{\nabla}} \left((\vec{u}+\vec{v})\cdot_L,\Dso^{-\gamma}\brac{\int_0^1 d\psi_{\gamma(s)} \ \Dso^\gamma \Tilde{\vec{w}} \ ds }\right)} \rVert_{L^{2}(\R^d)} \label{main_2}\\
    + & \lVert H_{\Dso}\brac{\vec{v}, H_{\abs{\nabla}} \left((\vec{u}+\vec{v})\cdot_L,\Dso^{-\gamma}\brac{H_{\Dso^\gamma}\brac{\int_0^1 d\psi_{\gamma(s)} \ ds, \Tilde{\vec{w}}}}\right)} \rVert_{L^{2}(\R^d)} \label{main_3}.
\end{align}
For \eqref{main_1}
\begin{align*}
    & \lVert H_{\Dso}\brac{\vec{v}, H_{\abs{\nabla}} \left((\vec{u}+\vec{v})\cdot_L,\Dso^{-\gamma}\brac{\int_0^1 \Dso^\gamma d\psi_{\gamma(s)}\Tilde{\vec{w}} \ ds}\right)} \rVert_{L^{2}(\R^d)}\\
    &\overset{\cref{lem:prod_rule}}{\lesssim} \lVert \Dso^{1-\sigma}\vec{v} \rVert_{L^{\frac{2d}{1-\sigma}}(\R^d)} \ \lVert \Dso^\sigma H_{\abs{\nabla}} \left((\vec{u}+\vec{v})\cdot_L,\Dso^{-\gamma}\brac{\int_0^1 \Dso^\gamma d\psi_{\gamma(s)}\Tilde{\vec{w}} \ ds}\right) \rVert_{L^{\frac{2d}{d + 2\sigma - 1}}(\R^d)}\\
    &\overset{\cref{lem:prod_rule_2}}{\lesssim} \lVert \Dso^{1-\sigma}\vec{v} \rVert_{L^{\frac{2d}{1-\sigma}}(\R^d)} \ \lVert \Dso^{1+\sigma-\gamma}\brac{\vec{u}+\vec{v}} \rVert_{L^{\frac{2d}{1+\sigma -\gamma}}(\R^d)} \ \lVert \int_0^1 \Dso^\gamma d\psi_{\gamma(s)}\Tilde{\vec{w}} \ ds \rVert_{L^{\frac{2d}{d + \gamma - 2}}(\R^d)}\\
    &\overset{\cref{lem:gag_in}}{\lesssim} \lVert \Dso^{1-\sigma}\vec{v} \rVert_{L^{\frac{2d}{1-\sigma}}(\R^d)} \ \lVert \Dso^{1+\sigma-\gamma}\brac{\vec{u}+\vec{v}} \rVert_{L^{\frac{2d}{1+\sigma -\gamma}}(\R^d)} \ \lVert \int_0^1 \Dso^\gamma d\psi_{\gamma(s)} \ ds \rVert_{L^{\frac{2d}{\gamma}}(\R^d)} \ \lVert \Tilde{\vec{w}} \rVert_{L^{\frac{2d}{d - 2}}(\R^d)}\\
    \lesssim & \lVert \Dso\vec{v} \rVert_{L^{2d}(\R^d)}^{1-\sigma} \ \lVert \Dso\brac{\vec{u}+\vec{v}} \rVert_{L^{2d}(\R^d)}^{1+\sigma-\gamma} \ \lVert \int_0^1 \Dso d\psi_{\gamma(s)} \ ds \rVert_{L^{2d}(\R^d)}^\gamma \ \lVert \Tilde{\vec{w}} \rVert_{L^{\frac{2d}{d - 2}}(\R^d)}\\
    &\overset{\cref{comm_lem}}{\lesssim_{\Lambda}} \lVert \Dso\vec{v} \rVert_{L^{2d}(\R^d)}^{1-\sigma} \ \lVert \Dso\brac{\vec{u}+\vec{v}} \rVert_{L^{2d}(\R^d)}^{1+\sigma-\gamma} \ \lVert \int_0^1 \Dso d\psi_{\gamma(s)} \ ds \rVert_{L^{2d}(\R^d)}^\gamma \ \lVert \Dso\Tilde{\vec{w}} \rVert_{L^{2d}(\R^d)}\\
    \lesssim_{\Lambda} & \lVert \Dso\vec{v} \rVert_{L^{2d}(\R^d)}^{1-\sigma} \ \brac{\lVert \Dso\vec{u} \rVert_{L^{2d}(\R^d)} + \lVert \Dso\vec{v} \rVert_{L^{2d}(\R^d)}}^{1+\sigma-\gamma} \ \brac{\lVert \Dso\vec{u} \rVert_{L^{2d}(\R^d)} + \lVert \Dso\vec{v} \rVert_{L^{2d}(\R^d)}}^\gamma \ \lVert \Dso\Tilde{\vec{w}} \rVert_{L^{2d}(\R^d)}\\
    \lesssim_{\Lambda} & \brac{\lVert \Dso\vec{u} \rVert_{L^{2d}(\R^d)} + \lVert \Dso\vec{v} \rVert_{L^{2d}(\R^d)}}^2 \ \lVert \Dso\Tilde{\vec{w}} \rVert_{L^{2d}(\R^d)}\\
    \lesssim_{\Lambda} & \brac{\lVert \Dso\vec{u} \rVert_{L^{2d}(\R^d)}^2 + \lVert \Dso\vec{v} \rVert_{L^{2d}(\R^d)}^2} \ \lVert \Dso\Tilde{\vec{w}} \rVert_{L^2(\R^d)}.
\end{align*}

For \eqref{main_2}
\begin{align*}
    & \lVert H_{\Dso}\brac{\vec{v}, H_{\abs{\nabla}} \left((\vec{u}+\vec{v})\cdot_L,\Dso^{-\gamma}\brac{\int_0^1 d\psi_{\gamma(s)} \Dso^\gamma\Tilde{\vec{w}} \ ds}\right)} \rVert_{L^{2}(\R^d)}\\
    \overset{\cref{lem:prod_rule}}{\lesssim} &\lVert \Dso^{1-\sigma}\vec{v} \rVert_{L^{\frac{2d}{2(1-\sigma) -1}}(\R^d)} \ \lVert \Dso^\sigma H_{\abs{\nabla}} \left((\vec{u}+\vec{v})\cdot_L,\Dso^{-\gamma}\brac{\int_0^1 d\psi_{\gamma(s)} \Dso^\gamma\Tilde{\vec{w}} \ ds}\right) \rVert_{L^{\frac{2d}{d + 2\sigma - 1}}(\R^d)}\\
    \overset{\cref{lem:prod_rule_2}}{\lesssim} & \lVert \Dso^{1-\sigma}\vec{v} \rVert_{L^{\frac{2d}{2(1-\sigma) -1}}(\R^d)} \ \lVert \Dso^{1+\sigma-\gamma}\brac{\vec{u}+\vec{v}} \rVert_{L^{\frac{2d}{2(1+\sigma -\gamma) - 1}}(\R^d)} \ \lVert \int_0^1 d\psi_{\gamma(s)} \Dso^\gamma\Tilde{\vec{w}} \ ds \rVert_{L^{\frac{2d}{d + \gamma - 2}}(\R^d)}\\
    \lesssim_{\Lambda} & \lVert \Dso^{1-\sigma}\vec{v} \rVert_{L^{\frac{2d}{2(1-\sigma) -1}}(\R^d)} \ \lVert \Dso^{1+\sigma-\gamma}\brac{\vec{u}+\vec{v}} \rVert_{L^{\frac{2d}{2(1+\sigma -\gamma) - 1}}(\R^d)} \ \lVert \Dso^\gamma\Tilde{\vec{w}} \rVert_{L^{\frac{2d}{d - 2(1-\gamma)}}(\R^d)}\\
    \lesssim_{\Lambda} & \brac{\lVert \Dso\vec{u} \rVert_{L^{2d}(\R^d)} + \lVert \Dso\vec{v} \rVert_{L^{2d}(\R^d)}}^2 \ \lVert \Dso\Tilde{\vec{w}} \rVert_{L^{2d}(\R^d)}\\
    \lesssim_{\Lambda} & \brac{\lVert \Dso\vec{u} \rVert_{L^{2d}(\R^d)}^2 + \lVert \Dso\vec{v} \rVert_{L^{2d}(\R^d)}^2} \ \lVert \Dso\Tilde{\vec{w}} \rVert_{L^2(\R^d)}.
\end{align*}
For \eqref{main_2}
\begin{align*}
    & \lVert H_{\Dso}\brac{\vec{v}, H_{\abs{\nabla}} \left((\vec{u}+\vec{v})\cdot_L,\Dso^{-\gamma}H_{\Dso^\gamma}\brac{\int_0^1 d\psi_{\gamma(s)} \ ds, \Tilde{\vec{w}}}\right)} \rVert_{L^{2}(\R^d)}\\
    &\overset{\cref{lem:prod_rule}}{\lesssim} \lVert \Dso^{1-\sigma}\vec{v} \rVert_{L^{\frac{2d}{1-\sigma}}(\R^d)} \ \lVert \Dso^\sigma H_{\abs{\nabla}} \left((\vec{u}+\vec{v})\cdot_L,\Dso^{-\gamma} H_{\Dso^\gamma}\brac{\int_0^1 d\psi_{\gamma(s)} \ ds, \Tilde{\vec{w}}}\right) \rVert_{L^{\frac{2d}{d + 2\sigma - 1}}(\R^d)}\\
    & \overset{\cref{lem:prod_rule_2}}{\lesssim} \lVert \Dso^{1-\sigma}\vec{v} \rVert_{L^{\frac{2d}{1-\sigma}}(\R^d)} \ \lVert \Dso^{1+\sigma-\gamma}\brac{\vec{u}+\vec{v}} \rVert_{L^{\frac{2d}{2(1+\sigma -\gamma) - 1}}(\R^d)} \ \lVert H_{\Dso^\gamma}\brac{\int_0^1 d\psi_{\gamma(s)} \ ds, \Tilde{\vec{w}}} \rVert_{L^{\frac{2d}{d + \gamma - 2}}(\R^d)}\\
    \lesssim & \lVert \Dso^{1-\sigma}\vec{v} \rVert_{L^{\frac{2d}{1-\sigma}}(\R^d)} \ \lVert \Dso^{1+\sigma-\gamma}\brac{\vec{u}+\vec{v}} \rVert_{L^{\frac{2d}{2(1+\sigma -\gamma) - 1}}(\R^d)} \ \lVert \int_0^1 \Dso^{\gamma - \alpha} d\psi_{\gamma(s)} \ ds \rVert_{L^{\frac{2d}{2(\gamma - \alpha) - \sigma}}(\R^d)} \ \lVert \Dso^\alpha\Tilde{\vec{w}} \rVert_{L^{\frac{2d}{d - 2(1-\alpha)}}(\R^d)}\\
    \lesssim & \lVert \Dso^{1-\sigma}\vec{v} \rVert_{L^{\frac{2d}{1-\sigma}}(\R^d)} \ \lVert \Dso^{1+\sigma-\gamma}\brac{\vec{u}+\vec{v}} \rVert_{L^{\frac{2d}{2(1+\sigma -\gamma) - 1}}(\R^d)} \ \lVert \int_0^1 \Dso^{\gamma - \alpha} d\psi_{\gamma(s)} \ ds \rVert_{L^{\frac{2d}{2(\gamma - \alpha) - \sigma}}(\R^d)} \ \lVert \Dso\Tilde{\vec{w}} \rVert_{L^2(\R^d)}\\
    & \overset{\cref{lem:gag_in}}{\lesssim} \lVert \Dso\vec{v} \rVert_{L^{2d}(\R^d)}^{1-\sigma} \ \lVert \Dso^{1+\sigma-\gamma}\brac{\vec{u}+\vec{v}} \rVert_{L^{\frac{2d}{2(1+\sigma -\gamma) - 1}}(\R^d)} \ \lVert \int_0^1 \Dso d\psi_{\gamma(s)} \ ds \rVert_{L^{2d}(\R^d)}^{\sigma} \ \lVert \Dso\Tilde{\vec{w}} \rVert_{L^2(\R^d)}\\
    & \overset{\cref{comm_lem}}{\lesssim_{\Lambda}} \lVert \Dso\vec{v} \rVert_{L^{2d}(\R^d)}^{1-\sigma} \ \lVert \Dso^{1+\sigma-\gamma}\brac{\vec{u}+\vec{v}} \rVert_{L^{\frac{2d}{2(1+\sigma -\gamma) - 1}}(\R^d)} \ \brac{\lVert \Dso\vec{u} \rVert_{L^{2d}(\R^d)} + \lVert \Dso\vec{v} \rVert_{L^{2d}(\R^d)}}^{\sigma} \ \lVert \Dso\Tilde{\vec{w}} \rVert_{L^2(\R^d)}\\
    \lesssim_{\Lambda} & \lVert \Dso^{1+\sigma-\gamma}\brac{\vec{u}+\vec{v}} \rVert_{L^{\frac{2d}{2(1+\sigma -\gamma) - 1}}(\R^d)} \ \brac{\lVert \Dso\vec{u} \rVert_{L^{2d}(\R^d)} + \lVert \Dso\vec{v} \rVert_{L^{2d}(\R^d)}} \ \lVert \Dso\Tilde{\vec{w}} \rVert_{L^2(\R^d)}\\
    \lesssim_{\Lambda} & \brac{\lVert \Dso\vec{u} \rVert_{L^{2d}(\R^d)}^2 + \lVert \Dso\vec{v} \rVert_{L^{2d}(\R^d)}^2} \ \lVert \Dso\Tilde{\vec{w}} \rVert_{L^2(\R^d)}.
\end{align*}
Therefore,
\begin{align*}
    &\lVert H_{\Dso}\brac{\vec{v}, \langle \vec{v},\Dso\vec{w} \rangle_L} \rVert_{L^{2}(\R^d)}\\
    \lesssim_{\Lambda} & \brac{\Dso^{1+\sigma}\vec{u} \rVert_{L^{\frac{2d}{2(1+\sigma) - 1}}(\R^d)}^2 + \lVert \Dso^{1+\sigma}\vec{v} \rVert_{L^{\frac{2d}{2(1+\sigma) - 1}}(\R^d)}^2}  \ \lVert \Dso\Tilde{\vec{w}} \rVert_{L^{2}(\R^d)}.
\end{align*}
\end{proof}

\begin{lemma}\label{comm_lemm}
Let $d\geq 3$. Then for $\sigma \in (0,1/2]$,
\begin{align*}
&\|H_{\Dso}\brac{\vec{v} \wedge_L, \langle \Dso \vec{u}, \vec{v} \rangle_L \vec{w}} \|_{L^{2}(\R^d)} \\
&\lesssim_{\Lambda} \|\nabla \Tilde{\vec{w}}\|_{L^2(\R^d)}\, \brac{\|\Ds{1+\sigma} \vec{u}\|_{L^{\frac{2d}{2(1+\sigma)-1}}(\R^d)}^2 + \|\Ds{1+\sigma} \vec{v}\|_{L^{\frac{2d}{2(1+\sigma)-1}}(\R^d)}^2},
\end{align*}
and 
\begin{align*}
&\|H_{\Dso}\brac{\vec{v} \wedge_L, \langle \vec{v},\vec{w} \rangle_L \Dso \vec{u}}  \|_{L^{2}(\R^d)} \\
&\lesssim_{\Lambda} \|\nabla \Tilde{\vec{w}}\|_{L^2(\R^d)}\, \brac{\|\Ds{1+\sigma} \vec{u}\|_{L^{\frac{2d}{2(1+\sigma)-1}}(\R^d)}^2 + \|\Ds{1+\sigma} \vec{v}\|_{L^{\frac{2d}{2(1+\sigma)-1}}(\R^d)}^2}.
\end{align*}
\end{lemma}
\begin{proof}
From small \eqref{w_eq}, we see that
\begin{align*}
    &H_{\Dso}\brac{\vec{v} \wedge_L, \langle \vec{v},\vec{w} \rangle_L \Dso \vec{u}} \\
    =& \sum_{\ell=1}^m \sideset{}{'}{\sum}_{j=1}^3  H_{\Dso}\brac{\vec{v} \wedge_L, \vec{v}^j \int_0^1 \frac{\partial \psi^j_{\Tilde{\gamma}(s)}}{\partial p^\ell} \Tilde{\vec{w}}^\ell \ ds \Dso \vec{u}}.
\end{align*}
Using \cref{comm_lem},
\begin{align}
    & \lVert H_{\Dso}\brac{\vec{v} \wedge_L, \vec{v}^j \int_0^1\frac{\partial \psi^j_{\Tilde{\gamma}(s)}}{\partial p^\ell} \Tilde{\vec{w}}^\ell \ ds \Dso \vec{u}}\rVert_{L^{2}(\R^d)}\nonumber\\
    \lesssim & \lVert H_{\Dso}\brac{\vec{v} \vec{v}^j \int_0^1 \frac{\partial \psi^j_{\Tilde{\gamma}(s)}}{\partial p^\ell} \ ds \wedge_L, \Tilde{\vec{w}}^\ell \Dso \vec{u} }\rVert_{L^{2}(\R^d)} \label{comm_lem_1}\\
    + & \lVert H_{\Dso}\brac{\vec{v}, \vec{v}^j \int_0^1 \frac{\partial \psi^j_{\Tilde{\gamma}(s)}}{\partial p^\ell} \ ds } \wedge_L\Tilde{\vec{w}}^\ell \Dso \vec{u} \rVert_{L^{2}(\R^d)}\label{comm_lem_2}\\
    +& \lVert \vec{v} \wedge_L H_{\Dso}\brac{\vec{v}^j \int_0^1 \frac{\partial \psi^j_{\Tilde{\gamma}(s)}}{\partial p^\ell} \ ds, \Tilde{\vec{w}}^\ell \Dso \vec{u}} \rVert_{L^{2}(\R^d)}\label{comm_lem_3}.
\end{align}
For $\sigma \in (0,\frac{1}{2})$
\begin{align*}
    &\Dso^{1-\sigma} \brac{\vec{v} \vec{v}^j \int_0^1 \frac{\partial \psi^j_{\Tilde{\gamma}(s)}}{\partial p^\ell} \ ds}\\
    &= \Dso^{1-\sigma} \vec{v} \vec{v}^j \int_0^1 \frac{\partial \psi^j_{\Tilde{\gamma}(s)}}{\partial p^\ell} \ ds\\
    &+ \vec{v} \Dso^{1-\sigma}\vec{v}^j \int_0^1 \frac{\partial \psi^j_{\Tilde{\gamma}(s)}}{\partial p^\ell} \ ds\\
    &+ \vec{v} \vec{v}^j \int_0^1 \Dso^{1-\sigma}\frac{\partial \psi^j_{\Tilde{\gamma}(s)}}{\partial p^\ell} \ ds\\
    &+ H_{\Dso^{1-\sigma}}\brac{\vec{v} \vec{v}^j, \int_0^1 \frac{\partial \psi^j_{\Tilde{\gamma}(s)}}{\partial p^\ell} \ ds}\\
    &+ H_{\Dso^{1-\sigma}}\brac{\vec{v}, \vec{v}^j} \int_0^1 \frac{\partial \psi^j_{\Tilde{\gamma}(s)}}{\partial p^\ell} \ ds.
\end{align*}
So by \cref{lem:prod_rule} for a small $\sigma \in (0,\frac{1}{2})$ and $\gamma \in (\frac{1}{4},\frac{1}{2})$, \cref{lem:gag_in}, and \cref{isom_est}, we have
\begin{align*}
    &\lVert \Dso^{1-\sigma}\brac{\vec{v} \vec{v}^j \int_0^1 \frac{\partial \psi^j_{\Tilde{\gamma}(s)}}{\partial p^\ell} \ ds} \rVert_{L^{\frac{2d}{2(1-\sigma)-1}}(\R^d)}\\
    &\lesssim_{\Lambda} \brac{\lVert \Dso \vec{u} \rVert_{L^{2d}(\R^d)} + \lVert \Dso \vec{v} \rVert_{L^{2d}(\R^d)}}\\
    &+ \lVert H_{\Dso^{1-\sigma}}\brac{\vec{v} \vec{v}^j, \int_0^1 \frac{\partial \psi^j_{\Tilde{\gamma}(s)}}{\partial p^\ell} \ ds} \rVert_{L^{\frac{2d}{2(1-\sigma)-1}}(\R^d)} \\
    &+ \lVert H_{\Dso^{1-\sigma}}\brac{\vec{v}, \vec{v}^j} \rVert_{L^{\frac{2d}{2(1-\sigma)-1}}(\R^d)}\\
    &\lesssim_{\Lambda} \brac{\lVert \Dso \vec{u} \rVert_{L^{2d}(\R^d)} + \lVert \Dso \vec{v} \rVert_{L^{2d}(\R^d)}}\\
    &+ \lVert H_{\Dso^{1-\sigma}}\brac{\vec{v} \vec{v}^j, \int_0^1 \frac{\partial \psi^j_{\Tilde{\gamma}(s)}}{\partial p^\ell} \ ds} \rVert_{L^{\frac{2d}{2(1-\sigma)-1}}(\R^d)} \\
    &+ \lVert \Dso^{1-\sigma - \gamma}\vec{v} \rVert_{L^{\frac{2d}{2(1-\sigma -\gamma)-\frac{1}{2}}}(\R^d)} \lVert \Dso^\gamma \vec{v}^j\rVert_{L^{\frac{2d}{2\gamma-\frac{1}{2}}}(\R^d)}\\
    &\lesssim_{\Lambda} \brac{\lVert \Dso \vec{u} \rVert_{L^{2d}(\R^d)} + \lVert \Dso \vec{v} \rVert_{L^{2d}(\R^d)}}\\
    &+ \lVert H_{\Dso^{1-\sigma}}\brac{\vec{v} \vec{v}^j, \int_0^1 \frac{\partial \psi^j_{\Tilde{\gamma}(s)}}{\partial p^\ell} \ ds} \rVert_{L^{\frac{2d}{2(1-\sigma)-1}}(\R^d)} \\
    &+ \lVert H_{\Dso^{1-\sigma}}\brac{\vec{v}, \vec{v}^j} \rVert_{L^{\frac{2d}{2(1-\sigma)-1}}(\R^d)}\\
    & \overset{\cref{lem:gag_in}}{\lesssim_{\Lambda}} \brac{\lVert \Dso \vec{u} \rVert_{L^{2d}(\R^d)} + \lVert \Dso \vec{v} \rVert_{L^{2d}(\R^d)}}\\
    &+ \lVert H_{\Dso^{1-\sigma}}\brac{\vec{v} \vec{v}^j, \int_0^1 \frac{\partial \psi^j_{\Tilde{\gamma}(s)}}{\partial p^\ell} \ ds} \rVert_{L^{\frac{2d}{2(1-\sigma)-1}}(\R^d)} \\
    &+ \lVert \Dso\vec{v} \rVert_{L^{2d}(\R^d)}^\frac{1}{2} \lVert \Dso \vec{v}^j\rVert_{L^{2d}(\R^d)}^{\frac{1}{2}}\\
    &\lesssim_{\Lambda} \brac{\lVert \Dso \vec{u} \rVert_{L^{2d}(\R^d)} + \lVert \Dso \vec{v} \rVert_{L^{2d}(\R^d)}}\\
    &+ \lVert H_{\Dso^{1-\sigma}}\brac{\vec{v} \vec{v}^j, \int_0^1 \frac{\partial \psi^j_{\Tilde{\gamma}(s)}}{\partial p^\ell} \ ds} \rVert_{L^{\frac{2d}{2(1-\sigma)-1}}(\R^d)}.
\end{align*}
For the last term above, we again use \cref{lem:prod_rule} twice with $\gamma \in (\frac{1}{4}, \frac{1}{2})$ and $\gamma' \in (\frac{1}{2}, \gamma)$, and finally use \cref{lem:gag_in} and \cref{isom_est} to get 
\begin{align*}
    &\lVert H_{\Dso^{1-\sigma}}\brac{\vec{v} \vec{v}^j, \int_0^1 \frac{\partial \psi^j_{\Tilde{\gamma}(s)}}{\partial p^\ell} \ ds} \rVert_{L^{\frac{2d}{2(1-\sigma)-1}}(\R^d)}\\
    &\lesssim \lVert \Dso^{1-\sigma -\gamma} \brac{\vec{v} \vec{v}^j} \rVert_{L^{\frac{2d}{2(1-\sigma-\gamma) - \frac{1}{2}}}(\R^d)} \ \lVert \int_0^1 \Dso^\gamma \frac{\partial \psi^j_{\Tilde{\gamma}(s)}}{\partial p^\ell} \ ds \rVert_{L^{\frac{2d}{2\gamma - \frac{1}{2}}}(\R^d)}\\
    &\lesssim \brac{\lVert \Dso^{1 - \sigma - \gamma} \vec{v} \rVert_{L^{\frac{2d}{2(1-\sigma-\gamma) - \frac{1}{2}}}(\R^d)} + \lVert H_{\Dso^{1 - \sigma - \gamma}}\brac{\vec{v}, \vec{v}^j} \rVert_{L^{\frac{}{}}(\R^d)}} \ \lVert \int_0^1 \Dso^\gamma \frac{\partial \psi^j_{\Tilde{\gamma}(s)}}{\partial p^\ell} \ ds \rVert_{L^{\frac{2d}{2\gamma - \frac{1}{2}}}(\R^d)}\\
    &\lesssim \lVert \Dso^{1 - \sigma - \gamma} \vec{v} \rVert_{L^{\frac{2d}{2(1-\sigma-\gamma) - \frac{1}{2}}}(\R^d)} \ \lVert \int_0^1 \Dso^\gamma \frac{\partial \psi^j_{\Tilde{\gamma}(s)}}{\partial p^\ell} \ ds \rVert_{L^{\frac{2d}{2\gamma - \frac{1}{2}}}(\R^d)}\\
    &+\lVert \Dso^{1 - \sigma - \gamma - \gamma'}\vec{v} \lVert_{L^{\frac{2d}{2(1-\sigma-\gamma - \gamma')- \frac{1}{4}}}(\R^d)} \ \lVert \Dso^{\gamma'}\vec{v}^j \rVert_{L^{\frac{2d}{2\gamma' - \frac{1}{4}}}(\R^d)} \ \lVert \int_0^1 \Dso^\gamma \frac{\partial \psi^j_{\Tilde{\gamma}(s)}}{\partial p^\ell} \ ds \rVert_{L^{\frac{2d}{2\gamma - \frac{1}{2}}}(\R^d)}\\
    &\overset{\cref{lem:gag_in}}{\lesssim} \lVert \Dso \vec{v} \rVert_{L^{2d}(\R^d)}^{\frac{1}{2}} \ \lVert \int_0^1 \Dso \frac{\partial \psi^j_{\Tilde{\gamma}(s)}}{\partial p^\ell} \ ds \rVert_{L^{2d}(\R^d)}^{\frac{1}{2}}\\
    &+ \lVert \Dso\vec{v} \rVert_{L^{2d}(\R^d)}^{\frac{1}{4}} \ \lVert \Dso\vec{v}^j \rVert_{L^{2d}(\R^d)}^{\frac{1}{4}} \ \lVert \int_0^1 \Dso \frac{\partial \psi^j_{\Tilde{\gamma}(s)}}{\partial p^\ell} \ ds \rVert_{L^{2d}(\R^d)}^\frac{1}{2}\\
    &\lesssim_{\Lambda} \lVert \Dso \vec{v} \rVert_{L^{2d}(\R^d)}^{\frac{1}{2}} \ \brac{\lVert \Dso \vec{u} \rVert_{L^{2d}(\R^d)} + \lVert \Dso \vec{v} \rVert_{L^{2d}(\R^d)}}^{\frac{1}{2}}\\
    &+ \lVert \Dso\vec{v} \rVert_{L^{2d}(\R^d)}^{\frac{1}{4}} \ \lVert \Dso\vec{v}^j \rVert_{L^{2d}(\R^d)}^{\frac{1}{4}} \ \brac{\lVert \Dso \vec{u} \rVert_{L^{2d}(\R^d)} + \lVert \Dso \vec{v} \rVert_{L^{2d}(\R^d)}}^{\frac{1}{2}}\\
    &\lesssim_{\Lambda} \lVert \Dso \vec{v} \rVert_{L^{2d}(\R^d)}^{\frac{1}{2}} \ \brac{\lVert \Dso \vec{u} \rVert_{L^{2d}(\R^d)}^{\frac{1}{2}} + \lVert \Dso \vec{v} \rVert_{L^{2d}(\R^d)}^{\frac{1}{2}}}\\
    &+ \lVert \Dso\vec{v} \rVert_{L^{2d}(\R^d)}^{\frac{1}{4}} \ \lVert \Dso\vec{v}^j \rVert_{L^{2d}(\R^d)}^{\frac{1}{4}} \ \brac{\lVert \Dso \vec{u} \rVert_{L^{2d}(\R^d)}^{\frac{1}{2}} + \lVert \Dso \vec{v} \rVert_{L^{2d}(\R^d)}^{\frac{1}{2}}}\\
    &\lesssim_{\Lambda} \lVert \Dso \vec{u} \rVert_{L^{2d}(\R^d)} + \lVert \Dso \vec{v} \rVert_{L^{2d}(\R^d)}
\end{align*}
Hence, for \eqref{comm_lem_1}, using \cref{lem:prod_rule}, \cref{lem:prod_rule_frac}, and the above,
\begin{align*}
    &\lVert H_{\Dso}\brac{\vec{v} \vec{v}^j \int_0^1 \frac{\partial \psi^j_{\Tilde{\gamma}(s)}}{\partial p^\ell} \ ds \ \wedge_L, \Tilde{\vec{w}}^\ell \Dso \vec{u} }\rVert_{L^{2}(\R^d)}\\
    \lesssim & \lVert \Dso^{1-\sigma} \brac{\vec{v} \vec{v}^j \int_0^1 \frac{\partial \psi^j_{\Tilde{\gamma}(s)}}{\partial p^\ell} \ ds} \rVert_{L^{\frac{2d}{2(1-\sigma)-1}}(\R^d)} \ \lVert \Dso^\sigma\brac{\Tilde{\vec{w}}^\ell \Dso \vec{u}}\rVert_{L^{\frac{2d}{d+2\sigma - 1}}(\R^d)} \\
    \lesssim_{\Lambda} &\brac{\lVert \Dso \vec{u} \rVert_{L^{2d}(\R^d)} + \lVert \Dso \vec{v} \rVert_{L^{2d}(\R^d)}} \ \lVert \Dso^{1+\sigma} \vec{u}\rVert_{L^{\frac{2d}{2(1+\sigma)-1}}(\R^d)} \ \lVert \Dso \Tilde{\vec{w}} \rVert_{L^2(\R^d)}\\
    \lesssim_{\Lambda} & \brac{\|\Ds{1+\sigma} \vec{u}\|_{L^{\frac{2d}{2(1+\sigma)-1}}(\R^d)}^2 + \|\Ds{1+\sigma} \vec{v}\|_{L^{\frac{2d}{2(1+\sigma)-1}}(\R^d)}^2} \ \|\nabla \Tilde{\vec{w}}\|_{L^2(\R^d)}.
\end{align*}

Now for \eqref{comm_lem_2}, we use \cref{lem:prod_rule}, \cref{lem:gag_in}, and \cref{lem:prod_rule_frac} to get
\begin{align*}
    &\lVert H_{\Dso}\brac{\vec{v}, \vec{v}^j \int_0^1 \frac{\partial \psi^j_{\Tilde{\gamma}(s)}}{\partial p^\ell} \ ds } \wedge_L\Tilde{\vec{w}}^\ell \Dso \vec{u} \rVert_{L^{2}(\R^d)}\\
    \lesssim & \lVert H_{\Dso}\brac{\vec{v}, \vec{v}^j\int_0^1 \frac{\partial \psi^j_{\Tilde{\gamma}(s)}}{\partial p^\ell} \ ds} \rVert_{L^{2d}(\R^d)} \ \lVert \Tilde{\vec{w}}^\ell \Dso \vec{u} \rVert_{L^{\frac{2d}{d-1}}(R^d)}\\
    \lesssim & \lVert \Dso^{1-\sigma} \vec{v} \rVert_{L^{\frac{2d}{1-\sigma}}(\R^d)} \ \lVert \Dso^{\sigma} \brac{\vec{v}^j \int_0^1 \frac{\partial \psi^j_{\Tilde{\gamma}(s)}}{\partial p^\ell} \ ds }\rVert_{L^{\frac{2d}{\sigma}}(\R^d)} \ \lVert \Dso \vec{u} \rVert_{L^{2d}(\R^d)} \ \lVert \Dso \Tilde{\vec{w}} \rVert_{L^2(R^d)}.
\end{align*}
Observe that for $\gamma \in (\frac{1}{4}\sigma, \frac{3}{4}\sigma)$, we use \cref{lem:prod_rule} and \cref{lem:gag_in} to get
\begin{align*}
    &\lVert \Dso^{\sigma} \brac{\vec{v}^j \int_0^1 \frac{\partial \psi^j_{\Tilde{\gamma}(s)}}{\partial p^\ell} \ ds }\rVert_{L^{\frac{2d}{\sigma}}(\R^d)}\\
    &\lesssim \lVert \Dso^{\sigma} \vec{v} \rVert_{L^{\frac{2d}{\sigma}}(\R^d)} + \lVert \int_0^1 \Dso^\sigma\frac{\partial \psi^j_{\Tilde{\gamma}(s)}}{\partial p^\ell} \ ds \rVert_{L^{\frac{2d}{\sigma}}(\R^d)} \\
    &+ \lVert H_{\Dso^\sigma}\brac{\vec{v}^j, \int_0^1 \frac{\partial \psi^j_{\Tilde{\gamma}(s)}}{\partial p^\ell} \ ds } \rVert_{L^{\frac{2d}{\sigma}}(\R^d)}\\
    &\lesssim \lVert \Dso^{\sigma} \vec{v} \rVert_{L^{\frac{2d}{\sigma}}(\R^d)} + \lVert \int_0^1 \Dso^\sigma\frac{\partial \psi^j_{\Tilde{\gamma}(s)}}{\partial p^\ell} \ ds \rVert_{L^{\frac{2d}{\sigma}}(\R^d)} \\
    &+ \lVert \Dso^{\sigma-\gamma}\vec{v}^j \rVert_{L^{\frac{2d}{2(\sigma-\alpha) - \frac{\sigma}{2}}}(\R^d)} \ \lVert \int_0^1 \Dso^{\gamma}\frac{\partial \psi^j_{\Tilde{\gamma}(s)}}{\partial p^\ell} \ ds \rVert_{L^{\frac{2d}{2\gamma - \frac{\sigma}{2}}}(\R^d)}\\
    &\lesssim \lVert \Dso \vec{v} \rVert_{L^{2d}(\R^d)}^\sigma + \lVert \int_0^1 \Dso\frac{\partial \psi^j_{\Tilde{\gamma}(s)}}{\partial p^\ell} \ ds \rVert_{L^{2d}(\R^d)}^\sigma \\
    &+ \lVert \Dso\vec{v}^j \rVert_{L^{2d}(\R^d)}^{\frac{\sigma}{2}} \ \lVert \int_0^1 \Dso\frac{\partial \psi^j_{\Tilde{\gamma}(s)}}{\partial p^\ell} \ ds \rVert_{L^{2d}(\R^d)}^{\frac{\sigma}{2}}\\
    &\lesssim_{\Lambda} \lVert \Dso \vec{v} \rVert_{L^{2d}(\R^d)}^\sigma + \lVert \Dso\vec{u} \rVert_{L^{2d}(\R^d)}^\sigma + \lVert \Dso\vec{v} \rVert_{L^{2d}(\R^d)}^\sigma \\
    &+ \lVert \Dso\vec{v}^j \rVert_{L^{2d}(\R^d)}^{\frac{\sigma}{2}} \ \brac{\lVert \Dso\vec{u} \rVert_{L^{2d}(\R^d)}^{\frac{\sigma}{2}} + \lVert \Dso\vec{v} \rVert_{L^{2d}(\R^d)}^{\frac{\sigma}{2}}}\\
    &\lesssim_{\Lambda} \lVert \Dso\vec{u} \rVert_{L^{2d}(\R^d)}^\sigma + \lVert \Dso\vec{v} \rVert_{L^{2d}(\R^d)}^\sigma.
\end{align*}
Therefore,
\begin{align*}
    &\lVert H_{\Dso}\brac{\vec{v}, \vec{v}^j \int_0^1 \frac{\partial \psi^j_{\Tilde{\gamma}(s)}}{\partial p^\ell} \ ds } \wedge_L\Tilde{\vec{w}}^\ell \Dso \vec{u} \rVert_{L^{2}(\R^d)}\\
    \lesssim & \lVert \Dso^{1-\sigma} \vec{v} \rVert_{L^{\frac{2d}{1-\sigma}}(\R^d)} \ \lVert \Dso^{\sigma} \brac{\vec{v}^j \int_0^1 \frac{\partial \psi^j_{\Tilde{\gamma}(s)}}{\partial p^\ell} \ ds }\rVert_{L^{\frac{2d}{\sigma}}(\R^d)} \ \lVert \Dso \vec{u} \rVert_{L^{2d}(\R^d)} \ \lVert \Dso \Tilde{\vec{w}} \rVert_{L^2(R^d)}\\
    \lesssim_{\Lambda} & \lVert \Dso \vec{v} \rVert_{L^{2d}(\R^d)}^{1-\sigma} \ \brac{\lVert \Dso\vec{u} \rVert_{L^{2d}(\R^d)}^\sigma + \lVert \Dso\vec{v} \rVert_{L^{2d}(\R^d)}^\sigma} \ \lVert \Dso \vec{u} \rVert_{L^{2d}(\R^d)} \ \lVert \Dso \Tilde{\vec{w}} \rVert_{L^2(R^d)}\\
    \lesssim_{\Lambda} & \lVert \Dso\vec{u} \rVert_{L^{2d}(\R^d)}^2 + \lVert \Dso\vec{v} \rVert_{L^{2d}(\R^d)}^2\\
\end{align*}
Similar to \eqref{comm_lem_1}, for \eqref{comm_lem_3} we have
\begin{align*}
    &\lVert \vec{v} \wedge_LH_{\Dso}\brac{\vec{v}^j \int_0^1 \frac{\partial \psi^j_{\Tilde{\gamma}(s)}}{\partial p^\ell} \ ds, \Tilde{\vec{w}}^\ell \Dso \vec{u}} \rVert_{L^{2}(\R^d)}\\
    \lesssim_{\Lambda} & \brac{\|\Ds{1+\sigma} \vec{u}\|_{L^{\frac{2d}{2(1+\sigma)-1}}(\R^d)}^2 + \|\Ds{1+\sigma} \vec{v}\|_{L^{\frac{2d}{2(1+\sigma)-1}}(\R^d)}^2} \ \|\Dso \Tilde{\vec{w}}\|_{L^2(\R^d)}.
\end{align*}
Thus, for $\sigma \in (0,\frac{1}{2}]$
\begin{align*}
&\|H_{\Dso}\brac{\vec{v} \wedge_L, \langle \Dso \vec{u}, \vec{v} \rangle_L \vec{w}} \|_{L^{2}(\R^d)} \\
&\lesssim_{\Lambda} \|\nabla \Tilde{\vec{w}}\|_{L^2(\R^d)}\, \brac{\|\Ds{1+\sigma} \vec{u}\|_{L^{\frac{2d}{2(1+\sigma)-1}}(\R^d)}^2 + \|\Ds{1+\sigma} \vec{v}\|_{L^{\frac{2d}{2(1+\sigma)-1}}(\R^d)}^2}.
\end{align*}
Similarly,
\begin{align*}
&\|H_{\Dso}\brac{\vec{v} \wedge_L, \langle \Dso \vec{u}, \vec{v} \rangle_L \vec{w}} \|_{L^{2}(\R^d)} \\
&\lesssim_{\Lambda}\brac{\|\Ds{1+\sigma} \vec{u}\|_{L^{\frac{2d}{2(1+\sigma)-1}}(\R^d)} + \|\Ds{1+\sigma} \vec{v}\|_{L^{\frac{2d}{2(1+\sigma)-1}}(\R^d)}} \ \|\Dso \Tilde{\vec{w}}\|_{L^2(\R^d)},
\end{align*} 
so we conclude.
\end{proof}
We pause here to address an issue in estimating \eqref{eq:diffenergy_3_3}.  If we were to use the definition of the Leibniz rule operator, we would get 
\begin{align*}
    \Dso \vec{w} = \int_0^1 \Dso d\psi_{\Tilde{\gamma}(s)} \Tilde{\vec{w}} \ ds + \int_0^1 d\psi_{\Tilde{\gamma}(s)} \Dso\Tilde{\vec{w}} \ ds + \int_0^1 H_{\Dso}\brac{\Dso d\psi_{\Tilde{\gamma}(s)}, \Tilde{\vec{w}}} \ ds.
\end{align*}
The term containing $H_{\Dso}\Big(\vec{v} \wedge_L, \int_0^1 \Dso d\psi_{\Tilde{\gamma}(s)} \Tilde{\vec{w}} \ ds\Big)$ will lead to the $\Dso d\psi_{\Tilde{\gamma}(s)}$ term taking more than one fractional derivative, which leads to either $\vec{u}$ or $\vec{v}$ taking more fractional derivatives than desired. If we instead use the Riesz transform expansion of the fractional Laplacian on $\Dso \vec{w}$, $\vec{u}$ or $\vec{v}$ take the desired amount of fractional derivatives. We denote the Riesz trasform by $\mathcal{R}^e$. Note that $\Dso = \brac{-\Delta}^\frac{1}{2} = c\sum_{e=1}^d \mathcal{R}^e \partial_e$ for some constant $c$. So,
\begin{align}
    \Dso \vec{w}
    &= c\sum_{e=1}^d\mathcal{R}^e \partial_e\brac{\int_0^1 d \psi_{\Tilde{\gamma}(s)} \Tilde{\vec{w}} \ ds}\nonumber\\
    &= \int_0^1 c \sum_{e=1}^{d}\mathcal{R}^e \brac{d^2 \psi_{\Tilde{\gamma}(s)} \partial_e \Tilde{\gamma}(s) \ \Tilde{\vec{w}}} \ ds + \int_0^1 c_2 \sum_{e=1}^{d}\mathcal{R}^e \brac{d \psi_{\Tilde{\gamma}(s)} \partial_e \Tilde{\vec{w}}} \ ds\nonumber\\
    &= \int_0^1 sc \sum_{e=1}^{d}\mathcal{R}^e \brac{d^2 \psi_{\Tilde{\gamma}(s)} \partial_e \Tilde{\vec{u}} \ \Tilde{\vec{w}}} \ ds
    + \int_0^1 (1-s)c \sum_{e=1}^{d}\mathcal{R}^e \brac{d^2 \psi_{\Tilde{\gamma}(s)} \partial_e \Tilde{\vec{v}} \ \Tilde{\vec{w}} } \ ds\nonumber\\
    &+ \int_0^1 c \sum_{e=1}^{d}\mathcal{R}^e \brac{d \psi_{\Tilde{\gamma}(s)} \partial_e \Tilde{\vec{w}}} \ ds\nonumber\\
    &= \int_0^1 sc \sum_{e=1}^{d}\mathcal{R}^e \brac{d^2 \psi_{\Tilde{\gamma}(s)} d\varphi_{\vec{u}} \partial_e \vec{u} \ \Tilde{\vec{w}}} \ ds \label{riesz_1}\\
    &+ \int_0^1 (1-s)c \sum_{e=1}^{d}\mathcal{R}^e \brac{d^2 \psi_{\Tilde{\gamma}(s)} d\varphi_{\vec{v}} \partial_e \vec{v} \ \Tilde{\vec{w}}} \ ds\label{riesz_2}\\
    &+ \int_0^1 c \sum_{e=1}^{d}\mathcal{R}^e \brac{d \psi_{\Tilde{\gamma}(s)} \partial_e \Tilde{\vec{w}}} \ ds \label{riesz_3}.
\end{align}
So to estimate \eqref{eq:diffenergy_3_3}, we need the following, related to \eqref{riesz_1} and \eqref{riesz_2}.
\begin{lemma} \label{riesz_estimate}
    For  $\sigma \in (0,1)$ and $\vec{f} \in \{\vec{u}, \vec{v}\}$, we have that 
    \begin{align*}
        &\left\lVert H_{\Dso}\Big(\vec{v} \wedge_L, \int_0^1\mathcal{R}^e \brac{d^2 \psi_{\Tilde{\gamma}(s)} d\varphi_{\vec{f}} \partial_e \vec{f} \ \Tilde{\vec{w}}} \ ds\Big) \right\rVert_{L^2(\R^d)} \\
        \lesssim_{\Lambda} & \brac{\lVert \Dso^{1+\sigma}\vec{u} \rVert_{L^{\frac{2d}{2(1+\sigma)-1}}(\R^d)}^2 + \lVert \Dso^{1+\sigma}\vec{v} \rVert_{L^{\frac{2d}{2(1+\sigma)-1}}(\R^d)}^2} \ \lVert \Dso \Tilde{\vec{w}} \rVert_{L^{2}(\R^d)}.
    \end{align*}
    In particular,
    \begin{align*}
         &\int_{\R^d} \langle \vec{v} \wedge_L H_{\Dso}\Big(\vec{v} \wedge_L, \int_0^1\mathcal{R}^e \brac{d^2 \psi_{\Tilde{\gamma}(s)} d\varphi_{\vec{f}} \partial_e \vec{f} \ \Tilde{\vec{w}}} \ ds\Big), \vec{X}^r_{\vec{u}}\rangle_L \ \langle \vec{E}^r_{\Tilde{\vec{u}}}, \partial_t \Tilde{\vec{w}} \rangle_{\R^m}\\
         \lesssim_{\Lambda} & \brac{\lVert \Dso^{1+\sigma}\vec{u} \rVert_{L^{\frac{2d}{2(1+\sigma)-1}}(\R^d)}^2 + \lVert \Dso^{1+\sigma}\vec{v} \rVert_{L^{\frac{2d}{2(1+\sigma)-1}}(\R^d)}^2} \ \brac{\lVert \partial_t \Tilde{\vec{w}} \rVert_{L^{2}(\R^d)}^2 + \lVert \Dso \Tilde{\vec{w}} \rVert_{L^{2}(\R^d)}^2}.\\
    \end{align*}
\end{lemma} 
\begin{proof}
Using the definition of the commutator operator, we have
\begin{align*}
    &\mathcal{R}^e \brac{d^2 \psi_{\Tilde{\gamma}(s)} d\varphi_{\vec{f}}\partial_e \vec{f} \ \Tilde{\vec{w}}}\\
    &=\Dso^{-\sigma}\Dso^\sigma\mathcal{R}^e \brac{d^2 \psi_{\Tilde{\gamma}(s)} d\varphi_{\vec{f}}\partial_e \vec{f} \ \Tilde{\vec{w}}}\\
    &= \Dso^{-\sigma}\mathcal{R}^e \brac{\Dso^\sigma \brac{d^2\psi_{\Tilde{\gamma}(s)}}d\varphi_{\vec{f}}\partial_e \vec{f} \ \Tilde{\vec{w}}}\\
    &+ \Dso^{-\sigma}\mathcal{R}^e \brac{d^2 \psi_{\Tilde{\gamma}(s)} \Dso^\sigma d\varphi_{\vec{f}}\partial_e \vec{f}   \ \Tilde{\vec{w}}}\\
    &+ \Dso^{-\sigma}\mathcal{R}^e \brac{d^2 \psi_{\Tilde{\gamma}(s)} d\varphi_{\vec{f}}\Dso^\sigma\partial_e \vec{f}  \ \Tilde{\vec{w}}}\\
    &+ \Dso^{-\sigma}\mathcal{R}^e \brac{d^2 \psi_{\Tilde{\gamma}(s)} d\varphi_{\vec{f}}\partial_e \vec{f} \ \Dso^\sigma\Tilde{\vec{w}}}\\
    &+ \Dso^{-\sigma}\mathcal{R}^e H_{\Dso^\sigma}\Big(d^2 \psi_{\Tilde{\gamma}(s)} d\varphi_{\vec{f}}, \partial_e \vec{f} \ \Tilde{\vec{w}}\Big)\\
    &+ \Dso^{-\sigma}\mathcal{R}^e\Big( H_{\Dso^\sigma}\brac{d^2 \psi_{\Tilde{\gamma}(s)}, d\varphi_{\vec{f}}}\partial_e \vec{f} \ \Tilde{\vec{w}}\Big)\\
    &+ \Dso^{-\sigma}\mathcal{R}^e \Big( d^2 \psi_{\Tilde{\gamma}(s)} d\varphi_{\vec{f}}  H_{\Dso^\sigma}(\partial_e \vec{f}, \Tilde{\vec{w}}) \Big).
\end{align*}
Hence,
\begin{align}
     &\left\lVert H_{\Dso}\brac{\vec{v} \wedge_L, \int_0^1\mathcal{R}^e \brac{d^2 \psi_{\Tilde{\gamma}(s)} d\varphi_{\vec{f}} \partial_e \vec{f} \ \Tilde{\vec{w}}} \ ds }   \right\rVert_{L^2(\R^d)} \nonumber\\
     \lesssim & \left\lVert H_{\abs{\nabla}}\brac{\vec{v} \wedge_L,\int_0^1 \Dso^{-\sigma}\mathcal{R}^e \Big(\Dso^\sigma \brac{d^2\psi_{\Tilde{\gamma}(s)}}d\varphi_{\vec{f}}\partial_e \vec{f} \ \Tilde{\vec{w}} \Big) \ ds } \right\rVert_{L^2(\R^d)} \label{final_1}\\
     &+ \left\lVert H_{\abs{\nabla}}\brac{\vec{v} \wedge_L,\int_0^1 \Dso^{-\sigma}\mathcal{R}^e \Big( d^2 \psi_{\Tilde{\gamma}(s)} \Dso^\sigma \brac{d\varphi_{\vec{f}} } \partial_e \vec{f}  \ \Tilde{\vec{w}} \Big) \ ds} \right\rVert_{L^2(\R^d)}\label{final_2}\\
     &+ \left\lVert H_{\abs{\nabla}}\brac{ \vec{v} \wedge_L,\int_0^1 \Dso^{-\sigma} \mathcal{R}^e \Big( d^2 \psi_{\Tilde{\gamma}(s)} d\varphi_{\vec{f}}\Dso^\sigma\brac{\partial_e \vec{f}} \ \Tilde{\vec{w}} \Big) \ ds} \right \rVert_{L^2(\R^d)}\label{final_3}\\
     &+ \left\lVert H_{\abs{\nabla}}\brac{\vec{v} \wedge_L,\int_0^1 \Dso^{-\sigma}\mathcal{R}^e \Big( d^2 \psi_{\Tilde{\gamma}(s)} d\varphi_{\vec{f}}\partial_e \vec{f} \ \Dso^\sigma\Tilde{\vec{w}} \Big) \ ds} \right\rVert_{L^2(\R^d)}\label{final_4}\\
     &+ \left\lVert H_{\abs{\nabla}}\brac{\vec{v} \wedge_L,\int_0^1 \Dso^{-\sigma} \mathcal{R}^e H_{\Dso^\sigma}\Big(d^2 \psi_{\Tilde{\gamma}(s)} d\varphi_{\vec{f}}, \partial_e \vec{f} \ \Tilde{\vec{w}}\Big) \ ds} \right\rVert_{L^2(\R^d)}\label{final_5}\\
     &+ \left\lVert H_{\abs{\nabla}}\brac{\vec{v} \wedge_L,\int_0^1 \Dso^{-\sigma}\mathcal{R}^e \Big( H_{\Dso^\sigma}\brac{d^2 \psi_{\Tilde{\gamma}(s)}, d\varphi_{\vec{f}}}\partial_e \vec{f} \ \Tilde{\vec{w}} \Big)                                      \ ds }\right\rVert_{L^2(\R^d)}\label{final_6}\\
     &+ \left\lVert H_{\abs{\nabla}}\brac{\vec{v} \wedge_L,\int_0^1 \Dso^{-\sigma}\mathcal{R}^e \Big( d^2 \psi_{\Tilde{\gamma}(s)} H_{\Dso^\sigma}( d\varphi_{\vec{f}}, \partial_e \Tilde{\vec{f}} ) \ \Tilde{\vec{w}} \Big) \ ds} \right\rVert_{L^2(\R^d)}\label{final_7}.
\end{align}
So for \eqref{final_1}, using \cref{lem:prod_rule} and \cref{lem:gag_in}, we have
\begin{align*}
    & \left\lVert H_{\abs{\nabla}}\brac{\vec{v} \wedge_L,\int_0^1 \Dso^{-\sigma}\mathcal{R}^e \Big(\Dso^\sigma \brac{d^2\psi_{\Tilde{\gamma}(s)}}d\varphi_{\vec{f}}\partial_e \vec{f} \ \Tilde{\vec{w}} \Big) \ ds } \right\rVert_{L^2(\R^d)}\\
    \lesssim & \lVert \Dso^{1-\sigma} \vec{v} \rVert_{L^{\frac{2d}{1-\sigma}}(\R^d)} \ \left\lVert \int_0^1 \mathcal{R}^e \Big( \Dso^\sigma \brac{d^2\psi_{\Tilde{\gamma}(s)}}d\varphi_{\vec{f}}\partial_e \vec{f} \ \Tilde{\vec{w}} \Big) \ ds \right\rVert_{L^{\frac{2d}{d+\sigma-1}}(\R^d)}\\
    \lesssim & \lVert \Dso^{1-\sigma} \vec{v} \rVert_{L^{\frac{2d}{1-\sigma}}(\R^d)} \ \left\lVert \int_0^1 \Dso^\sigma \brac{d^2\psi_{\Tilde{\gamma}(s)}}d\varphi_{\vec{f}}\partial_e \vec{f} \ \Tilde{\vec{w}} \ ds \right\rVert_{L^{\frac{2d}{d+\sigma-1}}(\R^d)}\\
    \lesssim & \lVert \Dso^{1-\sigma} \vec{v} \rVert_{L^{\frac{2d}{1-\sigma}}(\R^d)} \ \lVert \int_0^1 \Dso^\sigma \brac{d^2\psi_{\Tilde{\gamma}(s)}}d\varphi_{\vec{f}}\partial_e \vec{f} \ ds  \rVert_{L^{\frac{2d}{1+\sigma}}(\R^d)} \ \lVert \Tilde{\vec{w}} \rVert_{L^{\frac{2d}{d-2}}(\R^d)}\\
    \lesssim & \lVert \Dso^{1-\sigma} \vec{v} \rVert_{L^{\frac{2d}{1-\sigma}}(\R^d)} \ \lVert \int_0^1\Dso^\sigma d^2\psi_{\Tilde{\gamma}(s)} \ ds \rVert_{L^{\frac{2d}{\sigma}}(\R^d)} \ \lVert \partial_e \vec{f} \rVert_{L^{2d}(\R^d)} \ \lVert \Tilde{\vec{w}} \rVert_{L^{\frac{2d}{d-2}}(\R^d)}\\
    \lesssim & \lVert \Dso \vec{v} \rVert_{L^{2d}(\R^d)}^{1-\sigma} \ \lVert \int_0^1\Dso d^2\psi_{\Tilde{\gamma}(s)} \ ds \rVert_{L^{2d}(\R^d)}^\sigma \ \lVert \nabla \vec{f} \rVert_{L^{2d}(\R^d)} \ \lVert \Dso \Tilde{\vec{w}} \rVert_{L^{2}(\R^d)}\\
    \lesssim & \lVert \Dso \vec{v} \rVert_{L^{2d}(\R^d)}^{1-\sigma} \ \lVert \int_0^1\Dso d^2\psi_{\Tilde{\gamma}(s)} \ ds \rVert_{L^{2d}(\R^d)}^\sigma \ \lVert \Dso \vec{f} \rVert_{L^{2d}(\R^d)} \ \lVert \Dso \Tilde{\vec{w}} \rVert_{L^{2}(\R^d)}\\
    \lesssim & \brac{\lVert \Dso \vec{v} \rVert_{L^{2d}(\R^d)} + \lVert \int_0^1\Dso d^2\psi_{\Tilde{\gamma}(s)} \ ds \rVert_{L^{2d}(\R^d)}} \ \lVert \Dso \vec{f} \rVert_{L^{2d}(\R^d)} \ \lVert \Dso \Tilde{\vec{w}} \rVert_{L^{2}(\R^d)}\\
    \lesssim_{\Lambda} & \brac{\lVert \Dso \vec{u} \rVert_{L^{2d}(\R^d)} + \lVert \Dso \vec{v} \rVert_{L^{2d}(\R^d)}} \ \lVert \Dso \vec{f} \rVert_{L^{2d}(\R^d)} \ \lVert \Dso \Tilde{\vec{w}} \rVert_{L^{2}(\R^d)}\\
    \lesssim_{\Lambda} & \brac{\lVert \Dso \vec{u} \rVert_{L^{2d}(\R^d)}^2 + \lVert \Dso \vec{v} \rVert_{L^{2d}(\R^d)}^2} \ \lVert \Dso \Tilde{\vec{w}} \rVert_{L^{2}(\R^d)}.
\end{align*}
For \eqref{final_2}, using \cref{lem:prod_rule} with a small $\sigma$ and\cref{lem:gag_in} we have
\begin{align*}
    & \left\lVert H_{\abs{\nabla}}\brac{\vec{v} \wedge_L,\int_0^1 \Dso^{-\sigma} \mathcal{R}^e \brac{d^2 \psi_{\Tilde{\gamma}(s)} \Dso^\sigma \brac{d\varphi_{\vec{f}}} \partial_e \vec{f} \ \Tilde{\vec{w}}} \ ds} \right\rVert_{L^2(\R^d)}\\
    \lesssim & \lVert \Dso^{1-\sigma} \vec{v} \rVert_{L^{\frac{2d}{1-\sigma}}(\R^d)} \ \left\lVert \int_0^1 \mathcal{R}^e \brac{d^2 \psi_{\Tilde{\gamma}(s)} \Dso^\sigma \brac{d\varphi_{\vec{f}}}\partial_e \vec{f} \ \Tilde{\vec{w}} } \ ds \right\rVert_{L^{\frac{2d}{d + \sigma - 1}}(\R^d)}\\
     \lesssim & \lVert \Dso^{1-\sigma} \vec{v} \rVert_{L^{\frac{2d}{1-\sigma}}(\R^d)} \ \left\lVert \int_0^1 d^2 \psi_{\Tilde{\gamma}(s)} \Dso^\sigma \brac{d\varphi_{\vec{f}}} \partial_e \vec{f} \ \Tilde{\vec{w}} \ ds \right\rVert_{L^{\frac{2d}{d + \sigma - 1}}(\R^d)}\\
    \lesssim & \lVert \Dso^{1-\sigma} \vec{v} \rVert_{L^{\frac{2d}{1-\sigma}}(\R^d)} \ \lVert \int_0^1 d^2 \psi_{\Tilde{\gamma}(s)} \Dso^\sigma d\varphi_{\vec{f}}\partial_e \vec{f} \ ds \rVert_{L^{\frac{2d}{1+\sigma}}(\R^d)} \ \lVert \Tilde{\vec{w}} \rVert_{L^{\frac{2d}{d - 2}}(\R^d)}\\
    \lesssim & \lVert \Dso^{1-\sigma} \vec{v} \rVert_{L^{\frac{2d}{1-\sigma}}(\R^d)} \ \lVert \Dso^\sigma d\varphi_{\vec{f}}\partial_e \vec{f} \rVert_{L^{\frac{2d}{1+\sigma}}(\R^d)} \ \lVert \Dso \Tilde{\vec{w}} \rVert_{L^{2}(\R^d)}\\
    \lesssim & \lVert \Dso^{1-\sigma} \vec{v} \rVert_{L^{\frac{2d}{1-\sigma}}(\R^d)} \ \lVert \Dso^\sigma d\varphi_{\vec{f}} \rVert_{L^{\frac{2d}{\sigma}}(\R^d)} \ \lVert \partial_e \vec{f} \rVert_{L^{2d}(\R^d)} \ \lVert \Dso \Tilde{\vec{w}} \rVert_{L^{2}(\R^d)}\\
    \lesssim & \lVert \Dso \vec{v} \rVert_{L^{2d}(\R^d)}^{1-\sigma} \ \lVert \Dso d\varphi_{\vec{f}} \rVert_{L^{2d}(\R^d)}^\sigma \ \lVert \partial_e \vec{f} \rVert_{L^{2d}(\R^d)} \ \lVert \Dso \Tilde{\vec{w}} \rVert_{L^{2}(\R^d)}\\
    \lesssim & \lVert \Dso \vec{v} \rVert_{L^{2d}(\R^d)}^{1-\sigma} \ \lVert \nabla d\varphi_{\vec{f}} \rVert_{L^{2d}(\R^d)}^\sigma \ \lVert \partial_e \vec{f} \rVert_{L^{2d}(\R^d)} \ \lVert \Dso \Tilde{\vec{w}} \rVert_{L^{2}(\R^d)}\\
    \lesssim & \lVert \Dso \vec{v} \rVert_{L^{2d}(\R^d)}^{1-\sigma} \ \lVert d^2\varphi_{\vec{f}} \nabla \vec{f} \rVert_{L^{2d}(\R^d)}^\sigma \ \lVert \partial_e \vec{f} \rVert_{L^{2d}(\R^d)} \ \lVert \Dso \Tilde{\vec{w}} \rVert_{L^{2}(\R^d)}\\
    \lesssim_{\Lambda} & \lVert \Dso \vec{v} \rVert_{L^{2d}(\R^d)}^{1-\sigma} \ \lVert \nabla \vec{f} \rVert_{L^{2d}(\R^d)}^\sigma \ \lVert \partial_e \vec{f} \rVert_{L^{2d}(\R^d)} \ \lVert \Dso \Tilde{\vec{w}} \rVert_{L^{2}(\R^d)}\\
    \lesssim_{\Lambda} & \brac{\lVert \Dso \vec{u} \rVert_{L^{2d}(R^d)}^2 + \lVert \Dso \vec{v} \rVert_{L^{2d}(\R^d)}^2} \ \lVert \Dso \Tilde{\vec{w}} \rVert_{L^{2}(\R^d)}
\end{align*}
For \eqref{final_3}, using \cref{lem:prod_rule}, we have
\begin{align*}
    & \left\lVert H_{\abs{\nabla}}\brac{\vec{v} \wedge_L,\int_0^1 \Dso^{-\sigma}\mathcal{R}^e \brac{d^2 \psi_{\Tilde{\gamma}(s)} d\varphi_{\vec{f}}\Dso^\sigma\brac{\partial_e \vec{f}}  \ \Tilde{\vec{w}}} \ ds} \right\rVert_{L^2(\R^d)}\\
    \lesssim & \lVert \Dso^{1-\sigma} \vec{v} \rVert_{L^{\frac{2d}{2(1-\sigma)-1}}(\R^d)} \ \left\lVert \int_0^1 \mathcal{R}^e \brac{d^2 \psi_{\Tilde{\gamma}(s)} d\varphi_{\vec{f}}\Dso^\sigma\partial_e \vec{f} \ \Tilde{\vec{w}} }\ ds \right\rVert_{L^{\frac{2d}{d + 2\sigma - 1}}(\R^d)}\\
    \lesssim & \lVert \Dso^{1-\sigma} \vec{v} \rVert_{L^{\frac{2d}{2(1-\sigma)-1}}(\R^d)} \ \left\lVert \int_0^1 d^2 \psi_{\Tilde{\gamma}(s)} d\varphi_{\vec{f}}\Dso^\sigma\partial_e \vec{f} \ \Tilde{\vec{w}} \ ds \right\rVert_{L^{\frac{2d}{d + 2\sigma - 1}}(\R^d)}\\
    \lesssim & \lVert \Dso^{1-\sigma} \vec{v} \rVert_{L^{\frac{2d}{2(1-\sigma)-1}}(\R^d)} \ \left\lVert \int_0^1 \mathcal{R}^e \brac{d^2 \psi_{\Tilde{\gamma}(s)} d\varphi_{\vec{f}}\Dso^\sigma\partial_e \vec{f} } \ ds \right\rVert_{L^{\frac{2d}{2(1+\sigma)-1}}(\R^d)} \ \lVert \Tilde{\vec{w}} \rVert_{L^{\frac{2d}{d - 2}}(\R^d)}\\
    \lesssim & \lVert \Dso^{1-\sigma} \vec{v} \rVert_{L^{\frac{2d}{2(1-\sigma)-1}}(\R^d)} \ \left\lVert \int_0^1 d^2 \psi_{\Tilde{\gamma}(s)} d\varphi_{\vec{f}}\Dso^\sigma\partial_e \vec{f} \ ds \right\rVert_{L^{\frac{2d}{2(1+\sigma)-1}}(\R^d)} \ \lVert \Tilde{\vec{w}} \rVert_{L^{\frac{2d}{d - 2}}(\R^d)}\\
    \lesssim_{\Lambda} & \lVert \Dso^{1-\sigma} \vec{v} \rVert_{L^{\frac{2d}{2(1-\sigma)-1}}(\R^d)} \ \lVert\Dso^\sigma\partial_e \vec{f} \rVert_{L^{\frac{2d}{2(1+\sigma)-1}}(\R^d)} \ \lVert \Dso \Tilde{\vec{w}} \rVert_{L^{2}(\R^d)}\\
    \lesssim_{\Lambda} & \lVert \Dso^{1-\sigma} \vec{v} \rVert_{L^{\frac{2d}{2(1-\sigma)-1}}(\R^d)} \ \lVert\Dso^{1+\sigma} \vec{f} \rVert_{L^{\frac{2d}{2(1+\sigma)-1}}(\R^d)} \ \lVert \Dso \Tilde{\vec{w}} \rVert_{L^{2}(\R^d)}\\
    \lesssim_{\Lambda} & \brac{\lVert\Dso^{1+\sigma} \vec{u} \rVert_{L^{\frac{2d}{2(1+\sigma)-1}}(\R^d)}^2 + \lVert\Dso^{1+\sigma} \vec{v} \rVert_{L^{\frac{2d}{2(1+\sigma)-1}}(\R^d)}^2} \ \lVert \Dso \Tilde{\vec{w}} \rVert_{L^{2}(\R^d)}.
\end{align*}
For \eqref{final_4}, we again use \cref{lem:prod_rule}
\begin{align*}
    & \lVert H_{\abs{\nabla}}\brac{\vec{v} \wedge_L,\int_0^1 \Dso^{-\sigma}\mathcal{R}^e \brac{d^2 \psi_{\Tilde{\gamma}(s)} d\varphi_{\vec{f}}\partial_e \vec{f} \ \Dso^\sigma\Tilde{\vec{w}}} \ ds} \rVert_{L^2(\R^d)}\\
    \lesssim & \lVert \Dso^{1-\sigma} \vec{v} \rVert_{L^{\frac{2d}{2(1-\sigma)-1}}(R^d)} \ \lVert \int_0^1 \mathcal{R}^e \brac{d^2 \psi_{\Tilde{\gamma}(s)} d\varphi_{\vec{f}}\partial_e \vec{f} \ \Dso^\sigma\Tilde{\vec{w}} } \ ds \rVert_{L^{\frac{2d}{d+2\sigma-1}}(\R^d)}\\
    \lesssim & \lVert \Dso^{1-\sigma} \vec{v} \rVert_{L^{\frac{2d}{2(1-\sigma)-1}}(R^d)} \ \lVert \int_0^1 d^2 \psi_{\Tilde{\gamma}(s)} d\varphi_{\vec{f}}\partial_e \vec{f} \ \Dso^\sigma\Tilde{\vec{w}} \ ds \rVert_{L^{\frac{2d}{d+2\sigma-1}}(\R^d)}\\
    \lesssim & \lVert \Dso^{1-\sigma} \vec{v} \rVert_{L^{\frac{2d}{2(1-\sigma)-1}}(R^d)} \ \lVert \int_0^1 d^2 \psi_{\Tilde{\gamma}(s)} d\varphi_{\vec{f}}\partial_e \vec{f} \ ds \rVert_{L^{2d}(\R^d)} \lVert \Dso^\sigma\Tilde{\vec{w}} \rVert_{L^{\frac{2d}{d-2(1-\sigma)}}(\R^d)}\\
    \lesssim_{\Lambda} & \lVert \Dso^{1-\sigma} \vec{v} \rVert_{L^{\frac{2d}{2(1-\sigma)-1}}(R^d)} \ \lVert \partial_e \vec{f} \rVert_{L^{2d}(\R^d)} \lVert \Dso^\sigma\Tilde{\vec{w}} \rVert_{L^{\frac{2d}{d-2(1-\sigma)}}(\R^d)}\\
    \lesssim_{\Lambda} & \lVert \Dso^{1-\sigma} \vec{v} \rVert_{L^{\frac{2d}{2(1-\sigma)-1}}(R^d)} \ \lVert \Dso\vec{f} \rVert_{L^{2d}(\R^d)} \ \lVert \Dso^\sigma\Tilde{\vec{w}} \rVert_{L^{\frac{2d}{d-2(1-\sigma)}}(\R^d)}\\
    \lesssim_{\Lambda} & \brac{\lVert \Dso \vec{u} \rVert_{L^{2d}(R^d)}^2 + \lVert \Dso \vec{v} \rVert_{L^{2d}(\R^d)}^2} \ \lVert \Dso\Tilde{\vec{w}} \rVert_{L^{2}(\R^d)}.
\end{align*}
Next, for \eqref{final_4} we use \cref{lem:prod_rule} twice with $\sigma$ and small $\gamma \in (0,\sigma)$, and \cref{lem:gag_in}
\begin{align*}
    &\left\lVert H_{\abs{\nabla}}\brac{\vec{v} \wedge_L,\int_0^1 \Dso^{-\sigma} \mathcal{R}^e H_{\Dso^\sigma}\Big(d^2 \psi_{\Tilde{\gamma}(s)} d\varphi_{\vec{f}}, \partial_e \vec{f} \ \Tilde{\vec{w}}\Big) \ ds}\right\rVert_{L^2(\R^d)}\\
    \lesssim & \lVert \Dso^{1-\sigma} \vec{v} \rVert_{L^{\frac{2d}{1-\sigma}}(R^d)} \ \left\lVert \int_0^1 \mathcal{R}^e H_{\Dso^\sigma}\brac{d^2 \psi_{\Tilde{\gamma}(s)} d\varphi_{\vec{f}}, \partial_e \vec{f} \ \Tilde{\vec{w}}} \ ds \right\rVert_{L^{\frac{2d}{d+\sigma-1}}(\R^d)}\\
    \lesssim & \lVert \Dso^{1-\sigma} \vec{v} \rVert_{L^{\frac{2d}{1-\sigma}}(R^d)} \ \left\lVert \int_0^1 H_{\Dso^\sigma}\brac{d^2 \psi_{\Tilde{\gamma}(s)} d\varphi_{\vec{f}}, \partial_e \vec{f} \ \Tilde{\vec{w}}} \ ds \right\rVert_{L^{\frac{2d}{d+\sigma-1}}(\R^d)}\\
    \lesssim & \lVert \Dso^{1-\sigma} \vec{v} \rVert_{L^{\frac{2d}{1-\sigma}}(R^d)} \ \left\lVert \int_0^1 \Dso^{\sigma-\gamma} \brac{d^2\psi_{\Tilde{\gamma}(s)} d\varphi_{\vec{f}}} \ ds \right\rVert_{L^{\frac{2d}{2(\sigma-\gamma)-\sigma}}(\R^d)} \ \lVert \Dso^\gamma \brac{\partial_e \vec{f} \ \Tilde{\vec{w}} } \rVert_{L^{\frac{2d}{d+2\gamma-1}}(\R^d)}\\
    \lesssim & \lVert \Dso \vec{v} \rVert_{L^{2d}(R^d)}^{1-\sigma} \ \left\lVert \int_0^1 \Dso\brac{d^2\psi_{\Tilde{\gamma}(s)} d\varphi_{\vec{f}}} \ ds \right\rVert_{L^{2d}(\R^d)}^{\sigma} \ \lVert \Dso^{1+\gamma} \vec{f} \rVert_{L^{\frac{2d}{2(1+\gamma)-1}}(\R^d)}\ \lVert \Dso\Tilde{\vec{w}} \rVert_{L^{2}(\R^d)}\\
    &\overset{\cref{isom_est}}{\lesssim_{\Lambda}} \lVert \Dso \vec{v} \rVert_{L^{2d}(R^d)}^{1-\sigma} \ \brac{\lVert \Dso\vec{u}\rVert_{L^{2d}(\R^d)} + \lVert \Dso\vec{v} \rVert_{L^{2d}(\R^d)}}^{\sigma} \ \lVert \Dso^{1+\gamma} \vec{f} \rVert_{L^{\frac{2d}{2(1+\gamma)-1}}(\R^d)}\ \lVert \Dso\Tilde{\vec{w}} \rVert_{L^{2}(\R^d)}\\
    \lesssim_{\Lambda} & \brac{\lVert \Dso \vec{u} \rVert_{L^{2d}(R^d)} + \lVert \Dso \vec{v} \rVert_{L^{2d}(R^d)}}^{1-\sigma} \ \brac{\lVert \Dso\vec{u}\rVert_{L^{2d}(\R^d)} + \lVert \Dso\vec{v} \rVert_{L^{2d}(\R^d)}}^{\sigma} \ \lVert \Dso^{1+\gamma} \vec{f} \rVert_{L^{\frac{2d}{2(1+\gamma)-1}}(\R^d)}\ \lVert \Dso\Tilde{\vec{w}} \rVert_{L^{2}(\R^d)}\\
    \lesssim_{\Lambda} & \brac{\lVert \Dso \vec{u} \rVert_{L^{2d}(R^d)} + \lVert \Dso \vec{v} \rVert_{L^{2d}(R^d)}} \ \lVert \Dso^{1+\gamma} \vec{f} \rVert_{L^{\frac{2d}{2(1+\gamma)-1}}(\R^d)}\ \lVert \Dso\Tilde{\vec{w}} \rVert_{L^{2}(\R^d)}\\
    \lesssim_{\Lambda} & \brac{\lVert\Dso^{1+\sigma} \vec{u} \rVert_{L^{\frac{2d}{2(1+\sigma)-1}}(\R^d)}^2 + \lVert\Dso^{1+\sigma} \vec{v} \rVert_{L^{\frac{2d}{2(1+\sigma)-1}}(\R^d)}^2} \ \lVert \Dso \Tilde{\vec{w}} \rVert_{L^{2}(\R^d)}.
\end{align*}
For \eqref{final_5} we use the same lemmata as above to get
\begin{align*}
    & \left\lVert H_{\abs{\nabla}}\brac{\vec{v} \wedge_L,\int_0^1 \Dso^{-\sigma}\mathcal{R}^e \Big( H_{\Dso^\sigma}\brac{d^2 \psi_{\Tilde{\gamma}(s)}, d\varphi_{\vec{f}}} \partial_e \vec{f} \ \Tilde{\vec{w}} \Big) \ ds} \right\rVert_{L^2(\R^d)}\\
    \lesssim & \lVert \Dso^{1-\sigma} \vec{v} \rVert_{L^{\frac{2d}{1-\sigma}}(R^d)} \ \left\lVert \int_0^1 \mathcal{R}^e \Big( H_{\Dso^\sigma}\brac{d^2 \psi_{\Tilde{\gamma}(s)}, d\varphi_{\vec{f}}} \partial_e \vec{f} \ \Tilde{\vec{w}} \Big) \ ds \right\rVert_{L^{\frac{2d}{d+\sigma-1}}(\R^d)}\\
    \lesssim & \lVert \Dso^{1-\sigma} \vec{v} \rVert_{L^{\frac{2d}{1-\sigma}}(R^d)} \ \left\lVert \int_0^1 H_{\Dso^\sigma}\brac{d^2 \psi_{\Tilde{\gamma}(s)}, d\varphi_{\vec{f}}} \partial_e \vec{f} \ \Tilde{\vec{w}} \ ds \right\rVert_{L^{\frac{2d}{d+\sigma-1}}(\R^d)}\\
    \lesssim & \lVert \Dso^{1-\sigma} \vec{v} \rVert_{L^{\frac{2d}{1-\sigma}}(R^d)} \ \left\lVert \int_0^1 H_{\Dso^\sigma}\brac{d^2 \psi_{\Tilde{\gamma}(s)}, d\varphi_{\vec{f}}} 
    \ ds \right\rVert_{L^{\frac{2d}{\sigma}}(\R^d)} \ \lVert \partial_e \vec{f} \ \Tilde{\vec{w}} \rVert_{L^{\frac{2d}{d-1}}(\R^d)}\\
    \lesssim & \lVert \Dso^{1-\sigma} \vec{v} \rVert_{L^{\frac{2d}{1-\sigma}}(R^d)} \ \left\lVert \int_0^1 H_{\Dso^\sigma}\brac{d^2 \psi_{\Tilde{\gamma}(s)}, d\varphi_{\vec{f}}} 
    \ ds \right\rVert_{L^{\frac{2d}{\sigma}}(\R^d)} \ \lVert \partial_e \vec{f} \rVert_{L^{2d}(\R^d)}\ \lVert \Tilde{\vec{w}} \rVert_{L^{\frac{2d}{d-2}}(\R^d)}\\
    \lesssim & \lVert \Dso^{1-\sigma} \vec{v} \rVert_{L^{\frac{2d}{1-\sigma}}(R^d)} \ \left\lVert \int_0^1 \Dso^{\sigma-\gamma}d^2 \psi_{\Tilde{\gamma}(s)} \ ds\right\rVert_{L^{\frac{2d}{\sigma-\gamma}}(\R^d)} \ \lVert \Dso^\gamma d\varphi_{\vec{f}} \rVert_{L^{\frac{2d}{\gamma}}(\R^d)} \ \lVert \partial_e \vec{f} \rVert_{L^{2d}(\R^d)}\ \lVert \Tilde{\vec{w}} \rVert_{L^{\frac{2d}{d-2}}(\R^d)}\\
    \lesssim & \lVert \Dso^{1-\sigma} \vec{v} \rVert_{L^{\frac{2d}{1-\sigma}}(R^d)} \ \left\lVert \int_0^1 \Dso d^2 \psi_{\Tilde{\gamma}(s)} \ ds\right\rVert_{L^{2d}(\R^d)}^{\sigma - \gamma} \ \lVert \Dso d\varphi_{\vec{f}} \rVert_{L^{2d}(\R^d)}^\gamma \ \lVert \partial_e \vec{f} \rVert_{L^{2d}(\R^d)}\ \lVert \Tilde{\vec{w}} \rVert_{L^{\frac{2d}{d-2}}(\R^d)}\\
    \lesssim_{\Lambda} & \lVert \Dso \vec{v} \rVert_{L^{2d}(R^d)}^{1-\sigma} \ \brac{\lVert \Dso \vec{u} \rVert_{L^{2d}(\R^d)} + \lVert \Dso \vec{v} \rVert_{L^{2d}(\R^d)}}^\sigma \ \lVert \partial_e \vec{f} \rVert_{L^{2d}(\R^d)}\ \lVert \Dso\Tilde{\vec{w}} \rVert_{L^{2d}(\R^d)}\\
    \lesssim_{\Lambda} & \brac{\lVert \Dso \vec{u} \rVert_{L^{2d}(R^d)} + \lVert \Dso \vec{v} \rVert_{L^{2d}(R^d)}} \ \lVert \partial_e \vec{f} \rVert_{L^{2d}(\R^d)}\ \lVert \Dso\Tilde{\vec{w}} \rVert_{L^{2d}(\R^d)}\\
    \lesssim_{\Lambda} & \brac{\lVert \Dso \vec{u} \rVert_{L^{2d}(R^d)}^2 + \lVert \Dso \vec{v} \rVert_{L^{2d}(R^d)}^2} \ \lVert \Dso\Tilde{\vec{w}} \rVert_{L^{2d}(\R^d)}.
\end{align*}
Finally for \eqref{final_6}, 
\begin{align*}
    & \left\lVert H_{\abs{\nabla}}\brac{\vec{v} \wedge_L,\int_0^1 \Dso^{-\sigma}\mathcal{R}^e \Big(d^2 \psi_{\Tilde{\gamma}(s)} d\varphi_{\vec{f}} \ H_{\Dso^\sigma}\brac{\partial_e \vec{f}, \Tilde{\vec{w}}} \Big) \ ds} \right\rVert_{L^2(\R^d)}\\
    \lesssim & \lVert \Dso^{1-\sigma} \vec{v} \rVert_{L^{\frac{2d}{2(1-\sigma)-1}}(R^d)} \ \left\lVert \int_0^1 \mathcal{R}^e \Big(d^2 \psi_{\Tilde{\gamma}(s)} d\varphi_{\vec{f}} \ H_{\Dso^\sigma}\brac{\partial_e \vec{f}, \Tilde{\vec{w}}} \Big) \ ds \right\rVert_{L^{\frac{2d}{d+2\sigma-1}}(\R^d)}\\
    \lesssim & \lVert \Dso^{1-\sigma} \vec{v} \rVert_{L^{\frac{2d}{2(1-\sigma)-1}}(R^d)} \ \left\lVert \int_0^1 d^2 \psi_{\Tilde{\gamma}(s)} d\varphi_{\vec{f}} \ H_{\Dso^\sigma}\brac{\partial_e \vec{f}, \Tilde{\vec{w}}} \ ds \right\rVert_{L^{\frac{2d}{d+2\sigma-1}}(\R^d)}\\
    \lesssim_{\Lambda} & \lVert \Dso^{1-\sigma} \vec{v} \rVert_{L^{\frac{2d}{2(1-\sigma)-1}}(R^d)} \ \left\lVert H_{\Dso^\sigma}\brac{\partial_e \vec{f}, \Tilde{\vec{w}}} \right\rVert_{L^{\frac{2d}{d+2\sigma-1}}(\R^d)}\\
    \lesssim_{\Lambda} & \lVert \Dso^{1-\sigma} \vec{v} \rVert_{L^{\frac{2d}{2(1-\sigma)-1}}(R^d)} \ \lVert \Dso^{\sigma-\gamma}\partial_e \vec{f} \rVert_{L^{\frac{2d}{2(1+ \sigma - \gamma) - 1}}(\R^d)} \ \lVert \Dso^\gamma\Tilde{\vec{w}} \rVert_{L^{\frac{2d}{d-2(1-\gamma)}}(\R^d)}\\
    \lesssim_{\Lambda} & \lVert \Dso^{1-\sigma} \vec{v} \rVert_{L^{\frac{2d}{2(1-\sigma)-1}}(R^d)} \ \lVert \Dso^{\sigma-\gamma}\partial_e \vec{f} \rVert_{L^{\frac{2d}{2(1+ \sigma - \gamma) - 1}}(\R^d)} \ \lVert \Dso^\gamma\Tilde{\vec{w}} \rVert_{L^{\frac{2d}{d-2(1-\gamma)}}(\R^d)}\\
    \lesssim_{\Lambda} & \brac{\lVert\Dso^{1+\sigma} \vec{u} \rVert_{L^{\frac{2d}{2(1+\sigma)-1}}(\R^d)}^2 + \lVert\Dso^{1+\sigma} \vec{v} \rVert_{L^{\frac{2d}{2(1+\sigma)-1}}(\R^d)}^2} \ \lVert \Dso\Tilde{\vec{w}} \rVert_{L^{2d}(\R^d)}.
\end{align*}
Hence,
\begin{align*}
         &\int_{\R^d} \langle \vec{v} \wedge_L H_{\Dso}\Big(\vec{v} \wedge_L, \Dso \Big(\int_0^1\mathcal{R}^e \brac{d^2 \psi_{\Tilde{\gamma}(s)} d\varphi_{\vec{f}} \partial_e \vec{f} \ \Tilde{\vec{w}}} \ ds\Big)\Big), \vec{X}^r_{\vec{u}}\rangle_L \ \langle \vec{E}^r_{\Tilde{\vec{u}}}, \partial_t \Tilde{\vec{w}} \rangle_{\R^m}\\
         \lesssim_{\Lambda} & \brac{\lVert \Dso^{1+\sigma}\vec{u} \rVert_{L^{\frac{2d}{2(1+\sigma)-1}}(\R^d)}^2 + \lVert \Dso^{1+\sigma}\vec{v} \rVert_{L^{\frac{2d}{2(1+\sigma)-1}}(\R^d)}^2} \ \brac{\lVert \partial_t \Tilde{\vec{w}} \rVert_{L^{2}(\R^d)}^2 + \lVert \Dso \Tilde{\vec{w}} \rVert_{L^{2}(\R^d)}^2}.\\
    \end{align*}
So we conclude.
\end{proof}
We continue on the route of estimating \eqref{eq:diffenergy_3_3}. To finish the estimate, we need to address the estimate related to \eqref{riesz_3}. We first expand
\begin{align*}
 \partial_t \vec{w} =& \vec{u} \wedge_L\Dso \vec{u} - \vec{v}\wedge_L \Dso \vec{v}\\
 =& \vec{w} \wedge_L\Dso \vec{u} + \vec{v} \wedge_L \Dso \vec{w}
\end{align*}
 
\medskip

and then we expand
\begin{align*}
    &\langle\mathbf{E}_{\Tilde{\vec{u}}}^r, \partial_t\Tilde{\vec{w}} \rangle_{\R^m}\\
    =&\langle\mathbf{E}_{\Tilde{\vec{u}}}^r, \partial_t\Tilde{\vec{u}} - \partial_t\Tilde{\vec{v}} \rangle_{\R^m}\\
    =&\langle\mathbf{E}_{\Tilde{\vec{u}}}^r, d\varphi_{\vec{u}}\partial_t\vec{u} \rangle_{\R^m} - \langle \mathbf{E}_{\Tilde{\vec{u}}}^r, d\varphi_{\vec{v}}\partial_t\vec{v} \rangle_{\R^m}\\
    =&\langle\mathbf{X}_{\vec{u}}^r, \partial_t\vec{u} \rangle_L - \langle \mathbf{E}_{\Tilde{\vec{v}}}^r, d\varphi_{\vec{v}}\partial_t\vec{v} \rangle_{\R^m} - \langle \mathbf{E}_{\Tilde{\vec{u}}}^r - \mathbf{E}_{\Tilde{\vec{v}}}^r, d\varphi_{\vec{v}}\partial_t\vec{v} \rangle_{\R^m}\\
    =&\langle\mathbf{X}_{\vec{u}}^r, \partial_t\vec{u} \rangle_L - \langle \mathbf{X}_{\vec{v}}^r, \partial_t\vec{v} \rangle_L - \langle \mathbf{E}_{\Tilde{\vec{u}}}^r - \mathbf{E}_{\Tilde{\vec{v}}}^r, d\varphi_{\vec{v}}\partial_t\vec{v} \rangle_{\R^m}\\
    =&\langle\mathbf{X}_{\vec{u}}^r - \mathbf{X}_{\vec{v}}^r, \partial_t\vec{u} \rangle_L + \langle \mathbf{X}_{\vec{v}}^r, \partial_t\vec{u} - \partial_t\vec{v} \rangle_L - \langle \mathbf{E}_{\Tilde{\vec{u}}}^r - \mathbf{E}_{\Tilde{\vec{v}}}^r, d\varphi_{\vec{v}}\partial_t\vec{v} \rangle_{\R^m}\\
    =&\langle\mathbf{X}_{\vec{u}}^r - \mathbf{X}_{\vec{v}}^r, \partial_t\vec{u} \rangle_L - \langle \mathbf{E}_{\Tilde{\vec{u}}}^r - \mathbf{E}_{\Tilde{\vec{v}}}^r, d\varphi_{\vec{v}}\partial_t\vec{v} \rangle_{\R^m} + \langle \mathbf{X}_{\vec{v}}^r, \vec{w} \wedge_L\Dso \vec{u}\rangle_L + \langle \mathbf{X}_{\vec{v}}^r, \vec{v} \wedge_L\Dso \vec{w}\rangle_L\\
    =&\langle\mathbf{X}_{\vec{u}}^r - \mathbf{X}_{\vec{v}}^r, \vec{u} \wedge_L \Dso\vec{u} \rangle_L - \langle \mathbf{E}_{\Tilde{\vec{u}}}^r - \mathbf{E}_{\Tilde{\vec{v}}}^r, d\varphi_{\vec{v}}\brac{\vec{v} \wedge_L \Dso\vec{v}} \rangle_{\R^m} + \langle \mathbf{X}_{\vec{v}}^r, \vec{w} \wedge_L\Dso \vec{u}\rangle_L + \langle \mathbf{X}_{\vec{v}}^r, \vec{v} \wedge_L\Dso \vec{w}\rangle_L.
\end{align*}
To finish estimating \eqref{eq:diffenergy_3_3}, we take \eqref{riesz_3} and expand
\begin{align}
    &\langle\vec{v} \wedge_L H_{\abs{\nabla}}\brac{\vec{v} \wedge_L,\int_0^1 \mathcal{R}^e \brac{d \psi_{\Tilde{\gamma}(s)} \partial_e \Tilde{\vec{w}}} \ ds}, \mathbf{X}_\vec{u}^r\rangle_L\langle\mathbf{E}_{\Tilde{\vec{u}}}^r, \Tilde{\vec{w}}_t \rangle_{\R^m}\nonumber\\
    =&\langle\vec{v} \wedge_L H_{\abs{\nabla}}\brac{\vec{v} \wedge_L,\int_0^1 \mathcal{R}^e \brac{d \psi_{\Tilde{\gamma}(s)} \partial_e \Tilde{\vec{w}}} \ ds}, \mathbf{X}_\vec{u}^r\rangle_L \langle\mathbf{X}_{\vec{u}}^r - \mathbf{X}_{\vec{v}}^r, \vec{u} \wedge_L \Dso\vec{u} \rangle_L\label{eq:last_3}\\
    -&\langle\vec{v} \wedge_L H_{\abs{\nabla}}\brac{\vec{v} \wedge_L,\int_0^1 \mathcal{R}^e \brac{d \psi_{\Tilde{\gamma}(s)} \partial_e \Tilde{\vec{w}}} \ ds}, \mathbf{X}_\vec{u}^r\rangle_L \langle \mathbf{E}_{\Tilde{\vec{u}}}^r - \mathbf{E}_{\Tilde{\vec{v}}}^r, d\varphi_{\vec{v}}\brac{\vec{v} \wedge_L \Dso\vec{v}} \rangle_{\R^m}\label{eq:last_4}\\
    +&\langle\vec{v} \wedge_L H_{\abs{\nabla}}\brac{\vec{v} \wedge_L,\int_0^1 \mathcal{R}^e \brac{d \psi_{\Tilde{\gamma}(s)} \partial_e \Tilde{\vec{w}}} \ ds}, \mathbf{X}_\vec{u}^r\rangle_L \langle \mathbf{X}_{\vec{v}}^r, \vec{w} \wedge_L\Dso \vec{u}\rangle_L\label{eq:last_1}\\
    +&\langle\vec{v} \wedge_L H_{\abs{\nabla}}\brac{\vec{v} \wedge_L,\int_0^1 \mathcal{R}^e \brac{d \psi_{\Tilde{\gamma}(s)} \partial_e \Tilde{\vec{w}}} \ ds}, \mathbf{X}_\vec{u}^r\rangle_L \langle \mathbf{X}_{\vec{v}}^r, \vec{v} \wedge_L\Dso \vec{w}\rangle_L\label{eq:last_2}.
\end{align}
Now we are left with four more terms to estimate, which will conclude the estimation of \eqref{eq:diffenergy_3_3}.
\medskip
Thus, we continue and estimate \eqref{eq:last_1}.
\begin{lemma}
For $d \geq 3$ and any $\sigma \in (0,1/2]$ we have 
\begin{align*}
 &\abs{\int_{\R^d} \langle \vec{v} \wedge_L H_{\Dso} (\vec{v} \wedge_L, \int_0^1 \mathcal{R}^e \brac{d \psi_{\Tilde{\gamma}(s)} \partial_e \Tilde{\vec{w}}} \ ds), \ \vec{w} \wedge_L\Dso{} \vec{u}\rangle_L}\\
 &\lesssim_{\Lambda} \|\Dso \Tilde{\vec{w}}\|_{L^{2}(\R^d)}^2\, \big(\|\Dso{}^{1+\sigma} \vec{u}\|_{\frac{2d}{2(1+\sigma)-1}}^2 + \|\Dso{}^{1+\sigma} \vec{v}\|_{\frac{2d}{2(1+\sigma)-1}}^2\big)
\end{align*}
\end{lemma}
\begin{proof}
Using fractional integration by parts, observe that 
\begin{align*}
    &\int_{\mathbb{R}^d} \langle H_{\Dso}(\vec{a}\wedge_L,\vec{b}), \vec{c}\rangle_L\\
    =&\int_{\mathbb{R}^d} \langle\Dso(\vec{a}\wedge_L \vec{b}), \vec{c}\rangle_L -\langle\Dso\vec{a}\wedge_L \vec{b}, \vec{c}\rangle_L - \langle\vec{a}\wedge_L \Dso\vec{b}, \vec{c}\rangle_L\\
    =&\int_{\mathbb{R}^d} \langle\vec{a}\wedge_L \vec{b}, \Dso\vec{c}\rangle_L + \langle\Dso\vec{a}\wedge_L \vec{c}, \vec{b}\rangle_L + \langle\vec{a}\wedge_L \vec{c}, \Dso\vec{b}\rangle_L\\
    =&\int_{\mathbb{R}^d} -\langle\vec{a}\wedge_L \Dso\vec{c}, \vec{b}\rangle_L + \langle\Dso\vec{a}\wedge_L \vec{c}, \vec{b}\rangle_L + \langle\Dso(\vec{a}\wedge_L \vec{c}), \vec{b}\rangle_L\\
    =&\int_{\mathbb{R}^d} \langle H_{\Dso}(\vec{a}\wedge_L,\vec{c}),\vec{b}\rangle_L +2\langle\Dso\vec{a}\wedge_L \vec{c}, \vec{b}\rangle_L.
\end{align*}
We recall the formula
\begin{align*}
\vec{a} \wedge_L(\vec{b} \wedge_L\vec{c})= \langle\vec{b}, \vec{a}\rangle_L\vec{c} - \langle\vec{c}, \vec{a}\rangle_L\vec{b} .
\end{align*}

Using the above, we have
\begin{align}
    &\int_{\R^d} \langle \vec{v} \wedge_L H_{\Dso} (\vec{v} \wedge_L,\int_0^1 \mathcal{R}^e \brac{d \psi_{\Tilde{\gamma}(s)} \partial_e \Tilde{\vec{w}}} \ ds), \ \vec{w} \wedge_L\Dso{} \vec{u}\rangle_L\nonumber\\
    =& -\int_{\R^d} \langle H_{\Dso} (\vec{v} \wedge_L, \int_0^1 \mathcal{R}^e \brac{d \psi_{\Tilde{\gamma}(s)} \partial_e \Tilde{\vec{w}}} \ ds), \ \vec{v} \wedge_L(\vec{w} \wedge_L\Dso{} \vec{u})\rangle_L \nonumber\\
    =&-\int_{\R^d} \Big\langle H_{\Dso} (\vec{v} \wedge_L,\int_0^1 \mathcal{R}^e \brac{d \psi_{\Tilde{\gamma}(s)} \partial_e \Tilde{\vec{w}}} \ ds), \ \langle \vec{v},\Dso\vec{u}\rangle_L \vec{w}\Big\rangle_L \nonumber\\
    &+ \Big\langle H_{\Dso} (\vec{v} \wedge_L,\int_0^1 \mathcal{R}^e \brac{d \psi_{\Tilde{\gamma}(s)} \partial_e \Tilde{\vec{w}}} \ ds), \ \langle \vec{v},\vec{w}\rangle_L \Dso\vec{u}\Big\rangle_L \nonumber\\
    =&-\int_{\R^d} \Big\langle H_{\Dso} \brac{\vec{v} \wedge_L, \langle \vec{v},\Dso\vec{u}\rangle_L \vec{w}}, \ \int_0^1 \mathcal{R}^e \brac{d \psi_{\Tilde{\gamma}(s)} \partial_e \Tilde{\vec{w}}} \ ds)\Big\rangle_L\label{comm_lem_a}\\
    &- 2\int_{\R^d} \Big\langle \Dso\vec{v} \wedge_L\langle \vec{v},\Dso\vec{u}\rangle_L \vec{w}, \ \int_0^1 \mathcal{R}^e \brac{d \psi_{\Tilde{\gamma}(s)} \partial_e \Tilde{\vec{w}}} \ ds\Big\rangle_L\label{comm_lem_b}\\
    &+ \int_{\R^d}\Big\langle H_{\Dso}\brac{\vec{v} \wedge_L, \langle \vec{v},\vec{w}\rangle_L \Dso\vec{u}}, \ \int_0^1 \mathcal{R}^e \brac{d \psi_{\Tilde{\gamma}(s)} \partial_e \Tilde{\vec{w}}} \ ds\Big\rangle_L\label{comm_lem_c}\\
    &+ 2\int_{\R^d}\Big\langle \Dso\vec{v} \wedge_L\langle \vec{v},\vec{w}\rangle_L \Dso\vec{u}, \ \int_0^1 \mathcal{R}^e \brac{d \psi_{\Tilde{\gamma}(s)} \partial_e \Tilde{\vec{w}}} \ ds\Big\rangle_L \label{comm_lem_d}.
\end{align}
Using $\cref{comm_lemm}$, we have that 
\begin{align*}
&\|H_{\Dso}\brac{\vec{v} \wedge_L, \langle \Dso \vec{u}, \vec{v} \rangle_L \vec{w}} \|_{L^{d}(\R^d)} \\
&\lesssim_{\Lambda} \|\nabla \Tilde{\vec{w}}\|_{L^2(\R^d)}\, \brac{\|\Ds{1+\sigma} \vec{u}\|_{L^{\frac{2d}{2(1+\sigma)-1}}(\R^d)} + \|\Ds{1+\sigma} \vec{v}\|_{L^{\frac{2d}{2(1+\sigma)-1}}(\R^d)}},
\end{align*}
and 
\begin{align*}
&\|H_{\Dso}\brac{\vec{v} \wedge_L, \langle \vec{v},\vec{w} \rangle_L \Dso \vec{u}}  \|_{L^{d}(\R^d)} \\
&\lesssim_{\Lambda} \|\nabla \Tilde{\vec{w}}\|_{L^2(\R^d)}\, \brac{\|\Ds{1+\sigma} \vec{u}\|_{L^{\frac{2d}{2(1+\sigma)-1}}(\R^d)} + \|\Ds{1+\sigma} \vec{v}\|_{L^{\frac{2d}{2(1+\sigma)-1}}(\R^d)}}.
\end{align*}
Thus, for \eqref{comm_lem_a}
\begin{align*}
    &\int_{\R^d} \Big\langle H_{\Dso} \brac{\vec{v} \wedge_L, \langle \vec{v},\Dso\vec{u}\rangle_L \vec{w}}, \ \int_0^1 \mathcal{R}^e \brac{d \psi_{\Tilde{\gamma}(s)} \partial_e \Tilde{\vec{w}}} \ ds)\Big\rangle_L\\
    &\lesssim_{\Lambda} \|\nabla \Tilde{\vec{w}}\|_{L^2(\R^d)}^2\, \brac{\|\Ds{1+\sigma} \vec{u}\|_{L^{\frac{2d}{2(1+\sigma)-1}}(\R^d)} + \|\Ds{1+\sigma} \vec{v}\|_{L^{\frac{2d}{2(1+\sigma)-1}}(\R^d)}}
\end{align*}
and for $\eqref{comm_lem_c}$
\begin{align*}
    &\int_{\R^d}\Big\langle H_{\Dso}\brac{\vec{v} \wedge_L, \langle \vec{v},\vec{w}\rangle_L \Dso\vec{u}}, \ \int_0^1 \mathcal{R}^e \brac{d \psi_{\Tilde{\gamma}(s)} \partial_e \Tilde{\vec{w}}} \ ds\Big\rangle_L\\
    &\lesssim_{\Lambda} \|\nabla \Tilde{\vec{w}}\|_{L^2(\R^d)}^2\, \brac{\|\Ds{1+\sigma} \vec{u}\|_{L^{\frac{2d}{2(1+\sigma)-1}}(\R^d)} + \|\Ds{1+\sigma} \vec{v}\|_{L^{\frac{2d}{2(1+\sigma)-1}}(\R^d)}}
\end{align*}
Now for $\eqref{comm_lem_b}$, we have that
\begin{align*}
    &\int_{\R^d} \Big\langle \Dso\vec{v} \wedge_L\langle \vec{v},\Dso\vec{u}\rangle_L \vec{w}, \ \int_0^1 \mathcal{R}^e \brac{d \psi_{\Tilde{\gamma}(s)} \partial_e \Tilde{\vec{w}}} \ ds\Big\rangle_L\\
    \lesssim_{\Lambda} & \lVert \Dso \vec{v} \rVert_{L^{2d}(\R^d)} \ \lVert \Dso\vec{u}\rVert_{L^{2d}(\R^d)} \ \lVert \Tilde{\vec{w}} \rVert_{L^{\frac{2d}{d-2}}(\R^d)} \ \lVert \Dso \Tilde{\vec{w}} \rVert_{L^{2}(\R^d)}\\
    \lesssim_{\Lambda} & \brac{\lVert \Dso \vec{u} \rVert_{L^{2d}(\R^d)}^2 + \lVert \Dso\vec{u}\rVert_{L^{2d}(\R^d)}^2} \ \lVert \Dso\Tilde{\vec{w}} \rVert_{L^{}(\R^d)}^2
\end{align*}
and for \eqref{comm_lem_d}, we have
\begin{align*}
    &\int_{\R^d}\Big\langle \Dso\vec{v} \wedge_L\langle \vec{v},\vec{w}\rangle_L \Dso\vec{u}, \int_0^1 \mathcal{R}^e \brac{d \psi_{\Tilde{\gamma}(s)} \partial_e \Tilde{\vec{w}}} \ ds \ ds\Big\rangle_L\\
    \lesssim_{\Lambda} & \lVert \Dso \vec{v} \rVert_{L^{2d}(\R^d)} \ \lVert \Dso\vec{u}\rVert_{L^{2d}(\R^d)} \ \lVert \Tilde{\vec{w}} \rVert_{L^{\frac{2d}{d-2}}(\R^d)} \ \lVert \Dso\vec{w} \rVert_{L^{2}(\R^d)}\\
    \lesssim_{\Lambda} & \brac{\lVert \Dso \vec{u} \rVert_{L^{2d}(\R^d)}^2 + \lVert \Dso\vec{u}\rVert_{L^{2d}(\R^d)}^2} \ \lVert \Dso\Tilde{\vec{w}} \rVert_{L^{}(\R^d)}^2
\end{align*}
We conclude.
\end{proof} 
Now that we have estimated \eqref{eq:last_1}, we estimate \eqref{eq:last_2}.
Using 
\begin{align*}
    \vec{a} \wedge_L(\vec{b} \wedge_L\vec{c})= \langle\vec{b}, \vec{a}\rangle_L\vec{c}  - \langle\vec{c}, \vec{a}\rangle_L\vec{b},
\end{align*}
we have
\begin{align}
    & \vec{v}\wedge_LH_{\abs{\nabla}}\brac{\vec{v} \wedge_L, \int_0^1 \mathcal{R}^e \brac{d \psi_{\Tilde{\gamma}(s)} \partial_e \Tilde{\vec{w}}} \ ds} \nonumber\\
    &= \sideset{}{'}{\sum}_{j=1}^3  \vec{v}^jH_{\abs{\nabla}}\brac{\vec{v}^j, \int_0^1 \mathcal{R}^e \brac{d \psi_{\Tilde{\gamma}(s)} \partial_e \Tilde{\vec{w}}} \ ds} - \sideset{}{'}{\sum}_{j=1}^3  \sum_{\ell=1}^m\vec{v}^j H_{\abs{\nabla}}\brac{\vec{v}, \int_0^1 \mathcal{R}^e\brac{\frac{\partial \psi^j_{\Tilde{\gamma}(s)}}{\partial p^\ell} \ \partial_e \Tilde{\vec{w}}^\ell} \ ds} \label{split}.
\end{align}
We expand the second term of \eqref{split} using the definition of the Leibniz operator to get
\begin{align}
    &\sideset{}{'}{\sum}_{j=1}^3 \sum_{\ell=1}^m\langle \vec{v}^jH_{\abs{\nabla}}\brac{\vec{v}, \int_0^1 \mathcal{R}^e\brac{\frac{\partial \psi^j_{\Tilde{\gamma}(s)}}{\partial p^\ell} \ \partial_e \Tilde{\vec{w}}^\ell} \ ds}, \vec{v} \wedge_L\Dso\vec{w} \rangle_L\nonumber\\
    =& \sideset{}{'}{\sum}_{j=1}^3 \sum_{\ell=1}^m\langle \vec{v}^jH_{\abs{\nabla}}\brac{\vec{v}, \int_0^1 \mathcal{R}^e\brac{\frac{\partial \psi^j_{\Tilde{\gamma}(s)}}{\partial p^\ell} \ \partial_e \Tilde{\vec{w}}^\ell} \ ds}, \vec{v} \wedge_L\int_0^1 \Dso d\psi_{\Tilde{\gamma}(s)}\Tilde{\vec{w}} \ ds \rangle_L\label{split_1}\\
    &+ \sideset{}{'}{\sum}_{j=1}^3 \sum_{\ell=1}^m\langle \vec{v}^jH_{\abs{\nabla}}\brac{\vec{v}, \int_0^1 \mathcal{R}^e\brac{\frac{\partial \psi^j_{\Tilde{\gamma}(s)}}{\partial p^\ell} \ \partial_e \Tilde{\vec{w}}^\ell} \ ds}, \vec{v} \wedge_L\int_0^1 d \psi_{\Tilde{\gamma}(s)} \Dso \Tilde{\vec{w}} \ ds \rangle_L\label{split_2}\\
    &+ \sideset{}{'}{\sum}_{j=1}^3 \sum_{\ell=1}^m \langle \vec{v}^jH_{\abs{\nabla}}\brac{\vec{v}, \int_0^1 \mathcal{R}^e\brac{\frac{\partial \psi^j_{\Tilde{\gamma}(s)}}{\partial p^\ell} \ \partial_e \Tilde{\vec{w}}^\ell} \ ds}, \vec{v} \wedge_L\int_0^1  H_{\Dso}\brac{d\psi_{\Tilde{\gamma}(s)}, \Tilde{\vec{w}}} \ ds \rangle_L\label{split_3}.
\end{align}
To estimate \eqref{eq:last_2}, we need estimates related to the first and second term of \eqref{split}. We start by getting estimates related to the second term, and in particular, those related to \eqref{split_1}.
\begin{lemma}
For $d \geq 3$ and any $\sigma \in (0,1/2]$ we have 
\begin{align*}
    &\abs{\int_{\R^d} \sideset{}{'}{\sum}_{j=1}^3 \sum_{\ell=1}^m \langle \vec{v}^jH_{\abs{\nabla}}\brac{\vec{v}, \int_0^1 \mathcal{R}^e\brac{\frac{\partial \psi^j_{\Tilde{\gamma}(s)}}{\partial p^\ell} \ \partial_e \Tilde{\vec{w}}^\ell} \ ds}, \vec{v} \wedge_L\int_0^1 \Dso d\psi_{\Tilde{\gamma}(s)}\Tilde{\vec{w}} \ ds \rangle_L } \\
    &\lesssim_{\Lambda} \lVert \abs{\nabla} \Tilde{\vec{w}}\rVert_{L^{2}(\R^d)}^2\, \brac{\lVert\abs{\nabla}^{1+\sigma} \vec{u}\rVert_{L^\frac{2d}{2(1+\sigma)-1}(\mathbb{R}^d)}^2 + \lVert\abs{\nabla}^{1+\sigma} \vec{v}\rVert_{L^\frac{2d}{2(1+\sigma)-1}(\mathbb{R}^d)}^2}.
\end{align*}
\end{lemma}
\begin{proof}
Using fractional integration by parts, observe that 
\begin{align*}
    &\int_{\mathbb{R}^d} \langle H_{\Dso}(\vec{a}\wedge_L,\vec{b}), \vec{c}\rangle_L\\
    =&\int_{\mathbb{R}^d} \langle\Dso(\vec{a}\wedge_L \vec{b}), \vec{c}\rangle_L -\langle\Dso\vec{a}\wedge_L \vec{b}, \vec{c}\rangle_L - \langle\vec{a}\wedge_L \Dso\vec{b}, \vec{c}\rangle_L\\
    =&\int_{\mathbb{R}^d} \langle\vec{a}\wedge_L \vec{b}, \Dso\vec{c}\rangle_L + \langle\Dso\vec{a}\wedge_L \vec{c}, \vec{b}\rangle_L + \langle\vec{a}\wedge_L \vec{c}, \Dso\vec{b}\rangle_L\\
    =&\int_{\mathbb{R}^d} -\langle\vec{a}\wedge_L \Dso\vec{c}, \vec{b}\rangle_L + \langle\Dso\vec{a}\wedge_L \vec{c}, \vec{b}\rangle_L + \langle\Dso(\vec{a}\wedge_L \vec{c}), \vec{b}\rangle_L\\
    =&\int_{\mathbb{R}^d} \langle H_{\Dso}(\vec{a}\wedge_L,\vec{c}),\vec{b}\rangle_L +2\langle\Dso\vec{a}\wedge_L \vec{c}, \vec{b}\rangle_L
\end{align*}
and using \cref{comm_lemm}, we have
\begin{align}
    &\int_{\R^d} \langle \vec{v}^jH_{\abs{\nabla}}\brac{\vec{v}, \int_0^1 \mathcal{R}^e\brac{\frac{\partial \psi^j_{\Tilde{\gamma}(s)}}{\partial p^\ell} \ \partial_e \Tilde{\vec{w}}^\ell} \ ds}, \vec{v} \wedge_L\int_0^1 \Dso d\psi_{\Tilde{\gamma}(s)}\Tilde{\vec{w}} \ ds \rangle_L \nonumber\\
    &= \int_{\R^d} \langle H_{\abs{\nabla}}\brac{\vec{v}, \vec{v}^j\vec{v} \wedge_L\int_0^1 \Dso d\psi_{\Tilde{\gamma}(s)}\Tilde{\vec{w}} \ ds}\rangle_L \ \int_0^1 \mathcal{R}^e\brac{\frac{\partial \psi^j_{\Tilde{\gamma}(s)}}{\partial p^\ell} \ \partial_e \Tilde{\vec{w}}^\ell} \ ds \nonumber\\
    &+2\int_{\R^d} \langle \abs{\nabla}\vec{v}, \vec{v}^j\vec{v} \wedge_L\int_0^1 \Dso d\psi_{\Tilde{\gamma}(s)}\Tilde{\vec{w}} \ ds\rangle_L \ \int_0^1 \mathcal{R}^e\brac{\frac{\partial \psi^j_{\Tilde{\gamma}(s)}}{\partial p^\ell} \ \partial_e \Tilde{\vec{w}}^\ell} \ ds \nonumber\\
    &= \int_{\R^d} \langle \textcolor{blue}{H_{\abs{\nabla}}}\brac{\vec{v}, \vec{v}^j\vec{v} \wedge_L \textcolor{blue}{,} \int_0^1 \Dso d\psi_{\Tilde{\gamma}(s)}\Tilde{\vec{w}} \ ds}\rangle_L \ \int_0^1 \mathcal{R}^e\brac{\frac{\partial \psi^j_{\Tilde{\gamma}(s)}}{\partial p^\ell} \ \partial_e \Tilde{\vec{w}}^\ell} \ ds\label{1_almost_1}\\
    &+ \int_{\R^d} \langle H_{\abs{\nabla}}
    \brac{\vec{v}, \vec{v}^j\vec{v}} \wedge_L\int_0^1 \Dso d\psi_{\Tilde{\gamma}(s)}\Tilde{\vec{w}} \ ds\rangle_L \ \int_0^1 \mathcal{R}^e\brac{\frac{\partial \psi^j_{\Tilde{\gamma}(s)}}{\partial p^\ell} \ \partial_e \Tilde{\vec{w}}^\ell} \ ds \label{1_almost_2}\\
    &+ \int_{\R^d} \langle \vec{v}, H_{\abs{\nabla}}
    \brac{\vec{v}^j\vec{v} \wedge_L, \int_0^1 \Dso d\psi_{\Tilde{\gamma}(s)}\Tilde{\vec{w}} \ ds}\rangle_L \ \int_0^1 \mathcal{R}^e\brac{\frac{\partial \psi^j_{\Tilde{\gamma}(s)}}{\partial p^\ell} \ \partial_e \Tilde{\vec{w}}^\ell} \ ds\label{1_almost_3}\\
    &+ 2\int_{\R^d} \langle \abs{\nabla}\vec{v}, \vec{v}^j\vec{v} \wedge_L\int_0^1 \Dso d\psi_{\Tilde{\gamma}(s)}\Tilde{\vec{w}} \ ds\rangle_L \ \int_0^1 \mathcal{R}^e\brac{\frac{\partial \psi^j_{\Tilde{\gamma}(s)}}{\partial p^\ell} \ \partial_e \Tilde{\vec{w}}^\ell} \ ds \label{1_almost_4}.
\end{align} 
For \eqref{1_almost_1}, by \cref{riesz_estimate}, we have
\begin{align*}
    &\int_{\R^d} \langle \textcolor{blue}{H_{\abs{\nabla}}}\brac{\vec{v}, \vec{v}^j\vec{v} \wedge_L \textcolor{blue}{,} \int_0^1 \Dso d\psi_{\Tilde{\gamma}(s)}\Tilde{\vec{w}} \ ds}\rangle_L \ \int_0^1 \mathcal{R}^e\brac{\frac{\partial \psi^j_{\Tilde{\gamma}(s)}}{\partial p^\ell} \ \partial_e \Tilde{\vec{w}}^\ell} \ ds   \\
    \lesssim_{\Lambda} & \brac{\lVert \Dso^{1+\sigma}\vec{u} \rVert_{L^{\frac{2d}{2(1+\sigma)-1}}(\R^d)}^2 + \lVert \Dso^{1+\sigma}\vec{v} \rVert_{L^{\frac{2d}{2(1+\sigma)-1}}(\R^d)}^2}  \ \lVert \Dso \Tilde{\vec{w}} \rVert_{L^{2}(\R^d)}^2.
\end{align*}

For \eqref{1_almost_2}, we use \cref{lem:prod_rule_frac} and \cref{lem:gag_in} to get
\begin{align*}
    &\int_{\R^d} \langle H_{\abs{\nabla}}
    \brac{\vec{v}, \vec{v}^j\vec{v}} \wedge_L\int_0^1 \Dso d\psi_{\Tilde{\gamma}(s)}\Tilde{\vec{w}} \ ds\rangle_L \ \int_0^1 \mathcal{R}^e\brac{\frac{\partial \psi^j_{\Tilde{\gamma}(s)}}{\partial p^\ell} \ \partial_e \Tilde{\vec{w}}^\ell} \ ds\\
    \lesssim & \lVert H_{\abs{\nabla}}
    \brac{\vec{v}, \vec{v}^j\vec{v}} \rVert_{L^{2d}(\R^d)} \ \lVert \int_0^1 \Dso d\psi_{\Tilde{\gamma}(s)} \ ds \rVert_{L^{2d}(\R^d)} \ \lVert \Tilde{\vec{w}} \rVert_{L^{\frac{2d}{d-2}}(\R^d)} \ \lVert \Dso\Tilde{\vec{w}} \lVert_{L^{2}(\R^d)}\\
    \lesssim & \lVert \Dso^{1-\sigma}\vec{v} \rVert_{L^{\frac{2d}{1-\sigma}}(\R^d)} \ \lVert \Dso^{\sigma}\brac{\vec{v}^j\vec{v}} \rVert_{L^{\frac{2d}{\sigma}}(\R^d)} \ \lVert \int_0^1 \Dso d\psi_{\Tilde{\gamma}(s)} \ ds \rVert_{L^{2d}(\R^d)} \ \lVert \Dso\Tilde{\vec{w}} \lVert_{L^{2}(\R^d)}^2\\
    \lesssim & \lVert \Dso\vec{v} \rVert_{L^{2d}(\R^d)}^{1-\sigma} \ \lVert \Dso\brac{\vec{v}^j\vec{v}} \rVert_{L^{2d}(\R^d)}^{\sigma} \ \lVert \int_0^1 \Dso d\psi_{\Tilde{\gamma}(s)} \ ds \rVert_{L^{2d}(\R^d)} \ \lVert \Dso\Tilde{\vec{w}} \lVert_{L^{2}(\R^d)}^2\\
    \lesssim & \lVert \Dso\vec{v} \rVert_{L^{2d}(\R^d)}^{1-\sigma} \ \lVert \nabla\brac{\vec{v}^j\vec{v}} \rVert_{L^{2d}(\R^d)}^{\sigma} \ \lVert \int_0^1 \Dso d\psi_{\Tilde{\gamma}(s)} \ ds \rVert_{L^{2d}(\R^d)} \ \lVert \Dso\Tilde{\vec{w}} \lVert_{L^{2}(\R^d)}^2\\
    \lesssim & \lVert \Dso\vec{v} \rVert_{L^{2d}(\R^d)}^{1-\sigma} \ \lVert \nabla \vec{v} \rVert_{L^{2d}(\R^d)}^{\sigma} \ \lVert \int_0^1 \Dso d\psi_{\Tilde{\gamma}(s)} \ ds \rVert_{L^{2d}(\R^d)} \ \lVert \Dso\Tilde{\vec{w}} \lVert_{L^{2}(\R^d)}^2\\
    \lesssim & \lVert \Dso\vec{v} \rVert_{L^{2d}(\R^d)}^{1-\sigma} \ \lVert \Dso \vec{v} \rVert_{L^{2d}(\R^d)}^{\sigma} \ \lVert \int_0^1 \Dso d\psi_{\Tilde{\gamma}(s)} \ ds \rVert_{L^{2d}(\R^d)} \ \lVert \Dso\Tilde{\vec{w}} \lVert_{L^{2}(\R^d)}^2\\
    \lesssim & \lVert \Dso\vec{v} \rVert_{L^{2d}(\R^d)} \ \lVert \int_0^1 \Dso d\psi_{\Tilde{\gamma}(s)} \ ds \rVert_{L^{2d}(\R^d)} \ \lVert \Dso\Tilde{\vec{w}} \lVert_{L^{2}(\R^d)}^2\\
    \lesssim_{\Lambda} & \brac{\lVert \Dso\vec{u} \rVert_{L^{2d,2}(\R^d)}^2 + \lVert \Dso\vec{v} \rVert_{L^{2d,2}(\R^d)}^2} \ \lVert \Dso\Tilde{\vec{w}} \rVert_{L^{2}(\R^d)}^2.
\end{align*}
Next, for \eqref{1_almost_3}, we use \cref{riesz_estimate} to get
\begin{align*}
    &\int_{\R^d} \langle \vec{v}, H_{\abs{\nabla}}
    \brac{\vec{v}^j\vec{v} \wedge_L, \int_0^1 \mathcal{R}^e\brac{\frac{\partial \psi^j_{\Tilde{\gamma}(s)}}{\partial p^\ell} \ \partial_e \Tilde{\vec{w}}^\ell} \ ds}\rangle_L \ \int_0^1 \frac{\partial \psi_{\Tilde{\gamma}(s)}}{\partial p^\ell} \Dso \Tilde{\vec{w}}^\ell \ ds\\
    \lesssim_{\Lambda} & \brac{\lVert \Dso^{1+\sigma}\vec{u} \rVert_{L^{\frac{2d}{2(1+\sigma)-1}}(\R^d)}^2 + \lVert \Dso^{1+\sigma}\vec{v} \rVert_{L^{\frac{2d}{2(1+\sigma)-1}}(\R^d)}^2}  \ \lVert \Dso \Tilde{\vec{w}} \rVert_{L^{2}(\R^d)}^2.
\end{align*}

For \eqref{1_almost_4},
\begin{align*}
    &\int_{\R^d} \langle \abs{\nabla}\vec{v}, \vec{v}^j\vec{v} \wedge_L\int_0^1 \Dso d\psi_{\Tilde{\gamma}(s)}\Tilde{\vec{w}} \ ds\rangle_L \ \int_0^1 \mathcal{R}^e\brac{\frac{\partial \psi^j_{\Tilde{\gamma}(s)}}{\partial p^\ell} \ \partial_e \Tilde{\vec{w}}^\ell} \ ds\\
    \lesssim_{\Lambda} & \brac{\lVert \Dso\vec{u} \rVert_{L^{2d,2}(\R^d)}^2 + \lVert \Dso\vec{v} \rVert_{L^{2d,2}(\R^d)}^2} \ \lVert \Dso\Tilde{\vec{w}} \rVert_{L^{2}(\R^d)}^2.
\end{align*}
Hence, we conclude.
\end{proof}
Continuing in our pursuit of estimating \eqref{eq:last_2}, we estimate \eqref{split_2}.
\begin{lemma}
For $d \geq 3$ and any $\sigma \in (0,1/2)$ we have 
\begin{align*}
    &\abs{\int_{\R^d} \sideset{}{'}{\sum}_{j=1}^3 \sum_{\ell=1}^m \langle \vec{v}^j H_{\abs{\nabla}}\brac{\vec{v}, \int_0^1 \mathcal{R}^e\brac{\frac{\partial \psi^j_{\Tilde{\gamma}(s)}}{\partial p^\ell} \ \partial_e \Tilde{\vec{w}}^\ell} \ ds}, \vec{v} \wedge_L \int_0^1 d \psi_{\Tilde{\gamma}(s)} \Dso \Tilde{\vec{w}} \ ds \rangle_L} \\
    &\lesssim_{\Lambda} \lVert \abs{\nabla} \Tilde{\vec{w}}\rVert_{L^{2}(\R^d)}^2\, \brac{\lVert\abs{\nabla}^{1+\sigma} \vec{u}\rVert_{L^\frac{2d}{2(1+\sigma)-1}(\mathbb{R}^d)}^2 + \lVert\abs{\nabla}^{1+\sigma} \vec{v}\rVert_{L^\frac{2d}{2(1+\sigma)-1}(\mathbb{R}^d)}^2}.
\end{align*}
\end{lemma}
\begin{proof}
Using \cref{lem:comm_int}, we have that
\begin{align*}
    &H_{\Dso}(\vec{v},\int_0^1 \mathcal{R}^e\brac{\frac{\partial \psi^j_{\Tilde{\gamma}(s)}}{\partial p^\ell} \ \partial_e \Tilde{\vec{w}}^\ell} \ ds)(x)\\
    &= \int_{\mathbb{R}^n} \frac{\big(\vec{v}(x) - \vec{v}(y)\big)\big(\int_0^1 \mathcal{R}^e\brac{\frac{\partial \psi^j_{\Tilde{\gamma}(s)}}{\partial p^\ell} \ \partial_e \Tilde{\vec{w}}^\ell} \ ds(x) - \int_0^1 \mathcal{R}^e\brac{\frac{\partial \psi^j_{\Tilde{\gamma}(s)}}{\partial p^\ell} \ \partial_e \Tilde{\vec{w}}^\ell} \ ds(y)\big)}{ |x-y|^{d+1}} \ dy.
\end{align*}
Therefore, 
\begin{align}
    &\sideset{}{'}{\sum}_{j=1}^3 \sum_{\ell=1}^m \vec{v}^j(x)H_{\Dso}\brac{\vec{v},\int_0^1 \mathcal{R}^e\brac{\frac{\partial \psi^j_{\Tilde{\gamma}(s)}}{\partial p^\ell} \ \partial_e \Tilde{\vec{w}}^\ell} \ ds} \nonumber\\
    =& c \int_{\mathbb{R}^d} \frac{\big(\vec{v}(x) - \vec{v}(y)\big)\Big(\big\langle \vec{v}(x),\int_0^1 \mathcal{R}^e \brac{d\psi_{\Tilde{\gamma}(s)} \partial_e \Tilde{\vec{w}}} (x) \ ds\big\rangle_L - \big\langle \vec{v}(x),\int_0^1 \mathcal{R}^e \brac{d\psi_{\Tilde{\gamma}(s)} \partial_e \Tilde{\vec{w}}} (y) \ ds\big\rangle_L\Big)}{|x-y|^{d+1}} \ dy\nonumber\\
    =& c \int_{\mathbb{R}^d} \frac{\big(\vec{v}(x) - \vec{v}(y)\big)\Big(\big\langle \vec{v}(x),\int_0^1 \mathcal{R}^e \brac{d\psi_{\Tilde{\gamma}(s)} \partial_e \Tilde{\vec{w}}} (x) \ ds\big\rangle_L - \big\langle \vec{v}(y),\int_0^1 \mathcal{R}^e \brac{d\psi_{\Tilde{\gamma}(s)} \partial_e \Tilde{\vec{w}}} (y) \ ds\big\rangle_L\Big)}{|x-y|^{d+1}} \ dy\nonumber\\
    &-c\int_{\mathbb{R}^d}\frac{\big(\vec{v}(x) - \vec{v}(y)\big)\big\langle \vec{v}(x) - \vec{v}(y),\int_0^1 \mathcal{R}^e \brac{d\psi_{\Tilde{\gamma}(s)} \partial_e \Tilde{\vec{w}}} (y) \ ds \big\rangle_L}{ |x-y|^{d+1}} \ dy\nonumber\\
    =& H_{\Dso}\big(\vec{v}, \langle \vec{v}, \int_0^1 \mathcal{R}^e \brac{d\psi_{\Tilde{\gamma}(s)} \partial_e \Tilde{\vec{w}} } \ ds \rangle_L\big) \label{nice_1}\\
    &-c\int_{\mathbb{R}^d}\frac{\big(\vec{v}(x) - \vec{v}(y)\big)\big\langle \vec{v}(x) - \vec{v}(y),\int_0^1 \mathcal{R}^e \brac{d\psi_{\Tilde{\gamma}(s)} \partial_e \Tilde{\vec{w}}} (y) \ ds\big\rangle_L}{ |x-y|^{d+1}} \ dy \label{nice_2}.
\end{align}

Observe that 
\begin{align*}
    c\sum_{e=1}^d\mathcal{R}^e\brac{\int_0^1 d\psi_{\Tilde{\gamma}(s)} \partial_e \Tilde{\vec{w}} \ ds} 
    &= c\sum_{e=1}^d \mathcal{R}^e \partial_e \brac{\int_0^1 d\psi_{\Tilde{\gamma}(s)}\Tilde{\vec{w}} \ ds} - c\sum_{e=1}^d \mathcal{R}^e\brac{\int_0^1 d^2\psi_{\Tilde{\gamma}(s)} \partial_e \Tilde{\gamma}(s) \Tilde{\vec{w}} \ ds}\\
    &= \Dso \vec{w} - c\sum_{e=1}^d \mathcal{R}^e\brac{\int_0^1 d^2\psi_{\Tilde{\gamma}(s)} \partial_e \Tilde{\gamma}(s) \Tilde{\vec{w}} \ ds}\\
    &= \Dso \vec{w}\\
    &+\int_0^1 sc \sum_{e=1}^{d}\mathcal{R}^e \brac{d^2 \psi_{\Tilde{\gamma}(s)} d\varphi_{\vec{u}} \partial_e \vec{u} \ \Tilde{\vec{w}}} \ ds\\
    &+ \int_0^1 (1-s)c \sum_{e=1}^{d}\mathcal{R}^e \brac{d^2 \psi_{\Tilde{\gamma}(s)} d\varphi_{\vec{v}} \partial_e \vec{v} \ \Tilde{\vec{w}}}.
\end{align*}
Hence, for \eqref{nice_1}
\begin{align}
    &\left\lVert H_{\Dso}\brac{\vec{v},\left\langle \vec{v},c\sum_{e=1}^d\mathcal{R}^e\brac{\int_0^1 d\psi_{\Tilde{\gamma}(s)} \partial_e \Tilde{\vec{w}} \ ds} \right\rangle_L} \right\rVert_{L^2(\mathbb{R}^d)}\nonumber\\
    \lesssim & \lVert H_{\Dso}(\vec{v},\langle \vec{v},\Dso\vec{w}\rangle_L) \rVert_{L^2(\mathbb{R}^d)}\label{2_comm_lem_a}\\
    &+ \sum_{\ell=1}^m \sum_{e=1}^d\lVert H_{\Dso}(\vec{v},\langle \vec{v}, \mathcal{R}^e\brac{\int_0^1 d^2\psi_{\Tilde{\gamma}(s)} \partial_e \Tilde{\gamma}(s) \Tilde{\vec{w}} \ ds} \rangle_L) \rVert_{L^2(\mathbb{R}^d)}\label{2_comm_lem_b}.
\end{align}
 
For \eqref{2_comm_lem_a}, by \cref{prod_frac_vec}, we have
\begin{align*}
    &\lVert H_{\Dso}\brac{\vec{v}, \langle \vec{v},\Dso\vec{w} \rangle_L} \rVert_{L^{2}(\R^d)}\\
    \lesssim_{\Lambda} & \brac{\Dso^{1+\sigma}\vec{u} \rVert_{L^{\frac{2d}{2(1+\sigma) - 1}}(\R^d)}^2 + \lVert \Dso^{1+\sigma}\vec{v} \rVert_{L^{\frac{2d}{2(1+\sigma) - 1}}(\R^d)}^2}  \ \lVert \Dso\Tilde{\vec{w}} \rVert_{L^{2}(\R^d)}.
\end{align*}
For \eqref{2_comm_lem_b}, using \cref{comm_lem}, we have
\begin{align}
    &\lVert H_{\Dso}(\vec{v},\langle \vec{v}, \int_0^1 \mathcal{R}^e \brac{d^2 \psi_{\Tilde{\gamma}(s)} d\varphi_{\vec{v}} \partial_e \vec{v} \ \Tilde{\vec{w}}} \ ds \rangle_L) \rVert_{L^2(\mathbb{R}^d)}\nonumber\\
    \lesssim & \lVert H_{\Dso}\brac{\vec{v}\langle \vec{v},\int_0^1 \mathcal{R}^e \brac{d^2 \psi_{\Tilde{\gamma}(s)} d\varphi_{\vec{v}} \partial_e \vec{v} \ \Tilde{\vec{w}}} \ ds \rangle_L} \rVert_{L^2(\mathbb{R}^d)}\label{comm_lem_x}\\
    &+ \lVert H_{\Dso}\brac{\vec{v},\langle \vec{v}},\int_0^1 \mathcal{R}^e \brac{d^2 \psi_{\Tilde{\gamma}(s)} d\varphi_{\vec{v}} \partial_e \vec{v} \ \Tilde{\vec{w}}} \ ds \rangle_L\rVert_{L^2(\mathbb{R}^d)}\label{comm_lem_y}\\
    &+ \lVert\vec{v}\langle H_{\Dso}\brac{\vec{v},\int_0^1 \mathcal{R}^e \brac{d^2 \psi_{\Tilde{\gamma}(s)} d\varphi_{\vec{v}} \partial_e \vec{v} \ \Tilde{\vec{w}}} \ ds}\rangle_L \rVert_{L^2(\mathbb{R}^d)}\label{comm_lem_z}.
\end{align}
With small $\sigma \in (0,\frac{1}{2}]$, we use \cref{riesz_estimate} on \eqref{comm_lem_x} to get
\begin{align*}
    &\lVert H_{\Dso}\brac{\vec{v}\langle \vec{v},\int_0^1 \mathcal{R}^e \brac{d^2 \psi_{\Tilde{\gamma}(s)} d\varphi_{\vec{v}} \partial_e \vec{v} \ \Tilde{\vec{w}}} \ ds \rangle_L} \rVert_{L^2(\mathbb{R}^d)}\\
    \lesssim_{\Lambda} & \brac{\lVert \Dso^{1+\sigma}\vec{u} \rVert_{L^{\frac{2d}{2(1+\sigma)-1}}(\R^d)}^2 + \lVert \Dso^{1+\sigma}\vec{v} \rVert_{L^{\frac{2d}{2(1+\sigma)-1}}(\R^d)}^2}  \ \lVert \Dso \Tilde{\vec{w}} \rVert_{L^{2}(\R^d)}.
\end{align*}

For \eqref{comm_lem_y}, we use \cref{lem:prod_rule} and \cref{lem:gag_in} to get
\begin{align*}
    &\lVert H_{\Dso}\brac{\vec{v},\langle \vec{v}},\int_0^1 \mathcal{R}^e \brac{d^2 \psi_{\Tilde{\gamma}(s)} d\varphi_{\vec{v}} \partial_e \vec{v} \ \Tilde{\vec{w}}} \ ds \rangle_L\rVert_{L^2(\mathbb{R}^d)}\\
    \lesssim & \sum_{\ell=1}^m\sum_{j=1}^3\lVert H_{\Dso}\brac{\vec{v}, \vec{v}^j} \rVert_{L^{2d}(\R^d)} \ \lVert \int_0^1 \mathcal{R}^e \brac{d^2 \psi_{\Tilde{\gamma}(s)} d\varphi_{\vec{v}} \partial_e \vec{v} \ \Tilde{\vec{w}}} \ ds \rVert_{L^{\frac{2d}{d-1}}(\R^d)}\\
    \lesssim_{\Lambda} & \sum_{\ell=1}^m\sum_{j=1}^3\lVert H_{\Dso}\brac{\vec{v}, \vec{v}^j} \rVert_{L^{2d}(\R^d)} \ \lVert \int_0^1 d^2 \psi_{\Tilde{\gamma}(s)} d\varphi_{\vec{v}} \partial_e \vec{v} \ \Tilde{\vec{w}} \ ds \rVert_{L^{\frac{2d}{d-1}}(\R^d)}\\
    \lesssim_{\Lambda} & \sum_{\ell=1}^m\sum_{j=1}^3\lVert H_{\Dso}\brac{\vec{v}, \vec{v}^j} \rVert_{L^{2d}(\R^d)} \ \lVert \partial_e \vec{v} \ \Tilde{\vec{w}} \rVert_{L^{\frac{2d}{d-1}}(\R^d)}\\
    \lesssim & \sum_{\ell=1}^m\sum_{j=1}^3\lVert H_{\Dso}\brac{\vec{v}, \vec{v}^j} \rVert_{L^{2d}(\R^d)} \ \lVert \partial_e \vec{v} \rVert_{L^{2d}(\R^d)} \ \lVert \Tilde{\vec{w}} \rVert_{L^{\frac{2d}{d-2}}(\R^d)}\\
    \lesssim_{\Lambda} & \sum_{\ell=1}^m\sum_{j=1}^3\lVert \Dso^{1-\sigma}\vec{v} \rVert_{L^{\frac{2d}{1-\sigma}}(|R^d)} \ \lVert \Dso^\sigma\vec{v}^j \rVert_{L^{\frac{2d}{\sigma}}(\R^d)} \ \lVert \partial_e \vec{v} \rVert_{L^{2d}(\R^d)} \ \lVert \Tilde{\vec{w}} \rVert_{L^{\frac{2d}{d-2}}(\R^d)}\\
    \lesssim_{\Lambda} & \sum_{\ell=1}^m\sum_{j=1}^3\lVert \Dso\vec{v} \rVert_{L^{2d}(|R^d)}^{1-\sigma} \ \lVert \Dso\vec{v}^j \rVert_{L^{2d}(\R^d)}^\sigma \ \lVert \Dso \vec{v} \rVert_{L^{2d}(\R^d)} \ \lVert \Dso\Tilde{\vec{w}} \rVert_{L^{2d}(\R^d)}\\
    \lesssim_{\Lambda} & \brac{\lVert \Dso\vec{u} \rVert_{L^{2d}(|R^d)}^2 + \lVert \Dso\vec{v} \rVert_{L^{2d}(\R^d)}^2} \ \lVert \Dso\Tilde{\vec{w}} \rVert_{L^{2}(\R^d)}.
\end{align*}

For \eqref{comm_lem_z}, we use \cref{riesz_estimate}
\begin{align*}
    &\lVert\vec{v}\langle H_{\Dso}\brac{\vec{v},\int_0^1 \mathcal{R}^e \brac{d^2 \psi_{\Tilde{\gamma}(s)} d\varphi_{\vec{v}} \partial_e \vec{v} \ \Tilde{\vec{w}}} \ ds }\rangle_L \rVert_{L^2(\mathbb{R}^d)}\\
    \lesssim_{\Lambda} & \brac{\lVert \Dso^{1+\sigma}\vec{u} \rVert_{L^{\frac{2d}{2(1+\sigma)-1}}(\R^d)}^2 + \lVert \Dso^{1+\sigma}\vec{v} \rVert_{L^{\frac{2d}{2(1+\sigma)-1}}(\R^d)}^2}  \ \lVert \Dso \Tilde{\vec{w}} \rVert_{L^{2}(\R^d)}.
\end{align*}
For \eqref{nice_2}, we can use \cref{triple_comm_int} by taking $\alpha_1 = \alpha_2 = \frac{1 + \beta}{2}$ for any $\beta \in (0,1)$ and $p_1 = p_2 = \frac{2d}{\beta}$, $p_3 = \frac{2d}{d - 2\beta}$, we have
\begin{align*}
    &\brac{\int_{\mathbb{R}^d}\abs{\int_{\mathbb{R}^d} \frac{\big(\vec{v}(x) - \vec{v}(y)\big)\big\langle \vec{v}(x) - \vec{v}(y),\int_0^1 \mathcal{R}^e \brac{d\psi_{\Tilde{\gamma}(s)} \partial_e \Tilde{\vec{w}}} (y) \ ds \big\rangle_L}{ |x-y|^{d+1}} \ dy}^2 }^{\frac{1}{2}}\\
    &\lesssim \lVert \Ds{\frac{1 + \beta}{2}}\vec{v}\rVert_{L^\frac{2d}{\beta}(\mathbb{R}^d)} \ \lVert \Ds{\frac{1 + \beta}{2}}\vec{v} \rVert_{L^\frac{2d}{\beta}(\mathbb{R}^d)} \ \lVert \int_0^1 \mathcal{R}^e \brac{d\psi_{\Tilde{\gamma}(s)} \partial_e \Tilde{\vec{w}}} \ ds \rVert_{L^2(\mathbb{R}^d)}\\
    &\lesssim_{\Lambda} \lVert \Dso{\vec{v}}\rVert_{L^{2d}(\mathbb{R}^d)}^2 \ \lVert \Dso \Tilde{\vec{w}} \rVert_{L^2(\mathbb{R}^d)}.
\end{align*}
Therefore,
\begin{align*}
    &\abs{\int_{\R^d} \sideset{}{'}{\sum}_{j=1}^3 \langle \vec{v}^jH_{\abs{\nabla}}\brac{\vec{v}, \int_0^1 \mathcal{R}^e\brac{\frac{\partial \psi^j_{\Tilde{\gamma}(s)}}{\partial p^\ell} \ \partial_e \Tilde{\vec{w}}^\ell} \ ds11}, \vec{v} \wedge_L\int_0^1 d\psi_{\Tilde{\gamma}(s)} \abs\nabla \Tilde{\vec{w}} \ ds \rangle_L } \\
    &\lesssim_{\Lambda} \lVert \abs{\nabla} \Tilde{\vec{w}}\rVert_{L^{2}(\R^d)}^2\, \brac{\lVert\abs{\nabla}^{1+\sigma} \vec{u}\rVert_{L^\frac{2d}{2(1+\sigma)-1}(\mathbb{R}^d)}^2 + \lVert\abs{\nabla}^{1+\sigma} \vec{v}\rVert_{L^\frac{2d}{2(1+\sigma)-1}(\mathbb{R}^d)}^2}
\end{align*}
for small $\sigma \in (0,\frac{1}{2}]$.
\end{proof}
\medskip
After we estimate \eqref{split_3}, the estimate for the second term of \eqref{split} will be complete.
 \begin{lemma}
For $d \geq 3$ and any $\sigma \in (0,1/2]$ we have 
\begin{align*}
    &\abs{\int_{\R^d} \sideset{}{'}{\sum}_{j=1}^3 \sum_{\ell=1}^m\langle \vec{v}^jH_{\abs{\nabla}}\brac{\vec{v}, \int_0^1 \mathcal{R}^e\brac{\frac{\partial \psi^j_{\Tilde{\gamma}(s)}}{\partial p^\ell} \ \partial_e \Tilde{\vec{w}}^\ell} \ ds}, \vec{v} \wedge_L\int_0^1 H_{\Dso}\brac{d\psi_{\Tilde{\gamma}(s)}, \Tilde{\vec{w}}} \ ds \rangle_L} \\
    &\lesssim_{\Lambda} \lVert \abs{\nabla} \Tilde{\vec{w}}\rVert_{L^{2}(\R^d)}^2\, \brac{\lVert\abs{\nabla}^{1+\sigma} \vec{u}\rVert_{L^\frac{2d}{2(1+\sigma)-1}(\mathbb{R}^d)}^2 + \lVert\abs{\nabla}^{1+\sigma} \vec{v}\rVert_{L^\frac{2d}{2(1+\sigma)-1}(\mathbb{R}^d)}^2}.
\end{align*}
\end{lemma}
\begin{proof}
Using \cref{comm_lemm}, we have
\begin{align}
    &\int_{\R^d} \langle \vec{v}^jH_{\abs{\nabla}}\brac{\vec{v}, \int_0^1 \mathcal{R}^e\brac{\frac{\partial \psi^j_{\Tilde{\gamma}(s)}}{\partial p^\ell} \ \partial_e \Tilde{\vec{w}}^\ell} \ ds}, \vec{v} \wedge_L\int_0^1 H_{\Dso}\brac{d\psi_{\Tilde{\gamma}(s)}, \Tilde{\vec{w}}} \ ds \rangle_L \nonumber\\
    &= \int_{\R^d} \langle H_{\abs{\nabla}}\brac{\vec{v}, \vec{v}^j\vec{v} \wedge_L\int_0^1 H_{\Dso}\brac{d\psi_{\Tilde{\gamma}(s)}, \Tilde{\vec{w}}} \ ds}\rangle_L \ \int_0^1 \mathcal{R}^e\brac{\frac{\partial \psi^j_{\Tilde{\gamma}(s)}}{\partial p^\ell} \ \partial_e \Tilde{\vec{w}}^\ell} \ ds\nonumber\\
    &+2\int_{\R^d} \langle \abs{\nabla}\vec{v}, \vec{v}^j\vec{v} \wedge_L\int_0^1 H_{\Dso}\brac{d\psi_{\Tilde{\gamma}(s)}, \Tilde{\vec{w}}} \ ds\rangle_L \ \int_0^1 \mathcal{R}^e\brac{\frac{\partial \psi^j_{\Tilde{\gamma}(s)}}{\partial p^\ell} \ \partial_e \Tilde{\vec{w}}^\ell} \ ds \nonumber\\
    &= \int_{\R^d} \langle \textcolor{blue}{H_{\abs{\nabla}}}\brac{\vec{v}, \vec{v}^j\vec{v} \wedge_L \textcolor{blue}{,} \int_0^1 H_{\Dso}\brac{d\psi_{\Tilde{\gamma}(s)}, \Tilde{\vec{w}}} \ ds}\rangle_L \ \int_0^1 \mathcal{R}^e\brac{\frac{\partial \psi^j_{\Tilde{\gamma}(s)}}{\partial p^\ell} \ \partial_e \Tilde{\vec{w}}^\ell} \ ds\label{almost_3_1}\\
    &+ \int_{\R^d} \langle H_{\abs{\nabla}}
    \brac{\vec{v}, \vec{v}^j\vec{v}} \wedge_L\int_0^1 H_{\Dso}\brac{d\psi_{\Tilde{\gamma}(s)}, \Tilde{\vec{w}}} \ ds\rangle_L \ \int_0^1 \mathcal{R}^e\brac{\frac{\partial \psi^j_{\Tilde{\gamma}(s)}}{\partial p^\ell} \ \partial_e \Tilde{\vec{w}}^\ell} \ ds\label{almost_3_2}\\
    &+ \int_{\R^d} \langle \vec{v}, H_{\abs{\nabla}}
    \brac{\vec{v}^j\vec{v} \wedge_L, \int_0^1 H_{\Dso}\brac{d\psi_{\Tilde{\gamma}(s)}, \Tilde{\vec{w}}} \ ds}\rangle_L \ \int_0^1 \mathcal{R}^e\brac{\frac{\partial \psi^j_{\Tilde{\gamma}(s)}}{\partial p^\ell} \ \partial_e \Tilde{\vec{w}}^\ell} \ ds\label{almost_3_3}\\
    &+ 2\int_{\R^d} \langle \abs{\nabla}\vec{v}, \vec{v}^j\vec{v} \wedge_L\int_0^1 H_{\Dso}\brac{d\psi_{\Tilde{\gamma}(s)}, \Tilde{\vec{w}}} \ ds\rangle_L \ \int_0^1 \mathcal{R}^e\brac{\frac{\partial \psi^j_{\Tilde{\gamma}(s)}}{\partial p^\ell} \ \partial_e \Tilde{\vec{w}}^\ell} \ ds  \label{almost_4}.
\end{align} 
For \eqref{almost_3_1}, we use \cref{lem:prod_rule} and \cref{lem:prod_rule_2} for $\gamma \in (0,1)$ such that $1+\sigma - \gamma \in (0,1)$ to get
\begin{align*}
    &\int_{\R^d} \langle \textcolor{blue}{H_{\abs{\nabla}}}\brac{\vec{v}, \vec{v}^j\vec{v} \wedge_L \textcolor{blue}{,} \int_0^1 H_{\Dso}\brac{d\psi_{\Tilde{\gamma}(s)}, \Tilde{\vec{w}}} \ ds}\rangle_L \ \int_0^1 \mathcal{R}^e\brac{\frac{\partial \psi^j_{\Tilde{\gamma}(s)}}{\partial p^\ell} \ \partial_e \Tilde{\vec{w}}^\ell} \ ds\\
    \lesssim & \lVert \langle \textcolor{blue}{H_{\abs{\nabla}}}\brac{\vec{v}, \vec{v}^j\vec{v} \wedge_L \textcolor{blue}{,} \int_0^1 H_{\Dso}\brac{d\psi_{\Tilde{\gamma}(s)}, \Tilde{\vec{w}}} \ ds}\rangle_L \rVert_{L^{2}(\R^d)} \ \lVert \Dso \Tilde{\vec{w}} \rVert_{L^{2}(\R^d)}\\
    \lesssim & \lVert \Dso^{1-\sigma}\vec{v} \rVert_{L^{\frac{2d}{2(1-\sigma)-1}}(\R^d)} \ \lVert \Dso^\sigma \brac{\int_0^1 H_{\Dso}\brac{d\psi_{\Tilde{\gamma}(s)}, \Tilde{\vec{w}}} \ ds} \lVert_{L^{\frac{2d}{d + 2\sigma -1}}(\R^d)} \ \lVert \Dso \Tilde{\vec{w}} \rVert_{L^{2}(\R^d)}\\
    \lesssim & \lVert \Dso^{1-\sigma}\vec{v} \rVert_{L^{\frac{2d}{2(1-\sigma)-1}}(\R^d)} \ \lVert \int_0^1 \Dso^{\gamma} \frac{\partial \psi_{\Tilde{\gamma}(s)}}{\partial p^\ell} \ ds \rVert_{L^{\frac{2d}{2\gamma-1}}(\R^d)} \ \lVert \Dso^{1+\sigma - \gamma} \Tilde{\vec{w}}^\ell \lVert_{L^{\frac{2d}{d - 2(1+\sigma -\gamma)}}(\R^d)} \ \lVert \Dso\Tilde{\vec{w}} \lVert_{L^{2}(\R^d)}\\
    \lesssim_{\Lambda} & \brac{\lVert \Dso \vec{u} \rVert_{L^{2d}(\R^d)}^2 + \lVert \Dso \vec{v} \rVert_{L^{2d}(\R^d)}^2} \ \lVert \Dso\Tilde{\vec{w}} \lVert_{L^{2}(\R^d)}^2.
\end{align*}
For \eqref{almost_3_2}, we use \cref{infty_bound} and \cref{lem:prod_rule_2} to get
\begin{align*}
    &\int_{\R^d} \langle H_{\abs{\nabla}}
    \brac{\vec{v}, \vec{v}^j\vec{v}} \wedge_L\int_0^1 H_{\Dso}\brac{d\psi_{\Tilde{\gamma}(s)}, \Tilde{\vec{w}}} \ ds\rangle_L \ \int_0^1 \mathcal{R}^e\brac{\frac{\partial \psi^j_{\Tilde{\gamma}(s)}}{\partial p^\ell} \ \partial_e \Tilde{\vec{w}}^\ell} \ ds\\
    \lesssim & \lVert H_{\abs{\nabla}}
    \brac{\vec{v}, \vec{v}^j\vec{v}} \rVert_{L^\infty(\R^d)} \ \lVert \int_0^1 H_{\Dso}\brac{d\psi_{\Tilde{\gamma}(s)}, \Tilde{\vec{w}}} \ ds \rVert_{L^{\frac{2d}{d-1}}(\R^d)} \ \lVert \Dso\Tilde{\vec{w}} \lVert_{L^{2}(\R^d)}\\
    \lesssim & \lVert \Dso\vec{v} \rVert_{L^{2d,2}(\R^d)} \ \lVert \Dso\brac{\vec{v}^j\vec{v}} \rVert_{L^{2d,2}(\R^d)} \ \lVert \int_0^1 H_{\Dso}\brac{d\psi_{\Tilde{\gamma}(s)}, \Tilde{\vec{w}}} \ ds \rVert_{L^{\frac{2d}{d-1}}(\R^d)} \ \lVert \Dso\Tilde{\vec{w}} \lVert_{L^{2}(\R^d)}\\ 
    \lesssim & \lVert \Dso\vec{v} \rVert_{L^{2d,2}(\R^d)} \ \lVert \int_0^1 H_{\Dso}\brac{d\psi_{\Tilde{\gamma}(s)}, \Tilde{\vec{w}}} \ ds \rVert_{L^{\frac{2d}{d-1}}(\R^d)} \ \lVert \Dso\Tilde{\vec{w}} \lVert_{L^{2}(\R^d)}\\ 
    \lesssim & \lVert \Dso\vec{v} \rVert_{L^{2d,2}(\R^d)} \ \lVert \int_0^1 \Dso^{\gamma} d\psi_{\Tilde{\gamma}(s)}\ ds \rVert_{L^{\frac{2d}{2\gamma-1}}(\R^d)} \ \lVert \Dso^{1+\sigma - \gamma} \Tilde{\vec{w}}^\ell \lVert_{L^{\frac{2d}{d - 2(1+\sigma -\gamma)}}(\R^d)} \ \lVert \Dso\Tilde{\vec{w}} \lVert_{L^{2}(\R^d)}\\ 
    \lesssim_{\Lambda} & \brac{\lVert \Dso\vec{u} \rVert_{L^{2d,2}(\R^d)}^2 + \lVert \Dso\vec{v} \rVert_{L^{2d,2}(\R^d)}^2} \ \lVert \Dso\Tilde{\vec{w}} \rVert_{L^{2}(\R^d)}^2.
\end{align*}
Next, for \eqref{almost_3_3}, we use \cref{lem:prod_rule_frac} to get
\begin{align*}
    &\int_{\R^d} \langle \vec{v}, H_{\abs{\nabla}}
    \brac{\vec{v}^j\vec{v} \wedge_L, \int_0^1 H_{\Dso}\brac{d\psi_{\Tilde{\gamma}(s)}, \Tilde{\vec{w}}} \ ds}\rangle_L \ \int_0^1 \mathcal{R}^e\brac{\frac{\partial \psi^j_{\Tilde{\gamma}(s)}}{\partial p^\ell} \ \partial_e \Tilde{\vec{w}}^\ell} \ ds\\
    \lesssim & \lVert H_{\abs{\nabla}}
    \brac{\vec{v}^j\vec{v} \wedge_L, \int_0^1 H_{\Dso}\brac{d\psi_{\Tilde{\gamma}(s)}, \Tilde{\vec{w}}} \ ds} \rVert_{L^{2}(\R^d)} \ \lVert \Dso \Tilde{\vec{w}}\rVert_{L^{2}(\R^d)}\\
    \lesssim & \lVert \abs{\nabla}
    \brac{\vec{v}^j\vec{v}} \rVert_{L^{\frac{2d}{2(1-\sigma)-1}}(\R^d)} \ \lVert \Dso^\sigma \brac{\int_0^1 H_{\Dso}\brac{d\psi_{\Tilde{\gamma}(s)}, \Tilde{\vec{w}}} \ ds} \lVert_{L^{\frac{2d}{d + 2\sigma -1}}(\R^d)} \ \lVert \Dso \Tilde{\vec{w}} \rVert_{L^{2}(\R^d)}\\
    \lesssim & \lVert \Dso^{1-\sigma}\vec{v} \rVert_{L^{\frac{2d}{2(1-\sigma)-1}}(\R^d)} \ \lVert \int_0^1 \Dso^{\gamma} d\psi_{\Tilde{\gamma}(s)} \ ds \rVert_{L^{\frac{2d}{2\gamma-1}}(\R^d)} \ \lVert \Dso^{1+\sigma - \gamma} \Tilde{\vec{w}}^\ell \lVert_{L^{\frac{2d}{d - 2(1+\sigma -\gamma)}}(\R^d)} \ \lVert \Dso\Tilde{\vec{w}} \lVert_{L^{2}(\R^d)}\\
    \lesssim_{\Lambda} & \brac{\lVert \Dso^{1+\sigma}\vec{u} \rVert_{L^{\frac{2d}{2(1-\sigma)-1}}(\R^d)}^2 + \lVert \Dso^{1 + \sigma} \vec{v} \rVert_{L^{\frac{2d}{2(1+\sigma) - 1}}(\R^d)}^2} \ \lVert \Dso\Tilde{\vec{w}} \lVert_{L^{2}(\R^d)}^2.
\end{align*}
Hence, we conclude.
\end{proof}
Since we have finished estimating the second term of \eqref{split}, we now expand the first term of \eqref{split}.
\begin{align}
    &\langle \vec{v}^jH_{\abs{\nabla}}\brac{\vec{v}^j, \int_0^1 \mathcal{R}^e \brac{d \psi_{\Tilde{\gamma}(s)} \partial_e \Tilde{\vec{w}}} \ ds}, \vec{v} \wedge_L\Dso\vec{w} \rangle_L\nonumber\\
    =& cs\sum_{e=1}^d\langle \vec{v}^jH_{\abs{\nabla}}\brac{\vec{v}^j, \int_0^1 \mathcal{R}^e \brac{d \psi_{\Tilde{\gamma}(s)} \partial_e \Tilde{\vec{w}}} \ ds}, \vec{v} \wedge_L \int_0^1 \mathcal{R}^e \brac{d^2 \psi_{\Tilde{\gamma}(s)} d\varphi_{\vec{u}} \partial_e \vec{u} \ \Tilde{\vec{w}}} \ ds \rangle_L\label{split_2_1}\\
    &+ c(1-s)\sum_{e=1}^d\langle \vec{v}^jH_{\abs{\nabla}}\brac{\vec{v}^j, \int_0^1 \mathcal{R}^e \brac{d \psi_{\Tilde{\gamma}(s)} \partial_e \Tilde{\vec{w}}} \ ds}, \vec{v} \wedge_L \int_0^1 \mathcal{R}^e \brac{d^2 \psi_{\Tilde{\gamma}(s)} d\varphi_{\vec{v}} \partial_e \vec{v} \ \Tilde{\vec{w}}} \ ds \rangle_L\label{split_2_2}\\
    &+ c\sum_{e=1}^d\langle \vec{v}^jH_{\abs{\nabla}}\brac{\vec{v}^j, \int_0^1 \mathcal{R}^e \brac{d \psi_{\Tilde{\gamma}(s)} \partial_e \Tilde{\vec{w}}} \ ds}, \vec{v} \wedge_L \int_0^1 \mathcal{R}^e \brac{d \psi_{\Tilde{\gamma}(s)} \partial_e \Tilde{\vec{w}}} \ ds \rangle_L\label{split_2_3}.
\end{align}
We now have three more terms to estimate. The following gives us estimates for \eqref{split_2_1} and \eqref{split_2_2}.
\begin{lemma}
For $d \geq 2$ and $\vec{f} \in \{\vec{u},\vec{v}\}$,
\begin{align*}
    &\abs{\int_{\R^d} \sideset{}{'}{\sum}_{j=1}^3 \Big\langle \vec{v}^jH_{\Dso}\brac{\vec{v}^j, \int_0^1 \mathcal{R}^e \brac{d \psi_{\Tilde{\gamma}(s)} \partial_e \Tilde{\vec{w}}} \ ds}, \vec{v} \wedge_L
    \int_0^1 \mathcal{R}^e \brac{d^2 \psi_{\Tilde{\gamma}(s)} d\varphi_{\vec{f}} \partial_e \vec{f} \ \Tilde{\vec{w}}} \ ds\Big\rangle_L}\\
     \lesssim_{\Lambda} & \brac{\lVert \Dso^{1+\sigma}\vec{u} \rVert_{L^{\frac{2d}{2(1+\sigma)-1}}(\R^d)}^2 + \lVert \Dso^{1+\sigma}\vec{v} \rVert_{L^{\frac{2d}{2(1+\sigma)-1}}(\R^d)}^2} \ \lVert \Dso \Tilde{\vec{w}} \rVert_{L^{2}(\R^d)}^2.
\end{align*}
\end{lemma}
\begin{proof}
Using \cref{comm_lemm}, we have
\begin{align}
    &\sideset{}{'}{\sum}_{j=1}^3\int_{\R^d} \langle \vec{v}^jH_{\abs{\nabla}}\brac{\vec{v}^j, \int_0^1 \mathcal{R}^e \brac{d \psi_{\Tilde{\gamma}(s)} \partial_e \Tilde{\vec{w}}} \ ds}, \vec{v} \wedge_L\int_0^1 \mathcal{R}^e \brac{d^2 \psi_{\Tilde{\gamma}(s)} d\varphi_{\vec{f}} \partial_e \vec{f} \ \Tilde{\vec{w}}} \ ds \rangle_L \nonumber\\
    &= \sideset{}{'}{\sum}_{j=1}^3\int_{\R^d} \langle H_{\abs{\nabla}}\brac{\vec{v}^j, \vec{v}^j\vec{v} \wedge_L\int_0^1 \mathcal{R}^e \brac{d^2 \psi_{\Tilde{\gamma}(s)} d\varphi_{\vec{f}} \partial_e \vec{f} \ \Tilde{\vec{w}}} \ ds}, \int_0^1 \mathcal{R}^e \brac{d \psi_{\Tilde{\gamma}(s)} \partial_e \Tilde{\vec{w}}} \ ds \rangle_L\nonumber\\
    &+2\sideset{}{'}{\sum}_{j=1}^3\int_{\R^d} \langle \abs{\nabla}\vec{v}^j \vec{v}^j\vec{v} \wedge_L\int_0^1 \mathcal{R}^e \brac{d^2 \psi_{\Tilde{\gamma}(s)} d\varphi_{\vec{f}} \partial_e \vec{f} \ \Tilde{\vec{w}}} \ ds, \ \int_0^1 \mathcal{R}^e \brac{d \psi_{\Tilde{\gamma}(s)} \partial_e \Tilde{\vec{w}}} \ ds \rangle_L \nonumber\\
    &= \sideset{}{'}{\sum}_{j=1}^3\int_{\R^d} \langle \textcolor{blue}{H_{\abs{\nabla}}}\brac{\vec{v}^j \vec{v}^j\vec{v} \wedge_L \textcolor{blue}{,} \int_0^1 \mathcal{R}^e \brac{d^2 \psi_{\Tilde{\gamma}(s)} d\varphi_{\vec{f}} \partial_e \vec{f} \ \Tilde{\vec{w}}} \ ds}, \int_0^1 \mathcal{R}^e \brac{d \psi_{\Tilde{\gamma}(s)} \partial_e \Tilde{\vec{w}}} \ ds \rangle_L\nonumber\\
    &+ \sideset{}{'}{\sum}_{j=1}^3\int_{\R^d} \langle H_{\abs{\nabla}}
    \brac{\vec{v}^j, \vec{v}^j\vec{v}} \wedge_L\int_0^1 \mathcal{R}^e \brac{d^2 \psi_{\Tilde{\gamma}(s)} d\varphi_{\vec{f}} \partial_e \vec{f} \ \Tilde{\vec{w}}} \ ds,  \int_0^1 \mathcal{R}^e \brac{d \psi_{\Tilde{\gamma}(s)} \partial_e \Tilde{\vec{w}}} \ ds\rangle_L \nonumber\\
    &+ \sideset{}{'}{\sum}_{j=1}^3\int_{\R^d} \langle \vec{v}^j H_{\abs{\nabla}}
    \brac{\vec{v}^j\vec{v} \wedge_L, \int_0^1 \mathcal{R}^e \brac{d^2 \psi_{\Tilde{\gamma}(s)} d\varphi_{\vec{f}} \partial_e \vec{f} \ \Tilde{\vec{w}}} \ ds}, \int_0^1 \mathcal{R}^e \brac{d \psi_{\Tilde{\gamma}(s)} \partial_e \Tilde{\vec{w}}} \ ds\rangle_L\nonumber\\
    &+ 2\sideset{}{'}{\sum}_{j=1}^3\int_{\R^d} \langle \abs{\nabla}\vec{v}^j \vec{v}^j\vec{v} \wedge_L\int_0^1 \mathcal{R}^e \brac{d^2 \psi_{\Tilde{\gamma}(s)} d\varphi_{\vec{f}} \partial_e \vec{f} \ \Tilde{\vec{w}}} \ ds, \ \int_0^1 \mathcal{R}^e \brac{d \psi_{\Tilde{\gamma}(s)} \partial_e \Tilde{\vec{w}}} \ ds\rangle_L \nonumber\\
    &= -\int_{\R^d} \langle \textcolor{blue}{H_{\abs{\nabla}}}\brac{\vec{v} \wedge_L \textcolor{blue}{,} \int_0^1 \mathcal{R}^e \brac{d^2 \psi_{\Tilde{\gamma}(s)} d\varphi_{\vec{f}} \partial_e \vec{f} \ \Tilde{\vec{w}}} \ ds}, \int_0^1 \mathcal{R}^e \brac{d \psi_{\Tilde{\gamma}(s)} \partial_e \Tilde{\vec{w}}} \ ds \rangle_L\label{almost_2_1}\\
    &+ \sideset{}{'}{\sum}_{j=1}^3\int_{\R^d} \langle H_{\abs{\nabla}}
    \brac{\vec{v}^j, \vec{v}^j\vec{v}} \wedge_L\int_0^1 \mathcal{R}^e \brac{d^2 \psi_{\Tilde{\gamma}(s)} d\varphi_{\vec{f}} \partial_e \vec{f} \ \Tilde{\vec{w}}} \ ds,  \int_0^1 \mathcal{R}^e \brac{d \psi_{\Tilde{\gamma}(s)} \partial_e \Tilde{\vec{w}}} \ ds\rangle_L \label{almost_2_2}\\
    &+ \sideset{}{'}{\sum}_{j=1}^3\int_{\R^d} \langle \vec{v}^j H_{\abs{\nabla}}
    \brac{\vec{v}^j\vec{v} \wedge_L, \int_0^1 \mathcal{R}^e \brac{d^2 \psi_{\Tilde{\gamma}(s)} d\varphi_{\vec{f}} \partial_e \vec{f} \ \Tilde{\vec{w}}} \ ds}, \int_0^1 \mathcal{R}^e \brac{d \psi_{\Tilde{\gamma}(s)} \partial_e \Tilde{\vec{w}}} \ ds\rangle_L\label{almost_2_3}\\
    &+ 2\int_{\R^d} \langle \abs{\nabla}\vec{v}, \vec{v} \rangle_L \ \langle \vec{v} \wedge_L\int_0^1 \mathcal{R}^e \brac{d^2 \psi_{\Tilde{\gamma}(s)} d\varphi_{\vec{f}} \partial_e \vec{f} \ \Tilde{\vec{w}}} \ ds, \ \int_0^1 \mathcal{R}^e \brac{d \psi_{\Tilde{\gamma}(s)} \partial_e \Tilde{\vec{w}}} \ ds\rangle_L  \label{almost_2_4}.
\end{align} 
For \eqref{almost_2_1}, we use \cref{riesz_estimate} to get
\begin{align*}
    &\int_{\R^d} \langle \textcolor{blue}{H_{\abs{\nabla}}}\brac{\vec{v} \wedge_L \textcolor{blue}{,} \int_0^1 \mathcal{R}^e \brac{d^2 \psi_{\Tilde{\gamma}(s)} d\varphi_{\vec{f}} \partial_e \vec{f} \ \Tilde{\vec{w}}} \ ds},  \int_0^1 \mathcal{R}^e \brac{d \psi_{\Tilde{\gamma}(s)} \partial_e \Tilde{\vec{w}}} \ ds\rangle_L \\
     \lesssim_{\Lambda} & \brac{\lVert \Dso^{1+\sigma}\vec{u} \rVert_{L^{\frac{2d}{2(1+\sigma)-1}}(\R^d)}^2 + \lVert \Dso^{1+\sigma}\vec{v} \rVert_{L^{\frac{2d}{2(1+\sigma)-1}}(\R^d)}^2} \ \lVert \Dso \Tilde{\vec{w}} \rVert_{L^{2}(\R^d)}^2.
\end{align*}
For \eqref{almost_2_2}, we use \cref{lem:prod_rule_frac} and \cref{lem:gag_in} to get
\begin{align*}
    &\int_{\R^d} \langle H_{\abs{\nabla}}
    \brac{\vec{v}^j, \vec{v}^j\vec{v}} \wedge_L\int_0^1 \mathcal{R}^e \brac{d^2 \psi_{\Tilde{\gamma}(s)} d\varphi_{\vec{f}} \partial_e \vec{f} \ \Tilde{\vec{w}}} \ ds, \int_0^1 \mathcal{R}^e \brac{d \psi_{\Tilde{\gamma}(s)} \partial_e \Tilde{\vec{w}}} \ ds\rangle_L \\
    \lesssim_{\Lambda} & \lVert H_{\abs{\nabla}}
    \brac{\vec{v}^j, \vec{v}^j\vec{v}} \rVert_{L^{2d}(\R^d)} \ \lVert \partial_e \vec{f} \rVert_{L^{2d}(\R^d)} \ \lVert \Tilde{\vec{w}} \rVert_{L^{\frac{2d}{d-2}}(\R^d)} \ \lVert \Dso\Tilde{\vec{w}} \lVert_{L^{2}(\R^d)}\\
    \lesssim_{\Lambda} & \lVert \Dso^{1-\sigma}\vec{v}^j \rVert_{L^{\frac{2d}{1-\sigma}}(\R^d)} \ \lVert \Dso^\sigma\brac{\vec{v}^j\vec{v}} \rVert_{L^{\frac{2d}{\sigma}}(\R^d)} \ \lVert \partial_e \vec{f} \rVert_{L^{2d}(\R^d)} \ \lVert \Dso\Tilde{\vec{w}} \lVert_{L^{2}(\R^d)}^2\\
    \lesssim_{\Lambda} & \lVert \Dso\vec{v}^j \rVert_{L^{\frac{2d}{1-\sigma}}(\R^d)}^{1-\sigma} \ \lVert \Dso\brac{\vec{v}^j\vec{v}} \rVert_{L^{\frac{2d}{\sigma}}(\R^d)}^\sigma \ \lVert \partial_e \vec{f} \rVert_{L^{2d}(\R^d)} \ \lVert \Dso\Tilde{\vec{w}} \lVert_{L^{2}(\R^d)}^2\\
    \lesssim_{\Lambda} & \brac{\lVert \Dso\vec{u} \rVert_{L^{2d}(\R^d)}^2 + \lVert \Dso\vec{v} \rVert_{L^{2d}(\R^d)}^2} \ \lVert \Dso\Tilde{\vec{w}} \rVert_{L^{2}(\R^d)}^2.
\end{align*}
Next, for \eqref{almost_2_3}, we again use \cref{riesz_estimate} to get
\begin{align*}
    &\int_{\R^d} \langle \vec{v}^j H_{\abs{\nabla}}
    \brac{\vec{v}^j\vec{v} \wedge_L, \int_0^1 \mathcal{R}^e \brac{d^2 \psi_{\Tilde{\gamma}(s)} d\varphi_{\vec{f}} \partial_e \vec{f} \ \Tilde{\vec{w}}} \ ds}, \int_0^1 \mathcal{R}^e \brac{d \psi_{\Tilde{\gamma}(s)} \partial_e \Tilde{\vec{w}}} \ ds\rangle_L\\
    \lesssim_{\Lambda} & \brac{\lVert \Dso^{1+\sigma}\vec{u} \rVert_{L^{\frac{2d}{2(1+\sigma)-1}}(\R^d)}^2 + \lVert \Dso^{1+\sigma}\vec{v} \rVert_{L^{\frac{2d}{2(1+\sigma)-1}}(\R^d)}^2} \ \lVert \Dso \Tilde{\vec{w}} \rVert_{L^{2}(\R^d)}^2.
\end{align*}
For \eqref{almost_2_4}, we have
\begin{align*}
    &\int_{\R^d} \langle \abs{\nabla}\vec{v}, \vec{v} \rangle_L \ \langle \vec{v} \wedge_L\int_0^1 \mathcal{R}^e \brac{d^2 \psi_{\Tilde{\gamma}(s)} d\varphi_{\vec{f}} \partial_e \vec{f} \ \Tilde{\vec{w}}} \ ds \ \int_0^1 \mathcal{R}^e \brac{d \psi_{\Tilde{\gamma}(s)} \partial_e \Tilde{\vec{w}}} \ ds\rangle_L\\
    \lesssim_{\Lambda} & \lVert \Dso \vec{v} \rVert_{L^{2d}(\R^d)} \ \lVert \partial_e \vec{f} \rVert_{L^{2d}(\R^d)} \ \lVert \Tilde{\vec{w}} \rVert_{L^{\frac{2d}{d-2}}(\R^d)} \ \lVert \Dso \Tilde{\vec{w}} \rVert_{L^{2}(\R^d)}\\
    \lesssim_{\Lambda} & \brac{\lVert \Dso \vec{u} \rVert_{L^{2d}(\R^d)}^2 + \lVert \Dso \vec{v} \rVert_{L^{2d}(\R^d)}^2} \ \lVert \Dso \Tilde{\vec{w}} \rVert_{L^{2}(\R^d)}^2.
\end{align*}
We conclude.
\end{proof}
We next estimate \eqref{split_2_3}, which will conclude our estimate of \eqref{eq:last_2}.
\begin{lemma}
For $d \geq 2$,
\begin{align*}
    &\abs{\sideset{}{'}{\sum}_{j=1}^3\int_{\R^d} \Big\langle \vec{v}^jH_{\Dso}\brac{\vec{v}^j, \int_0^1 \mathcal{R}^e \brac{d \psi_{\Tilde{\gamma}(s)} \partial_e \Tilde{\vec{w}}} \ ds }, \vec{v} \wedge_L\int_0^1 \mathcal{R}^e \brac{d \psi_{\Tilde{\gamma}(s)} \partial_e \Tilde{\vec{w}}} \ ds \Big\rangle_L}\\
    & \lesssim_{\Lambda} \lVert \Dso \Tilde{\vec{w}}\rVert_{L^{2}(\R^d)}^2\, \lVert\Dso \vec{v}\rVert_{L^{2d}(\mathbb{R}^d)}^2.
\end{align*}
\end{lemma}
\begin{proof}
For the reader's convenience, we set $\vec{\Gamma} = \int_0^1 \mathcal{R}^e \brac{d \psi_{\Tilde{\gamma}(s)} \partial_e \Tilde{\vec{w}}} \ ds$. With $\vec{a} \cdot_L (\vec{b} \wedge_L \vec{c}) = \det(\vec{a},\vec{b},\vec{c})$, we have
\begin{align*}
    &\sideset{}{'}{\sum}_{j=1}^3\int_{\R^d} \Big\langle \vec{v}^jH_{\Dso}\brac{\vec{v}^j, \vec{\Gamma}}, \vec{v} \wedge_L\vec{\Gamma} \Big\rangle_L\\
    &= \int_{\R^d} \det\brac{\sideset{}{'}{\sum}_{j=1}^3 \vec{v}^j H_{\Dso}(\vec{v}^j, \vec{\Gamma}) \ , \ \vec{v} \ , \ \vec{\Gamma}}\\
    &= -\frac{1}{2}\int_{\R^d} \det\brac{\sideset{}{'}{\sum}_{j=1}^3 \vec{v}^j H_{\Dso}(\vec{v}^j, \vec{\Gamma}) \ , \ \vec{\Gamma} \ , \ \vec{v}}\\
    &+ \frac{1}{2}\int_{\R^d} \det\brac{\vec{\Gamma} \ , \ \sideset{}{'}{\sum}_{j=1}^3 \vec{v}^j H_{\Dso}(\vec{v}^j, \vec{\Gamma}) \ , \ \vec{v}}.
\end{align*}
Using the fact that two colinear entries in the determinant yields $0$, we get
\begin{align*}
    &\sideset{}{'}{\sum}_{j=1}^3\int_{\R^d}\Big\langle \vec{v}^jH_{\Dso}\brac{\vec{v}^j,\vec{\Gamma}}, \vec{v} \wedge_L\abs\nabla \vec{w} \Big\rangle_L\\
    &= -\frac{1}{2}\int_{\R^d} \det\brac{\sideset{}{'}{\sum}_{j=1}^3 \vec{v}^j \Dso\brac{\vec{v}^j \vec{\Gamma}} - \vec{v}^j\vec{v}^j\Dso\brac{\vec{\Gamma}} \ , \ \vec{\Gamma} \ , \ \vec{v}}\\
    &+ \frac{1}{2}\int_{\R^d} \det\brac{\vec{\Gamma} \ , \ \sideset{}{'}{\sum}_{j=1}^3 \vec{v}^j \Dso\brac{\vec{v}^j \vec{\Gamma}} - \vec{v}^j\vec{v}^j\Dso\brac{\vec{\Gamma}} \ , \ \vec{v}}\\
    &\overset{\langle \vec{v},\vec{v}\rangle_L = -1}{=} -\frac{1}{2}\int_{\R^d} \det\brac{\sideset{}{'}{\sum}_{j=1}^3 \vec{v}^j \Dso\brac{\vec{v}^j \vec{\Gamma}} + \Dso\brac{\vec{\Gamma}} \ , \ \vec{\Gamma} \ , \ \vec{v}}\\
    &+ \frac{1}{2}\int_{\R^d} \det\brac{\vec{\Gamma} \ , \ \sideset{}{'}{\sum}_{j=1}^3 \vec{v}^j \Dso\brac{\vec{v}^j \vec{\Gamma}} + \Dso\brac{\vec{\Gamma}} \ , \ \vec{v}}\\
    &=-\frac{1}{2}\int_{\R^d} \det\brac{\sideset{}{'}{\sum}_{j=1}^3 \vec{v}^j \Dso\brac{\vec{v}^j \vec{\Gamma}} \ , \ \vec{\Gamma} \ , \ \vec{v}}\\
    &- \frac{1}{2}\int_{\R^d} \det\brac{\Dso\brac{\vec{\Gamma}} \ , \ \vec{\Gamma} \ , \ \vec{v}}\\
    &+ \frac{1}{2}\int_{\R^d} \det\brac{\vec{\Gamma} \ , \ \sideset{}{'}{\sum}_{j=1}^3 \vec{v}^j \Dso\brac{\vec{v}^j \vec{\Gamma}} + \Dso\brac{\vec{\Gamma}} \ , \ \vec{v}}.
\end{align*}
Now we expand the determinant, denoting the Levi-Cevita tensor by $\epsilon_{\alpha.\beta,n} \in \{-1,0,1\}$.
\begin{align*}
    &=-\frac{1}{2}\sum_{\alpha,\beta,n = 1}^3 \epsilon_{\alpha,\beta,n} \int_{\R^d} \sideset{}{'}{\sum}_{j=1}^3 \vec{v}^j \Dso\brac{\vec{v}^j \vec{\Gamma}^\alpha} \ \vec{\Gamma}^\beta \ \vec{v}^n\\
    &- \frac{1}{2}\sum_{\alpha,\beta,n = 1}^3 \epsilon_{\alpha,\beta,n}\int_{\R^d} \Dso \brac{\vec{\Gamma}^\alpha} \ \vec{\Gamma}^\beta \ \vec{v}^n\\
    &+ \frac{1}{2}\sum_{\alpha,\beta,n = 1}^3 \epsilon_{\alpha,\beta,n}\int_{\R^d} \vec{\Gamma}^\alpha\ \sideset{}{'}{\sum}_{j=1}^3 \vec{v}^j \Dso\brac{\vec{v}^j \vec{\Gamma}^\beta} \vec{v}^n\\
    &+ \frac{1}{2}\sum_{\alpha,\beta,n = 1}^3 \epsilon_{\alpha,\beta,n}\int_{\R^d} \vec{\Gamma}^\alpha\ \sideset{}{'}{\sum}_{j=1}^3 \Dso \brac{\vec{\Gamma}^\beta} \ \vec{v}^n.
\end{align*}
We perform an integration by parts to factor out $\vec{\Gamma}^\alpha$ to obtain
\begin{align*}
    &=-\frac{1}{2} \sum_{\alpha,\beta,n = 1}^3 \epsilon_{\alpha,\beta,n} \int_{\R^d} \vec{\Gamma}^\alpha\sideset{}{'}{\sum}_{j=1}^3 \vec{v}^j \Dso\brac{\vec{v}^j \vec{\Gamma}^\beta \vec{v}^n}\\
    &- \frac{1}{2}\sum^3_{\alpha,\beta,m = 1} \epsilon_{\alpha,\beta,n}\int_{\R^d} \vec{\Gamma}^\alpha\ \Dso\brac{\vec{\Gamma}^\beta \vec{v}^n}\\
    &+ \frac{1}{2}\sum_{\alpha,\beta,n = 1}^3 \epsilon_{\alpha,\beta,n}\int_{\R^d} \vec{\Gamma}^\alpha \sideset{}{'}{\sum}_{j=1}^3 \vec{v}^j \Dso\brac{\vec{v}^j \vec{\Gamma}^\beta} \vec{v}^n\\
    &+ \frac{1}{2}\sum_{\alpha,\beta,n = 1}^3 \epsilon_{\alpha,\beta,n}\int_{\R^d} \vec{\Gamma}^\alpha \sideset{}{'}{\sum}_{j=1}^3 \Dso \brac{\vec{\Gamma}^\beta} \vec{v}^n.
\end{align*}
We regroup into
\begin{align*}
    &=-\frac{1}{2}\sum_{\alpha,\beta,n = 1}^3 \epsilon_{\alpha,\beta,n} \int_{\R^d} \vec{\Gamma}^\alpha \brac{\sideset{}{'}{\sum}_{j=1}^3 \brac{\vec{v}^j \Dso\brac{\vec{v}^j \vec{\Gamma}^\beta \ \vec{v}^n}} - \sideset{}{'}{\sum}_{j=1}^3 \brac{\vec{v}^j \Dso\brac{{\vec{v}^j} \vec{\Gamma}^\beta} \ \vec{v}^n}}\\
    &- \frac{1}{2}\sum_{\alpha,\beta,n = 1}^3 \epsilon_{\alpha,\beta,n}\int_{\R^d} \vec{\Gamma}^\alpha \ \brac{ \Dso\brac{\vec{\Gamma}^\beta \ \vec{v}^n} + \Dso \vec{\Gamma}^\beta \ \vec{v}^n}\\
    &=-\frac{1}{2}\sum_{\alpha,\beta,n = 1}^3 \epsilon_{\alpha,\beta,n} \int_{\R^d} \vec{\Gamma}^\alpha \sideset{}{'}{\sum}_{j=1}^3 \vec{v}^j \brac{H_{\Dso}\brac{\vec{v}^j \vec{\Gamma}^\beta, \vec{v}^n} + \vec{v}^j \vec{\Gamma}^\beta \Dso {\vec{v}^n}}\\
    &- \frac{1}{2}\sum_{\alpha,\beta,n = 1}^3 \epsilon_{\alpha,\beta,n}\int_{\R^d} \vec{\Gamma}^\alpha \brac{ H_{\Dso}\brac{\vec{\Gamma}^\beta, \vec{v}^n} - \vec{\Gamma}^\beta \Dso{\vec{v}^n}}.
\end{align*}

Since

\begin{align*}
    \sum^3_{\alpha,\beta,m = 1} \epsilon_{k\ell m} \vec{\Gamma}^\alpha \sideset{}{'}{\sum}_{j=1}^3\vec{v}^j\vec{v}^j\vec{\Gamma}^\beta\Dso \vec{v}^n = \det\brac{\vec{\Gamma},\ \sideset{}{'}{\sum}_{j=1}^3 \vec{v}^j\vec{v}^j\vec{\Gamma},\ \Dso \vec{v}} = 0
\end{align*}
and
\begin{align*}
    \sum^3_{\alpha,\beta,m = 1} \epsilon_{k\ell m} \vec{\Gamma}^\alpha\vec{\Gamma}^\beta\Dso \vec{v}^n = \det\brac{\vec{\Gamma},\ \sideset{}{'}{\sum}_{j=1}^3 \vec{v}^j\vec{v}^j\vec{\Gamma},\ \Dso \vec{v}} = 0.
\end{align*}
We finally obtain
\begin{align*}
    &\sideset{}{'}{\sum}_{j=1}^3\int_{\R^d}\Big\langle \vec{v}^jH_{\Dso}\brac{\vec{v}^j, \vec{\Gamma}}, \vec{v} \wedge_L\vec{\Gamma} \Big\rangle_L\\
    &=-\frac{1}{2}\sum_{\alpha,\beta,n = 1}^3 \epsilon_{\alpha,\beta,n} \int_{\R^d} \vec{\Gamma}^\alpha \sideset{}{'}{\sum}_{j=1}^3 \vec{v}^j H_{\Dso}\brac{\vec{v}^j \vec{\Gamma}^\beta, \vec{v}^n} \\
    &+ \frac{1}{2}\sum_{\alpha,\beta,n = 1}^3 \epsilon_{\alpha,\beta,n}\int_{\R^d} \vec{\Gamma}^\alpha \sideset{}{'}{\sum}_{j=1}^3 \vec{v}^j \vec{v}^j H_{\Dso}\brac{\vec{\Gamma}^\beta, \vec{v}^n}\\
    &=-\frac{1}{2}\sum_{\alpha,\beta,n = 1}^3 \epsilon_{\alpha,\beta,n} \int_{\R^d} \vec{\Gamma}^\alpha \sideset{}{'}{\sum}_{j=1}^3 \vec{v}^j\brac{H_{\Dso}\brac{\vec{v}^j \vec{\Gamma}^\beta, \vec{v}^n} - \vec{v}^j H_{\Dso}\brac{\vec{\Gamma}^\beta, \vec{v}^n}}.
\end{align*}

Consequently,
\begin{align*}
    &\abs{\sideset{}{'}{\sum}_{j=1}^3 \int_{\R^d} \langle \vec{v}^jH_{\Dso}(\vec{v}^j,\vec{\Gamma}), \vec{v} \wedge_L\vec{\Gamma} \rangle_L}\\
    &\lesssim \lVert \vec{\Gamma}\rVert_{L^2(\R^d)} \lVert \ \max_{\alpha \beta n} \lVert H_{\Dso}(\vec{v}^j\vec{\Gamma}^\beta,\vec{v}^n) - \vec{v}^jH_{\Dso}(\vec{\Gamma}^\beta,\vec{v}^n)\rVert_{L^2(\R^d)}\\
    &\lesssim_{\Lambda} \lVert \Dso\tilde{\vec{w}}\rVert_{L^2(\R^d)} \lVert \ \max_{\alpha \beta n} \lVert H_{\Dso}(\vec{v}^j\vec{\Gamma}^\beta,\vec{v}^n) - \vec{v}^jH_{\Dso}(\vec{\Gamma}^\beta,\vec{v}^n)\rVert_{L^2(\R^d)}.
\end{align*}
From \cref{lem:comm_difference} for any $\alpha \in (\frac{1}{2},1)$ ($d\ge 2$),
\begin{align*}
    &\lVert H_{\Dso}(\vec{v}^j\vec{\Gamma}^\beta,\vec{v}^n) - \vec{v}^jH_{\Dso}(\vec{\Gamma}^\beta,\vec{v}^n)\rVert_{L^2(\R^d)}\\
    &\lesssim \lVert \vec{\Gamma}\rVert_{L^2(\R^d)} \ \lVert \Dso^\alpha\vec{v}\rVert_{L^{\frac{2d}{2\alpha - 1}}(\R^d)}\\
    &\lesssim \lVert \vec{\Gamma}\rVert_{L^2(\R^d)} \ \lVert \Dso\vec{v}\rVert_{L^{2d}(\R^d)}\\
    &\lesssim_{\Lambda} \lVert \Dso\Tilde{\vec{w}}\rVert_{L^2(\R^d)} \ \lVert \Dso\vec{v}\rVert_{L^{2d}(\R^d)}.
\end{align*}
So we conclude.
\end{proof}
To complete the proof of \cref{thm:1}, we estimate the following two terms. To estimate \eqref{eq:last_3}, we use
\begin{align*}
    \vec{X}^r_{\vec{u}} - \vec{X}^r_{\vec{v}} 
    &= \int_0^1 d\vec{X}^r_{\gamma(s)}\vec{w} \ ds\\
    &= \int_0^1 \int_0^1 d\vec{X}^r_{\gamma(s)} d\psi_{\Tilde{\gamma}(s')} \Tilde{\vec{w}} \ ds \ ds'\\
    &= \int_0^1 \int_0^1  \sum_{k=1}^3\sum_{\ell = 1}^m \frac{\partial\vec{X}^{r}_{\gamma(s)}}{\partial p^k} \frac{\partial\psi^{k}_{\Tilde{\gamma}(s')}}{\partial q^\ell} \Tilde{\vec{w}}^\ell \ ds \ ds',
\end{align*}

where $\gamma(s) = s\vec{u} + (1-s)\vec{v}$.
\begin{lemma} Let $d\ge 3$ and $\sigma \in (1, d + \frac{1}{2})$. Then we have
    \begin{align*}
        &\abs{\int_{\R^d} \langle \vec{v} \wedge_LH_{\Dso}(\vec{v} \wedge_L, \int_0^1 \mathcal{R}^e \brac{d \psi_{\Tilde{\gamma}(s)} \partial_e \Tilde{\vec{w}}} \ ds), \vec{X}_{\vec{u}}^r \rangle_L \langle \vec{v} \wedge_L\Dso{\vec{v}, \vec{X}_{\vec{u}}^r} - \vec{X}_{\vec{v}}^r \rangle_L}\\
        \lesssim_{\Lambda} & \ \lVert \Dso\Tilde{\vec{w}} \rVert_{L^2(\R^d)}^2 \brac{\lVert \nabla^{1+\sigma}\vec{u} \rVert_{L^{\frac{2d}{2(1+\sigma)-1}}(\R^d)}^2 + \lVert \nabla^{1 + \sigma}\vec{v} \rVert_{L^{\frac{2d}{2(1+\sigma)-1}}(\R^d)}^2}.
    \end{align*}
\end{lemma}
\begin{proof}
\begin{align*}
    &\int_{\R^d} \langle \vec{v} \wedge_LH_{\Dso}(\vec{v} \wedge_L, \int_0^1 \mathcal{R}^e \brac{d \psi_{\Tilde{\gamma}(s)} \partial_e \Tilde{\vec{w}}} \ ds), \vec{X}_{\vec{u}}^r \rangle_L \langle \vec{v} \wedge_L\Dso{\vec{v}, \vec{X}_{\vec{u}}^r} - \vec{X}_{\vec{v}}^r \rangle_L\\
    =&\int_{\R^d} \langle \vec{v} \wedge_LH_{\Dso}(\vec{v} \wedge_L, \int_0^1 \mathcal{R}^e \brac{d \psi_{\Tilde{\gamma}(s)} \partial_e \Tilde{\vec{w}}} \ ds), \vec{X}_{\vec{u}}^r \rangle_L \left\langle \vec{v} \wedge_L\Dso\vec{v}, \int_0^1 d\vec{X}^r_{\gamma(s)}\vec{w} \ ds \right\rangle_L\\
    =&\int_{\R^d} \langle \vec{v} \wedge_LH_{\Dso}(\vec{v} \wedge_L, \int_0^1 \mathcal{R}^e \brac{d^2 \psi_{\Tilde{\gamma}(s)} d\varphi_{\vec{f}} \partial_e \vec{f} \ \Tilde{\vec{w}}} ds), \vec{X}_{\vec{u}}^r \rangle_L \left\langle \vec{v} \wedge_L\Dso\vec{v}, \int_0^1 \int_0^1  \sum_{k,\ell}\frac{\partial\vec{X}^{r}_{\gamma(s)}}{\partial p^k} \frac{\partial\psi^{k}_{\gamma(s)}}{\partial q^\ell} \Tilde{\vec{w}}^\ell \ ds \ ds'\right\rangle_L\\
    =&\sum_{\alpha,\beta,n =1}^3 \epsilon_{\alpha,\beta,n} \int_{\R^d} \langle \vec{v} \wedge_LH_{\Dso}(\vec{v} \wedge_L, \int_0^1 \mathcal{R}^e \brac{d^2 \psi_{\Tilde{\gamma}(s)} d\varphi_{\vec{f}} \partial_e \vec{f} \ \Tilde{\vec{w}}} ds), \vec{X}_{\vec{u}}^r \rangle_L  \vec{v}^\alpha \Dso\vec{v}^\beta \int_0^1 \int_0^1  \sum_{k,\ell}\frac{\partial\vec{X}^{r,n}_{\gamma(s)}}{\partial p^k} \frac{\partial\psi^{k}_{\gamma(s)}}{\partial q^\ell} \Tilde{\vec{w}}^\ell \ ds ds'.
\end{align*}
Let $G^{n,k}_\ell = \int_0^1 \int_0^1 \frac{\partial\vec{X}^{r,n}_{\gamma(s)}}{\partial p^k} \frac{\partial\psi^{k}_{\gamma(s)}}{\partial q^\ell} \ ds ds'$. So we have
\begin{align}
    &\int_{\R^d} \langle \vec{v} \wedge_LH_{\Dso}(\vec{v} \wedge_L, \int_0^1 \mathcal{R}^e \brac{d \psi_{\Tilde{\gamma}(s)} \partial_e \Tilde{\vec{w}}} \ ds), \vec{X}_{\vec{u}}^r \rangle_L \vec{v}^\alpha \Dso\vec{v}^\beta G^{n,k}_\ell \vec{w}^\ell \nonumber\\
    =& \int_{\R^d} \langle H_{\Dso}(\vec{v} \wedge_L, \int_0^1 \mathcal{R}^e \brac{d \psi_{\Tilde{\gamma}(s)} \partial_e \Tilde{\vec{w}}} \ ds), \vec{X}_{\vec{u}}^r \wedge_L\vec{v} \rangle_L \vec{v}^\alpha \Dso\vec{v}^\beta G^{n,k}_\ell \vec{w}^\ell\nonumber\\
    =& \int_{\R^d} \langle H_{\Dso}\brac{\vec{v} \wedge_L, \vec{v}^\alpha \Dso\vec{v}^\beta G^{n,k}_\ell \vec{w}^\ell\brac{\vec{X}_{\vec{u}}^r \wedge_L\vec{v}}}, \int_0^1 \mathcal{R}^e \brac{d \psi_{\Tilde{\gamma}(s)} \partial_e \Tilde{\vec{w}}} \ ds \rangle_L\nonumber\\
    &+ 2\int_{\R^d} \langle \Dso\vec{v} \wedge_L\vec{v}^\alpha \Dso\vec{v}^\beta G^{n,k}_\ell \vec{w}^\ell\brac{\vec{X}_{\vec{u}}^r \wedge_L\vec{v}}, \int_0^1 \mathcal{R}^e \brac{d \psi_{\Tilde{\gamma}(s)} \partial_e \Tilde{\vec{w}}} \ ds \rangle_L\nonumber\\
    =& \int_{\R^d} \langle H_{\Dso}\brac{\vec{v} \wedge_L\vec{v}^\alpha G^{n,k}_\ell \brac{\vec{X}_{\vec{u}}^r \wedge_L\vec{v}},\Dso\vec{v}^\beta \vec{w}^\ell}, \int_0^1 \mathcal{R}^e \brac{d \psi_{\Tilde{\gamma}(s)} \partial_e \Tilde{\vec{w}}} \ ds \rangle_L\label{end_1_1}\\
    &+ \int_{\R^d} \langle H_{\Dso}\brac{\vec{v} \wedge_L, \vec{v}^\alpha G^{n,k}_\ell \brac{\vec{X}_{\vec{u}}^r \wedge_L\vec{v}}}\Dso\vec{v}^\beta \vec{w}^\ell, \int_0^1 \mathcal{R}^e \brac{d \psi_{\Tilde{\gamma}(s)} \partial_e \Tilde{\vec{w}}} \ ds \rangle_L\label{end_1_2}\\
    &- \int_{\R^d} \langle \vec{v} \wedge_LH_{\Dso}\brac{\vec{v}^\alpha G^{n,k}_\ell \brac{\vec{X}_{\vec{u}}^r \wedge_L\vec{v}},\Dso\vec{v}^\beta \vec{w}^\ell}, \int_0^1 \mathcal{R}^e \brac{d \psi_{\Tilde{\gamma}(s)} \partial_e \Tilde{\vec{w}}} \ ds \rangle_L\label{end_1_3}\\
    &+ 2\int_{\R^d} \langle \Dso\vec{v} \wedge_L\vec{v}^\alpha \Dso\vec{v}^\beta G^{n,k}_\ell \vec{w}^\ell\brac{\vec{X}_{\vec{u}}^r \wedge_L\vec{v}}, \int_0^1 \mathcal{R}^e \brac{d \psi_{\Tilde{\gamma}(s)} \partial_e \Tilde{\vec{w}}} \ ds \rangle_L\label{end_1_4}.
\end{align}
For \eqref{end_1_1}, we use \cref{lem:prod_rule_frac} to get
\begin{align*}
    &\int_{\R^d} \langle H_{\Dso}\brac{\vec{v} \wedge_L\vec{v}^\alpha G^{n,k}_\ell \brac{\vec{X}_{\vec{u}}^r \wedge_L\vec{v}},\Dso\vec{v}^\beta \vec{w}^\ell}, \int_0^1 \mathcal{R}^e \brac{d \psi_{\Tilde{\gamma}(s)} \partial_e \Tilde{\vec{w}}} \ ds \rangle_L\\
    \lesssim_{\Lambda} & \lVert \Dso^{1-\sigma}\brac{\vec{v} \wedge_L\vec{v}^\alpha G^{n,k}_\ell \brac{\vec{X}_{\vec{u}}^r \wedge_L\vec{v}}} \rVert_{L^{\frac{2d}{2(1-\sigma) - 1}}(\R^d)} \ \lVert\Dso^{\sigma}\brac{\Dso\vec{v}^\beta \vec{w}^\ell} \rVert_{L^{\frac{2d}{d+2\sigma -1}}(\R^d)} \ \lVert \Dso \Tilde{\vec{w}} \rVert_{L^{2}(\R^d)}\\
    \lesssim_{\Lambda} & \lVert \Dso\brac{\vec{v} \wedge_L\vec{v}^\alpha G^{n,k}_\ell \brac{\vec{X}_{\vec{u}}^r \wedge_L\vec{v}}} \rVert_{L^{2d}(\R^d)} \ \lVert\Dso^{\sigma}\brac{\Dso\vec{v}^\beta \vec{w}^\ell} \rVert_{L^{\frac{2d}{d+2\sigma -1}}(\R^d)} \ \lVert \Dso \Tilde{\vec{w}} \rVert_{L^{2}(\R^d)}\\
    \overset{\cref{comp}}{\lesssim_{\Lambda}} & \lVert \nabla\brac{\vec{v} \wedge_L\vec{v}^\alpha G^{n,k}_\ell \brac{\vec{X}_{\vec{u}}^r \wedge_L\vec{v}}} \rVert_{L^{2d}(\R^d)} \ \lVert\Dso^{\sigma}\brac{\Dso\vec{v}^\beta \vec{w}^\ell} \rVert_{L^{\frac{2d}{d+2\sigma -1}}(\R^d)} \ \lVert \Dso \Tilde{\vec{w}} \rVert_{L^{2}(\R^d)}\\
    \lesssim_{\Lambda} & \brac{\lVert \Dso^{1+\sigma}\vec{u} \rVert_{L^{\frac{2d}{2(1+\sigma) - 1}}(\R^d)}^2 + \lVert \Dso^{1+\sigma}\vec{v} \rVert_{L^{\frac{2d}{2(1+\sigma) - 1}}(\R^d)}^2} \ \lVert \Dso \Tilde{\vec{w}} \rVert_{L^{2}(\R^d)}^2.
\end{align*}
For \eqref{end_1_2}, we use \cref{lem:gag_in} to get
\begin{align*}
    &\int_{\R^d} \langle H_{\Dso}\brac{\vec{v} \wedge_L, \vec{v}^\alpha G^{n,k}_\ell \brac{\vec{X}_{\vec{u}}^r \wedge_L\vec{v}}}\Dso\vec{v}^\beta \vec{w}^\ell, \int_0^1 \mathcal{R}^e \brac{d \psi_{\Tilde{\gamma}(s)} \partial_e \Tilde{\vec{w}}} \ ds \rangle_L\\
    \lesssim_{\Lambda} & \lVert \Dso^{1-\sigma} \vec{v} \rVert_{L^{\frac{2d}{1-\sigma}}(\R^d)} \ \lVert \Dso^\sigma\brac{\vec{v}^\alpha G^{n,k}_\ell \brac{\vec{X}_{\vec{u}}^r \wedge_L\vec{v}}} \rVert_{L^{\frac{2d}{\sigma}}(\R^d)} \ \lVert \Dso\vec{v}^\beta \rVert_{L^{2d}(\R^d)} \ \lVert \vec{w}^\ell \rVert_{L^{\frac{2d}{d-2}}(\R^d)} \ \lVert \Dso \Tilde{\vec{w}} \rVert_{L^{2}(\R^d)}\\
    \lesssim_{\Lambda} & \lVert \Dso \vec{v} \rVert_{L^{2d}(\R^d)}^{1-\sigma} \ \lVert \Dso\brac{\vec{v}^\alpha G^{n,k}_\ell \brac{\vec{X}_{\vec{u}}^r \wedge_L\vec{v}}} \rVert_{L^{2d}(\R^d)}^\sigma \ \lVert \Dso\vec{v} \rVert_{L^{2d}(\R^d)} \ \lVert \Dso \Tilde{\vec{w}} \rVert_{L^{2}(\R^d)}^2\\
    \lesssim_{\Lambda} & \lVert \Dso \vec{v} \rVert_{L^{2d}(\R^d)}^{1-\sigma} \ \lVert \nabla\brac{\vec{v}^\alpha G^{n,k}_\ell \brac{\vec{X}_{\vec{u}}^r \wedge_L\vec{v}}} \rVert_{L^{2d}(\R^d)}^\sigma \ \lVert \Dso\vec{v} \rVert_{L^{2d}(\R^d)} \ \lVert \Dso \Tilde{\vec{w}} \rVert_{L^{2}(\R^d)}^2\\
    \lesssim_{\Lambda} & \brac{\lVert \Dso \vec{u} \rVert_{L^{2d}(\R^d)}^2 + \lVert \Dso \vec{v} \rVert_{L^{2d}(\R^d)}^2} \ \lVert \Dso \Tilde{\vec{w}} \rVert_{L^{2}(\R^d)}^2.
\end{align*}
For \eqref{end_1_3}, using \cref{lem:prod_rule_frac}, we get
\begin{align*}
    &\int_{\R^d} \langle \vec{v} \wedge_LH_{\Dso}\brac{\vec{v}^\alpha G^{n,k}_\ell \brac{\vec{X}_{\vec{u}}^r \wedge_L\vec{v}},\Dso\vec{v}^\beta \vec{w}^\ell}, \int_0^1 \mathcal{R}^e \brac{d \psi_{\Tilde{\gamma}(s)} \partial_e \Tilde{\vec{w}}} \ ds \rangle_L\\
    \lesssim_{\Lambda} & \lVert \Dso^{1-\sigma}\brac{\vec{v}^\alpha G^{n,k}_\ell \brac{\vec{X}_{\vec{u}}^r \wedge_L\vec{v}}} \rVert_{L^{\frac{2d}{2(1-\sigma)-1}}(\R^d)} \ \lVert \Dso^\sigma\brac{\Dso\vec{v}^\beta \vec{w}^\ell} \rVert_{L^{\frac{2d}{d + 2\sigma -1}}(\R^d)} \ \lVert \Dso \Tilde{\vec{w}} \rVert_{L^{2}(\R^d)}\\
    \lesssim_{\Lambda} & \lVert \Dso\brac{\vec{v}^\alpha G^{n,k}_\ell \brac{\vec{X}_{\vec{u}}^r \wedge_L\vec{v}}} \rVert_{L^{2d}(\R^d)} \ \lVert \Dso^\sigma\brac{\Dso\vec{v}^\beta \vec{w}^\ell} \rVert_{L^{\frac{2d}{d + 2\sigma -1}}(\R^d)} \ \lVert \Dso \Tilde{\vec{w}} \rVert_{L^{2}(\R^d)}\\
    \overset{\cref{comp}}{\lesssim_{\Lambda}} & \lVert \nabla\brac{\vec{v}^\alpha G^{n,k}_\ell \brac{\vec{X}_{\vec{u}}^r \wedge_L\vec{v}}} \rVert_{L^{2d}(\R^d)} \ \lVert \Dso^\sigma\brac{\Dso\vec{v}^\beta \vec{w}^\ell} \rVert_{L^{\frac{2d}{d + 2\sigma -1}}(\R^d)} \ \lVert \Dso \Tilde{\vec{w}} \rVert_{L^{2}(\R^d)}\\
    \lesssim_{\Lambda} & \brac{\lVert \Dso^{1+\sigma} \vec{u} \rVert_{L^{\frac{2d}{2(1+\sigma) - 1}}(\R^d)}^2 + \lVert \Dso^{1+\sigma} \vec{v} \rVert_{L^{\frac{2d}{2(1+\sigma) - 1}}(\R^d)}^2} \ \lVert \Dso \Tilde{\vec{w}} \rVert_{L^{2}(\R^d)}^2.
\end{align*}
Finally, for \eqref{end_1_4} we have
\begin{align*}
    &\int_{\R^d} \langle \Dso\vec{v} \wedge_L\vec{v}^\alpha \Dso\vec{v}^\beta G^{n,k}_\ell \vec{w}^\ell\brac{\vec{X}_{\vec{u}}^r \wedge_L\vec{v}}, \int_0^1 \mathcal{R}^e \brac{d \psi_{\Tilde{\gamma}(s)} \partial_e \Tilde{\vec{w}}} \ ds \rangle_L\\
    \lesssim_{\Lambda} & \lVert \Dso\vec{v} \rVert_{L^{2d}(\R^d)}^2 \lVert \Tilde{\vec{w}} \rVert_{L^{\frac{2d}{d-2}}(\R^d)} \ \lVert \Dso \Tilde{\vec{w}} \rVert_{L^{2}(\R^d)}\\
    \lesssim_{\Lambda} & \lVert \Dso\vec{v} \rVert_{L^{2d}(\R^d)}^2 \ \lVert \Dso \Tilde{\vec{w}} \rVert_{L^{2}(\R^d)}^2.
\end{align*}
Therefore,
\begin{align*}
    &\abs{\int_{\R^d} \langle \vec{v} \wedge_LH_{\Dso}(\vec{v} \wedge_L, \int_0^1 \mathcal{R}^e \brac{d \psi_{\Tilde{\gamma}(s)} \partial_e \Tilde{\vec{w}}} \ ds), \vec{X}_{\vec{u}}^r \rangle_L \langle \vec{v} \wedge_L\Dso{\vec{v}, \vec{X}_{\vec{u}}^r} - \vec{X}_{\vec{v}}^r \rangle_L}\\
    \lesssim_{\Lambda} & \ \lVert \Dso\Tilde{\vec{w}} \rVert_{L^2(\R^d)}^2 \brac{\lVert \nabla^{1+\sigma}\vec{u} \rVert_{L^{\frac{2d}{2(1+\sigma)-1}}(\R^d)}^2 + \lVert \nabla^{1 + \sigma}\vec{v} \rVert_{L^{\frac{2d}{2(1+\sigma)-1}}(\R^d)}^2}.
 \end{align*}

\end{proof}

Recall that 
\begin{align*}
    \vec{E}^r_{\vec{u}} - \vec{E}^r_{\vec{v}} 
    &= \int_0^1 d\vec{E}^r_{\Tilde{\gamma}(s)}\Tilde{\vec{w}} \ ds\\
    &= \int_0^1 \sum_{\ell=1}^m\frac{\partial \vec{E}^r_{\Tilde{\gamma}(s)}}{\partial p^\ell} \Tilde{\vec{w}}^\ell \ ds,
\end{align*}
where $\Tilde{\gamma}(s) = s\Tilde{\vec{u}} + (1-s)\Tilde{\vec{v}}$.

\begin{lemma} Let $d \ge 3$ and $\sigma \in (1, d + \frac{1}{2})$. Then we have
    \begin{align*}
        &\abs{\int_{\R^d}\langle\vec{v} \wedge_L H_{\abs{\nabla}}\brac{\vec{v} \wedge_L, \int_0^1 \mathcal{R}^e \brac{d \psi_{\Tilde{\gamma}(s)} \partial_e \Tilde{\vec{w}}} \ ds}, \mathbf{X}_\vec{u}^r\rangle_L\langle \vec{v} \wedge_L\abs\nabla \vec{v}, \vec{X}_{\vec{v}}^k\rangle_L\langle\mathbf{E}_{\Tilde{\vec{u}}}^r, \vec{E}^k_{\Tilde{\vec{u}}}-\vec{E}^k_{\Tilde{\vec{v}}}\rangle}\\
        \lesssim_{\Lambda} & \brac{\lVert \nabla^{1+\sigma}\vec{u} \rVert_{L^{\frac{2d}{2(1+\sigma)-1}}(\R^d)}^2 + \lVert \nabla^{1 + \sigma}\vec{v} \rVert_{L^{\frac{2d}{2(1+\sigma)-1}}(\R^d)}^2}\lVert \Dso\vec{w} \rVert_{L^2(\R^d)}^2.
    \end{align*}
\end{lemma}
\begin{proof}
\begin{align*}
    &\int_{\R^d}\langle\vec{v} \wedge_L H_{\abs{\nabla}}\brac{\vec{v} \wedge_L, \int_0^1 \mathcal{R}^e \brac{d \psi_{\Tilde{\gamma}(s)} \partial_e \Tilde{\vec{w}}} \ ds}, \mathbf{X}_\vec{u}^r\rangle_L\langle \vec{v} \wedge_L\abs\nabla \vec{v}, \vec{X}_{\vec{v}}^k\rangle_L\langle\mathbf{E}_{\Tilde{\vec{u}}}^r, \vec{E}^k_{\Tilde{\vec{u}}}-\vec{E}^k_{\Tilde{\vec{v}}}\rangle_{\R^m}\\
    &=\int_{\R^d} \langle \vec{v} \wedge_LH_{\Dso}(\vec{v} \wedge_L, \int_0^1 \mathcal{R}^e \brac{d \psi_{\Tilde{\gamma}(s)} \partial_e \Tilde{\vec{w}}} \ ds), \vec{X}_{\vec{u}}^r \rangle_L \langle \vec{v} \wedge_L\Dso\vec{v}, \vec{X}^k_{\vec{v}} \rangle_L \left\langle \vec{E}^r_{\Tilde{\vec{u}}}, \int_0^1 d\vec{E}^k_{\gamma(s)} \Tilde{\vec{w}}^\ell\ ds \right\rangle\\
    =&\sum_{j=1}^m \int_{\R^d} \langle \vec{v} \wedge_LH_{\Dso}(\vec{v} \wedge_L, \int_0^1 \mathcal{R}^e \brac{d^2 \psi_{\Tilde{\gamma}(s)} d\varphi_{\vec{f}} \partial_e \vec{f} \ \Tilde{\vec{w}}} ds), \vec{X}_{\vec{u}}^r \rangle_L  \langle \vec{v} \wedge_L\Dso\vec{v}, \vec{X}^k_{\vec{v}} \rangle_L \vec{E}^{k,j}_{\Tilde{\vec{u}}}\int_0^1  \sum_{\ell=1}^m \frac{\partial\vec{E}^{r,j}_{\Tilde{\gamma}(s)}}{\partial q^\ell} \Tilde{\vec{w}}^\ell \ ds.
\end{align*}
Let $F^{j}_\ell = \int_0^1\frac{\partial\vec{E}^{k,j}_{\Tilde{\gamma}(s)}}{\partial p^{\ell}} \ ds$. So we have
\begin{align}
    &\int_{\R^d} \langle \vec{v} \wedge_LH_{\Dso}(\vec{v} \wedge_L, \int_0^1 \mathcal{R}^e \brac{d \psi_{\Tilde{\gamma}(s)} \partial_e \Tilde{\vec{w}}} \ ds), \vec{X}_{\vec{u}}^r \rangle_L \vec{v}^\alpha \Dso\vec{v}^\beta\vec{X}^{k,n}_{\vec{v}} \vec{E}^{r,j}_{\Tilde{\vec{u}}} F^{j}_\ell \Tilde{\vec{w}}^\ell \nonumber\\
    =& \int_{\R^d} \langle H_{\Dso}(\vec{v} \wedge_L, \int_0^1 \mathcal{R}^e \brac{d \psi_{\Tilde{\gamma}(s)} \partial_e \Tilde{\vec{w}}} \ ds), \vec{X}_{\vec{u}}^r \wedge_L\vec{v} \rangle_L \vec{v}^\alpha \Dso\vec{v}^\beta\vec{X}^{k,n}_{\vec{v}} \vec{E}^{r,j}_{\Tilde{\vec{u}}} F^{j}_\ell \Tilde{\vec{w}}^\ell\nonumber\\
    =& \int_{\R^d} \langle H_{\Dso}\brac{\vec{v} \wedge_L, \vec{v}^\alpha \Dso\vec{v}^\beta\vec{X}^{k,n}_{\vec{v}} \vec{E}^{r,j}_{\Tilde{\vec{u}}} F^{j}_\ell \Tilde{\vec{w}}^\ell\brac{\vec{X}_{\vec{u}}^r \wedge_L\vec{v}}}, \int_0^1 \mathcal{R}^e \brac{d \psi_{\Tilde{\gamma}(s)} \partial_e \Tilde{\vec{w}}} \ ds \rangle_L\nonumber\\
    &+ 2\int_{\R^d} \langle \Dso\vec{v} \wedge_L\vec{v}^\alpha \Dso\vec{v}^\beta\vec{X}^{k,n}_{\vec{v}} \vec{E}^{r,j}_{\Tilde{\vec{u}}} F^{j}_\ell \Tilde{\vec{w}}^\ell\brac{\vec{X}_{\vec{u}}^r \wedge_L\vec{v}}, \int_0^1 \mathcal{R}^e \brac{d \psi_{\Tilde{\gamma}(s)} \partial_e \Tilde{\vec{w}}} \ ds \rangle_L\nonumber\\
    =& \int_{\R^d} \langle H_{\Dso}\brac{\vec{v} \wedge_L\vec{v}^\alpha \vec{X}^{k,n}_{\vec{v}} \vec{E}^{r,j}_{\Tilde{\vec{u}}} F^{j}_\ell  \brac{\vec{X}_{\vec{u}}^r \wedge_L\vec{v}},\Dso\vec{v}^\beta \Tilde{\vec{w}}^\ell}, \int_0^1 \mathcal{R}^e \brac{d \psi_{\Tilde{\gamma}(s)} \partial_e \Tilde{\vec{w}}} \ ds \rangle_L\label{end_2_1}\\
    &+ \int_{\R^d} \langle H_{\Dso}\brac{\vec{v} \wedge_L, \vec{v}^\alpha \vec{X}^{k,n}_{\vec{v}} \vec{E}^{r,j}_{\Tilde{\vec{u}}} F^{j}_\ell  \brac{\vec{X}_{\vec{u}}^r \wedge_L\vec{v}}}\Dso\vec{v}^\beta \Tilde{\vec{w}}^\ell, \int_0^1 \mathcal{R}^e \brac{d \psi_{\Tilde{\gamma}(s)} \partial_e \Tilde{\vec{w}}} \ ds \rangle_L\label{end_2_2}\\
    &- \int_{\R^d} \langle \vec{v} \wedge_LH_{\Dso}\brac{\vec{v}^\alpha \vec{X}^{k,n}_{\vec{v}} \vec{E}^{r,j}_{\Tilde{\vec{u}}} F^{j}_\ell  \brac{\vec{X}_{\vec{u}}^r \wedge_L\vec{v}},\Dso\vec{v}^\beta \Tilde{\vec{w}}^\ell}, \int_0^1 \mathcal{R}^e \brac{d \psi_{\Tilde{\gamma}(s)} \partial_e \Tilde{\vec{w}}} \ ds \rangle_L\label{end_2_3}\\
    &+ 2\int_{\R^d} \langle \Dso\vec{v} \wedge_L\vec{v}^\alpha \Dso\vec{v}^\beta\vec{X}^{k,n}_{\vec{v}} \vec{E}^{r,j}_{\Tilde{\vec{u}}} F^{j}_\ell \Tilde{\vec{w}}^\ell\brac{\vec{X}_{\vec{u}}^r \wedge_L\vec{v}}, \int_0^1 \mathcal{R}^e \brac{d \psi_{\Tilde{\gamma}(s)} \partial_e \Tilde{\vec{w}}} \ ds \rangle_L\label{end_2_4}.
\end{align}
For \eqref{end_2_1}, we use \cref{lem:prod_rule_frac} to get
\begin{align*}
    &\int_{\R^d} \left\langle H_{\Dso}\brac{\vec{v} \wedge_L\vec{v}^\alpha \vec{X}^{k,n}_{\vec{v}} \vec{E}^{r,j}_{\Tilde{\vec{u}}} F^{j}_\ell  \brac{\vec{X}_{\vec{u}}^r \wedge_L\vec{v}},\Dso\vec{v}^\beta \Tilde{\vec{w}}^\ell}, \int_0^1 \mathcal{R}^e \brac{d \psi_{\Tilde{\gamma}(s)} \partial_e \Tilde{\vec{w}}} \ ds \right\rangle_L\\
    \lesssim_{\Lambda} & \lVert \Dso^{1-\sigma}\brac{\vec{v} \wedge_L\vec{v}^\alpha \vec{X}^{k,n}_{\vec{v}} \vec{E}^{r,j}_{\Tilde{\vec{u}}} F^{j}_\ell  \brac{\vec{X}_{\vec{u}}^r \wedge_L\vec{v}}} \rVert_{L^{\frac{2d}{2(1-\sigma) - 1}}(\R^d)} \ \lVert\Dso^{\sigma}\brac{\Dso\vec{v}^\beta \Tilde{\vec{w}}^\ell} \rVert_{L^{\frac{2d}{d+2\sigma -1}}(\R^d)} \ \lVert \Dso \Tilde{\vec{w}} \rVert_{L^{2}(\R^d)}\\
    \lesssim_{\Lambda} & \lVert \Dso\brac{\vec{v} \wedge_L\vec{v}^\alpha \vec{X}^{k,n}_{\vec{v}} \vec{E}^{r,j}_{\Tilde{\vec{u}}} F^{j}_\ell  \brac{\vec{X}_{\vec{u}}^r \wedge_L\vec{v}}} \rVert_{L^{2d}(\R^d)} \ \lVert\Dso^{\sigma}\brac{\Dso\vec{v}^\beta \Tilde{\vec{w}}^\ell} \rVert_{L^{\frac{2d}{d+2\sigma -1}}(\R^d)} \ \lVert \Dso \Tilde{\vec{w}} \rVert_{L^{2}(\R^d)}\\
    \lesssim_{\Lambda} & \lVert \nabla\brac{\vec{v} \wedge_L\vec{v}^\alpha \vec{X}^{k,n}_{\vec{v}} \vec{E}^{r,j}_{\Tilde{\vec{u}}} F^{j}_\ell  \brac{\vec{X}_{\vec{u}}^r \wedge_L\vec{v}}} \rVert_{L^{2d}(\R^d)} \ \lVert\Dso^{\sigma}\brac{\Dso\vec{v}^\beta \Tilde{\vec{w}}^\ell} \rVert_{L^{\frac{2d}{d+2\sigma -1}}(\R^d)} \ \lVert \Dso \Tilde{\vec{w}} \rVert_{L^{2}(\R^d)}\\
    \lesssim_{\Lambda} & \brac{\lVert \Dso^{1+\sigma}\vec{u} \rVert_{L^{\frac{2d}{2(1+\sigma) - 1}}(\R^d)}^2 + \lVert \Dso^{1+\sigma}\vec{u} \rVert_{L^{\frac{2d}{2(1+\sigma) - 1}}(\R^d)}^2} \ \lVert \Dso \Tilde{\vec{w}} \rVert_{L^{2}(\R^d)}^2.
\end{align*}
For \eqref{end_2_2}, we use \cref{lem:gag_in} to get
\begin{align*}
    &\int_{\R^d} \langle H_{\Dso}\brac{\vec{v} \wedge_L, \vec{v}^\alpha \vec{X}^{k,n}_{\vec{v}} \vec{E}^{r,j}_{\Tilde{\vec{u}}} F^{j}_\ell  \brac{\vec{X}_{\vec{u}}^r \wedge_L\vec{v}}}\Dso\vec{v}^\beta \Tilde{\vec{w}}^\ell, \int_0^1 \mathcal{R}^e \brac{d \psi_{\Tilde{\gamma}(s)} \partial_e \Tilde{\vec{w}}} \ ds \rangle_L\\
    \lesssim_{\Lambda} & \lVert \Dso^{1-\sigma} \vec{v} \rVert_{L^{\frac{2d}{1-\sigma}}(\R^d)} \ \lVert \Dso^\sigma\brac{\vec{v}^\alpha \vec{X}^{k,n}_{\vec{v}} \vec{E}^{r,j}_{\Tilde{\vec{u}}} F^{j}_\ell  \brac{\vec{X}_{\vec{u}}^r \wedge_L\vec{v}}} \rVert_{L^{\frac{2d}{\sigma}}(\R^d)} \ \lVert \Dso\vec{v}^\beta \rVert_{L^{2d}(\R^d)} \ \lVert \Tilde{\vec{w}}^\ell \rVert_{L^{\frac{2d}{d-2}}(\R^d)} \ \lVert \Dso \Tilde{\vec{w}} \rVert_{L^{2}(\R^d)}\\
    \lesssim_{\Lambda} & \lVert \Dso \vec{v} \rVert_{L^{2d}(\R^d)}^{1-\sigma} \ \lVert \Dso\brac{\vec{v}^\alpha \vec{X}^{k,n}_{\vec{v}} \vec{E}^{r,j}_{\Tilde{\vec{u}}} F^{j}_\ell  \brac{\vec{X}_{\vec{u}}^r \wedge_L\vec{v}}} \rVert_{L^{2d}(\R^d)}^\sigma \ \lVert \Dso\vec{v} \rVert_{L^{2d}(\R^d)} \ \lVert \Dso \Tilde{\vec{w}} \rVert_{L^{2}(\R^d)}^2\\
    \lesssim_{\Lambda} & \lVert \Dso \vec{v} \rVert_{L^{2d}(\R^d)}^{1-\sigma} \ \lVert \nabla\brac{\vec{v}^\alpha \vec{X}^{k,n}_{\vec{v}} \vec{E}^{r,j}_{\Tilde{\vec{u}}} F^{j}_\ell  \brac{\vec{X}_{\vec{u}}^r \wedge_L\vec{v}}} \rVert_{L^{2d}(\R^d)}^\sigma \ \lVert \Dso\vec{v} \rVert_{L^{2d}(\R^d)} \ \lVert \Dso \Tilde{\vec{w}} \rVert_{L^{2}(\R^d)}^2\\
    \lesssim_{\Lambda} & \brac{\lVert \Dso \vec{u} \rVert_{L^{2d}(\R^d)}^2 + \lVert \Dso \vec{v} \rVert_{L^{2d}(\R^d)}^2} \ \lVert \Dso \Tilde{\vec{w}} \rVert_{L^{2}(\R^d)}^2.
\end{align*}
For \eqref{end_2_3}, using \cref{lem:prod_rule_frac}, we get
\begin{align*}
    &\int_{\R^d} \langle \vec{v} \wedge_LH_{\Dso}\brac{\vec{v}^\alpha \vec{X}^{k,n}_{\vec{v}} \vec{E}^{r,j}_{\Tilde{\vec{u}}} F^{j}_\ell  \brac{\vec{X}_{\vec{u}}^r \wedge_L\vec{v}},\Dso\vec{v}^\beta \Tilde{\vec{w}}^\ell}, \int_0^1 \mathcal{R}^e \brac{d \psi_{\Tilde{\gamma}(s)} \partial_e \Tilde{\vec{w}}} \ ds \rangle_L\\
    \lesssim_{\Lambda} & \lVert \Dso^{1-\sigma}\brac{\vec{v}^\alpha \vec{X}^{k,n}_{\vec{v}} \vec{E}^{r,j}_{\Tilde{\vec{u}}} F^{j}_\ell  \brac{\vec{X}_{\vec{u}}^r \wedge_L\vec{v}}} \rVert_{L^{\frac{2d}{2(1-\sigma)-1}}(\R^d)} \ \lVert \Dso^\sigma\brac{\Dso\vec{v}^\beta \Tilde{\vec{w}}^\ell} \rVert_{L^{\frac{2d}{d + 2\sigma -1}}(\R^d)} \ \lVert \Dso \Tilde{\vec{w}} \rVert_{L^{2}(\R^d)}\\
    \lesssim_{\Lambda} & \lVert \Dso\brac{\vec{v}^\alpha \vec{X}^{k,n}_{\vec{v}} \vec{E}^{r,j}_{\Tilde{\vec{u}}} F^{j}_\ell  \brac{\vec{X}_{\vec{u}}^r \wedge_L\vec{v}}} \rVert_{L^{2d}(\R^d)} \ \lVert \Dso^\sigma\brac{\Dso\vec{v}^\beta \Tilde{\vec{w}}^\ell} \rVert_{L^{\frac{2d}{d + 2\sigma -1}}(\R^d)} \ \lVert \Dso \Tilde{\vec{w}} \rVert_{L^{2}(\R^d)}\\
    \lesssim_{\Lambda} & \lVert \nabla\brac{\vec{v}^\alpha \vec{X}^{k,n}_{\vec{v}} \vec{E}^{r,j}_{\Tilde{\vec{u}}} F^{j}_\ell  \brac{\vec{X}_{\vec{u}}^r \wedge_L\vec{v}}} \rVert_{L^{2d}(\R^d)} \ \lVert \Dso^\sigma\brac{\Dso\vec{v}^\beta \Tilde{\vec{w}}^\ell} \rVert_{L^{\frac{2d}{d + 2\sigma -1}}(\R^d)} \ \lVert \Dso \Tilde{\vec{w}} \rVert_{L^{2}(\R^d)}\\
    \lesssim_{\Lambda} & \brac{\lVert \Dso^{1+\sigma} \vec{u} \rVert_{L^{\frac{2d}{2(1+\sigma) - 1}}(\R^d)}^2 + \lVert \Dso^{1+\sigma} \vec{v} \rVert_{L^{\frac{2d}{2(1+\sigma) - 1}}(\R^d)}^2} \ \lVert \Dso \Tilde{\vec{w}} \rVert_{L^{2}(\R^d)}^2.
\end{align*}
Finally, for \eqref{end_2_4} we have
\begin{align*}
    &\int_{\R^d} \langle \Dso\vec{v} \wedge_L\vec{v}^\alpha \Dso\vec{v}^\beta\vec{X}^{k,n}_{\vec{v}} \vec{E}^{r,j}_{\Tilde{\vec{u}}} F^{j}_\ell \Tilde{\vec{w}}^\ell\brac{\vec{X}_{\vec{u}}^r \wedge_L\vec{v}}, \int_0^1 \mathcal{R}^e \brac{d \psi_{\Tilde{\gamma}(s)} \partial_e \Tilde{\vec{w}}} \ ds \rangle_L\\
    \lesssim_{\Lambda} & \lVert \Dso\vec{v} \rVert_{L^{2d}(\R^d)}^2 \lVert \Tilde{\vec{w}} \rVert_{L^{\frac{2d}{d-2}}(\R^d)} \ \lVert \Dso \Tilde{\vec{w}} \rVert_{L^{2}(\R^d)}\\
    \lesssim_{\Lambda} & \lVert \Dso\vec{v} \rVert_{L^{2d}(\R^d)}^2 \ \lVert \Dso \Tilde{\vec{w}} \rVert_{L^{2}(\R^d)}^2.
\end{align*}
Therefore,
\begin{align*}
    &\abs{\int_{\R^d} \langle \vec{v} \wedge_LH_{\Dso}(\vec{v} \wedge_L, \int_0^1 \mathcal{R}^e \brac{d \psi_{\Tilde{\gamma}(s)} \partial_e \Tilde{\vec{w}}} \ ds), \vec{X}_{\vec{u}}^r \rangle_L \langle \vec{v} \wedge_L\Dso{\vec{v}, \vec{X}_{\vec{u}}^r} - \vec{X}_{\vec{v}}^r \rangle_L}\\
    \lesssim_{\Lambda} & \ \lVert \Dso\Tilde{\vec{w}} \rVert_{L^2(\R^d)}^2 \brac{\lVert \nabla^{1+\sigma}\vec{u} \rVert_{L^{\frac{2d}{2(1+\sigma)-1}}(\R^d)}^2 + \lVert \nabla^{1 + \sigma}\vec{v} \rVert_{L^{\frac{2d}{2(1+\sigma)-1}}(\R^d)}^2}.
 \end{align*}
\end{proof}
This concludes the estimates required for our proof of \cref{thm:1}.
\section{Uniqueness Proof of Theorem~\ref{thm:1}}
By Gr\"{o}nwall's inequality,
\begin{align*}
    E'(t) \lesssim_{\Lambda} \Sigma(t) E(t)
\end{align*}
implies that 
\begin{align*}
    E(t) \lesssim_ E(0) \exp \brac{\int_0^t \Sigma(s) \ ds}
\end{align*} 
for all $t \in [0,T]$.
By assumption \eqref{first} and the estimate for $\Sigma(t)$, $\exp \brac{\int_0^t \Sigma(s) \ ds}$ is finite. If $\vec{u}(0) = \vec{0}$, then $\Tilde{\vec{u}}(0) = \Tilde{\vec{v}}(0)$. Thus, $\nabla \Tilde{\vec{u}}(0) = \nabla \Tilde{\vec{v}}(0)$. Since $\vec{u}$ and $\vec{v}$ solve \eqref{eq:half-wave},
\begin{align*}
    \partial_t \Tilde{\vec{u}} - \partial_t \Tilde{\vec{v}}|_{t=0} =d\varphi_{\vec{u}}\brac{\vec{u} \wedge_L\Dso \vec{u}}|_{t=0} - d\varphi_{\vec{v}}\brac{\vec{v} \wedge_L\Dso \vec{v}}|_{t=0} = 0.
\end{align*}
Therefore, $E(0) = 0$. Hence, $E(t) = 0$ for all $t \in [0,T]$. Hence, $\Tilde{\vec{w}}$ is a constant, but since $\Tilde{\vec{u}}(0) - \Tilde{\vec{v}}(0) = 0$, we have that $\Tilde{\vec{w}} = 0$. Therefore, $\Tilde{\vec{u}} = \Tilde{\vec{v}}$, which implies that $\vec{u} = \vec{v}$.
\appendix

\appendix
\section{Reformulation of Hyperbolic Half-Wave Maps Equation} \label{reform}
Using $\vec{a} \wedge_L\brac{\vec{b} \wedge_L\vec{c}} = \langle \vec{b},\vec{a} \rangle_L \vec{c} - \langle \vec{c}, \vec{a} \rangle_L \vec{b}$,
\begin{align*}
    \vec{u}_{tt} 
    &= \partial_t \vec{u} \wedge_L\Dso \vec{u} +\vec{u} \wedge_L\Dso \partial_t \vec{u}\\
    &= \brac{\vec{u} \wedge_L\Dso \vec{u}} \wedge_L\Dso \vec{u} +\vec{u} \wedge_L\Dso \brac{\vec{u} \wedge_L\Dso \vec{u}}\\
    &= -\langle\vec{u}, \Dso \vec{u}\rangle_L \Dso \vec{u} + \langle \Dso \vec{u}, \Dso \vec{u} \rangle_L\vec{u} +\vec{u} \wedge_L\Dso \brac{\vec{u} \wedge_L\Dso \vec{u}}
\end{align*}
and
\begin{align*}
    -\Delta\vec{u} = \langle\vec{u},\vec{u} \rangle_L \Delta\vec{u} = -\vec{u} \wedge_L\brac{\vec{u} \wedge_L(-\Delta)\vec{u}} - \langle\vec{u}, -\Delta\vec{u} \rangle_L,
\end{align*}
thus
\begin{align*}
    \vec{u}_{tt} - \Delta \vec{u}
    = &-\langle\vec{u}, \Dso \vec{u}\rangle_L \Dso \vec{u} + \langle \Dso \vec{u}, \Dso \vec{u} \rangle_L\vec{u} +\vec{u} \wedge_L\Dso \brac{\vec{u} \wedge_L\Dso \vec{u}} \\
    & -\vec{u} \wedge_L\brac{\vec{u} \wedge_L(-\Delta)\vec{u}} + \langle\vec{u}, \Delta\vec{u} \rangle_L\\
    = &-\langle\vec{u}, \Dso \vec{u}\rangle_L \Dso \vec{u} + \langle \Dso \vec{u}, \Dso \vec{u} \rangle_L\vec{u} +\vec{u} \wedge_L\Dso \brac{\vec{u} \wedge_L\Dso \vec{u}} \\
    & -\vec{u} \wedge_L\brac{\vec{u} \wedge_L(-\Delta)\vec{u}} - \langle \nabla\vec{u}, \nabla\vec{u} \rangle_L.
\end{align*}
Observe that $\langle \Dso \vec{u}, \Dso \vec{u} \rangle_L\vec{u} - \langle\vec{u}, \Dso \vec{u}\rangle_L \Dso \vec{u} = a\vec{u} + b\vec{X}^1_{\vec{u}} + c\vec{X}^2_{\vec{u}}$ for $a,b,c \in \R$. So
\begin{align*}
    a &=\big\langle \langle \Dso \vec{u}, \Dso \vec{u} \rangle_L\vec{u} - \langle\vec{u}, \Dso \vec{u}\rangle_L \Dso \vec{u},\vec{u} \big\rangle_L\\
    &= \langle \langle\vec{u},\vec{u} \rangle_L\Dso \vec{u} - \langle\vec{u} , \Dso \vec{u} \rangle_L\vec{u}, \Dso \vec{u} \rangle_L\\
    &= \langle\vec{u} \wedge_L\brac{\vec{u} \wedge_L\Dso \vec{u}}, \Dso \vec{u} \rangle_L\\
    &= -\langle\vec{u} \wedge_L\Dso \vec{u},\vec{u} \wedge_L\Dso \vec{u} \rangle_L\\
    &= -\langle \partial_t \vec{u}, \partial_t \vec{u} \rangle_L.
\end{align*}
Hence,
\begin{align*}
    \begin{aligned}
    \vec{u}_{tt} - \Delta \vec{u} = &(|\partial_t \vec{u}|_L^2 - |\nabla \vec{u}|_L^2) \vec{u} \\
    &+ \vec{u} \wedge_L|\nabla|(\vec{u} \wedge_L|\nabla|\vec{u}) - \vec{u} \wedge_L(\vec{u} \wedge_L(-\Delta \vec{u})) \\
    &- P_\vec{u}(\langle\vec{u} , |\nabla| \vec{u}\rangle_L|\nabla|\vec{u}).
\end{aligned}
\end{align*}
\section{Computations on the hyperbolic paraboloid}

\begin{lemma} \label{tang_def}
    we show $\brac{T_{\vec{u}}\mathbb{H}^2}^\perp =  \operatorname{span}\{\vec{u}\}$.
\end{lemma}
\begin{proof}
Let $\Phi: \mathbb{R}^2 \to \mathbb{H}^2\subset \mathbb{R}^3$ be defined by $\Phi(p_1,p_2) = \brac{\sqrt{1+p_1^2 + p_2^2}, p_1, p_2}$. Note that $\Phi$ is a diffeomorphism. Then
\begin{align*}
    d\Phi_p 
    =
    \begin{pmatrix}
        \frac{p_1}{\sqrt{1 + p_1 + p_2}} & \frac{p_2}{\sqrt{1 + p_1 + p_2}} \\
        1 & 0 \\
        0 & 1
    \end{pmatrix}
    : T_p \mathbb{H}^2 \to T_{\Phi(p)} \mathbb{H}^2 \subset \R^3.
\end{align*}
So 
\begin{align*}
    \langle d\Phi_{\Phi^{-1}\vec{u}}\vec{x}, \vec{u} \rangle_L 
    &= 
    \left\langle
    \begin{pmatrix}
        \frac{\vec{u}_1\vec{x}_1}{\sqrt{1+\vec{u}_1^2 + \vec{u}_2^2}} + \frac{\vec{u}_2\vec{x}_2}{\sqrt{1+\vec{u}_1^2 + \vec{u}_2^2}}\\
        \vec{x}_1\\
        \vec{x}_2
    \end{pmatrix},
    \begin{pmatrix}
    \vec{u}_0\\
    \vec{u}_1\\
    \vec{u}_2
    \end{pmatrix}
    \right\rangle_L\\
    &= 
    \left\langle
    \begin{pmatrix}
        \frac{\vec{u}_1\vec{x}_1}{\sqrt{1+\vec{u}_1^2 + \vec{u}_2^2}} + \frac{\vec{u}_2\vec{x}_2}{\sqrt{1+\vec{u}_1^2 + \vec{u}_2^2}}\\
        \vec{x}_1\\
        \vec{x}_2
    \end{pmatrix},
    \begin{pmatrix}
    \sqrt{1+\vec{u}_1^2 + \vec{u}_2^2}\\
    \vec{u}_1\\
    \vec{u}_2
    \end{pmatrix}
    \right\rangle_L\\
    &= -\vec{u}_1\vec{x}_1 - \vec{u}_2\vec{x}_2 + \vec{u}_1\vec{x}_1 + \vec{u}_2\vec{x}_2\\
    &=0
\end{align*}
for $\vec{x} \in \mathbb{R}^2$. Hence, $\brac{T_{\vec{u}}\mathbb{H}^2}^\perp =  \operatorname{span}\{\vec{u}\}$.
\end{proof}
\begin{remark} \label{comp_ortho}
    We can construct a global orthonormal frame for $\mathbb{H}^2$. Note that
    \begin{align*}
        d\Phi_p 
        \begin{pmatrix}
            1 \\
            0          
        \end{pmatrix}
        =
        \begin{pmatrix}
            \frac{p_1}{p_0}\\
            1\\
            0
        \end{pmatrix}
    \end{align*}
    and 
    \begin{align*}
        d\Phi_p 
        \begin{pmatrix}
            0 \\
            1          
        \end{pmatrix}
        =
        \begin{pmatrix}
            \frac{p_2}{p_0}\\
            0\\
            1
        \end{pmatrix}.
    \end{align*}
    Using the the Gram-Schmidt process, we have that 
    \begin{align*}
        \vec{X}^1_p
        =
        \begin{pmatrix}
            \frac{p_1}{\sqrt{p^2_0 - p_1^2}}\\
            \frac{p_0}{\sqrt{p^2_0 - p_1^2}}\\
            0
        \end{pmatrix}
    \end{align*}
        and
    \begin{align*}
        \vec{X}^2_p
        =
        \begin{pmatrix}
            \frac{p_2 p_0}{\sqrt{p^2_0 - p_1^2}}\\
            \frac{p_1p_2}{\sqrt{p^2_0 - p_1^2}}\\
            1
        \end{pmatrix}
    \end{align*}
    is a global orthonormal frame for $\mathbb{H}^2$.
\end{remark}
\begin{definition}
    Let $M$ be a smooth manifold. A \textit{\textbf{connection}} in $M$ is a map 
    \begin{align*}
        \nabla: \mathfrak{X}(M) \times \mathfrak{X}(M) \to \mathfrak{X}(M)
    \end{align*}
    written as $(X,Y) \mapsto \nabla_X Y$, satisfying the following:\\
    a) For $f_1,f_2 \in C^\infty(M)$ and $X_1,X_2 \in \mathfrak{X}(M)$,
    \begin{align*}
        \nabla_{f_1X_1 + f_2X_2} Y = f_1 \nabla_{X_1} Y + f_2 \nabla_{X_2} Y
    \end{align*}\\
     b) For $a_1,a_2 \in \R$ and $Y_1,Y_2 \in \mathfrak{X}(M)$,
    \begin{align*}
        \nabla_X\brac{a_1Y_1+a_2Y_2} = a_1\nabla_X Y_1 + a_2 \nabla_X Y_2.
    \end{align*}\\
    c) For $f \in C^\infty(M)$,
    \begin{align}
        \nabla_X(fY) = f\nabla_X Y + (Xf)Y,
    \end{align}
    and $\nabla_X Y$ is called the \textit{\textbf{covariant derivative of $Y$ in the direction of $X$}}.
\end{definition}

\begin{definition}
    Let $\gamma:I \to M$ be a smooth curve. A \textbf{\textit{smooth vector field along $\gamma$}} is a smooth map $V:I \to M$ such that $V(t) \in T_{\gamma(t)}M$ for every $t \in I$. The set of all smooth vector fields along $\gamma$ is denoted by $\mathfrak{X}(\gamma)$.
\end{definition}
\begin{definition} \label{cov_def}
    We say $V \in \mathfrak{X}(\gamma)$ is \textit{\textbf{extendable}} if there is a $\Tilde{V} \in \mathfrak{X}(M)$ such that on a neighborhood $N \subset M$ that contains $\gamma(I)$, we have $V = \Tilde{V} \circ \gamma$. 
\end{definition}
\begin{lemma} \label{unique_cov}
    Let $M$ be a smooth manifold and let $\nabla$ be a connection in $TM$. For each smooth curve $\gamma:I \to M$, the connection determines a unique operator 
    \begin{align}
        D_t: \mathfrak{X}(\gamma) \to \mathfrak{X}(\gamma),
    \end{align}
    called the \textit{\textbf{covariant derivative along $\gamma$}}, satisfying the following\\
    (i) For $a,b \in \mathbb{R}$,
    \begin{align*}
        D_t(aV + bW) = aD_tV + bD_tW.
    \end{align*}
    (ii) For $f \in C^\infty (I)$,
    \begin{align*}
        D_t(fV) = f'V + fD_t V.
    \end{align*}\\
    (iii) If $V \in \mathfrak{X}(\gamma)$ is extendable, then for every extension $\Tilde{V}$ of $V$,
    \begin{align*}
        D_tV(t) = \nabla_{\Dot{\gamma}(t)}\Tilde{V}.
    \end{align*}
\end{lemma}
\begin{proof}
    Let $t_0 \in I^\circ$. There is a chart $(\psi, U)$ such that $\gamma(t) \in U$ for any $t \in (t_0 - \epsilon,t+\epsilon)$. So $V(t) \in T_{\gamma(t)}M$. Thus, 
    \begin{align}
        V(t) = \sum_{j} V^j(t)\partial_j|_{\gamma(t)}
    \end{align}
    for $t \in (t_0-\epsilon,t_0\epsilon)$, where $V^1,\dots, V^n$ are smooth real-valued functions defined on $(t_0-\epsilon,t_0+\epsilon)$. Since each $\partial_j|_{\gamma(t)}$ is extendable,
    \begin{align*}
        D_t\partial_j|_{\gamma(t)} 
        &= \nabla_{\dot{\gamma}(t)} \partial_j|_{\gamma(t)}\\
        &= \sum_{i}\Dot{\gamma}^i(t) \nabla_{\partial_i}\partial_j|_{\gamma(t)}\\
         &= \sum_{ik}\Dot{\gamma}^i(t) \Gamma^k_{ij}\brac{\gamma(t)}\partial_k|_{\gamma(t)}.
    \end{align*}
    Thus, 
    \begin{align} 
    \begin{split}
        D_tV(t) 
        &= \sum_{j} \Dot{V}^j(t)\partial_j|_{\gamma(t)} + V^j(t)\nabla_{\dot{\gamma}(t)} \partial_j|_{\gamma(t)}\\
        &= \sum_{j} \Dot{V}^j(t)\partial_j|_{\gamma(t)} + \sum_{ik}\Dot{\gamma}^i(t) V^j(t)\Gamma^k_{ij}\brac{\gamma(t)}\partial_k|_{\gamma(t)}\\
        &= \sum_{ijk} \big(\Dot{V}^k(t) + \Dot{\gamma}^i(t) V^j(t)\Gamma^k_{ij}\brac{\gamma(t)}\big)\partial_k|_{\gamma(t)}. \label{cov.exp}
    \end{split}
    \end{align}
    Now suppose there is another covariant derivative operator $G_t$ along $\gamma$. Then
    \begin{align*}
        G_tV(t) 
        &= \sum_{j} \Dot{V}^j(t)\partial_j|_{\gamma(t)} + V^j(t)\nabla_{\dot{\gamma}(t)} \partial_j|_{\gamma(t)}\\
        &= \sum_{j} \Dot{V}^j(t)\partial_j|_{\gamma(t)} + \sum_{ik}\Dot{\gamma}^i(t) V^j(t)\Gamma^k_{ij}\brac{\gamma(t)}\partial_k|_{\gamma(t)}\\
        &= \sum_{ijk} \big(\Dot{V}^k(t) + \Dot{\gamma}^i(t) V^j(t)\Gamma^k_{ij}\brac{\gamma(t)}\big)\partial_k|_{\gamma(t)}\\
        &= D_tV(t).
    \end{align*}
    Hence, $D_t$ is unique.

    For existence, the image of $\gamma$ can be covered by charts. In each chart, $D_t$ is defined as in \eqref{cov.exp} so that in the overlap, the definitions agree. All that is left is to verify the three properties. 

    For $a,b \in \R$,
    \begin{align*}
        D_t(aV(t) + bW(t))
        &= \sum_{ijk} \big(\brac{a\Dot{V}^k(t) + b\Dot{W}^k(t)} + \Dot{\gamma}^i(t) \brac{aV^j(t)+bW^j}\Gamma^k_{ij}\brac{\gamma(t)}\big)\partial_k|_{\gamma(t)}\\
        &= a\sum_{ijk} \big(\Dot{V}^k(t) + \Dot{\gamma}^i(t) V^j(t)\Gamma^k_{ij}\brac{\gamma(t)}\big)\partial_k|_{\gamma(t)} \\
        &+ b\sum_{ijk} \big(\Dot{W}^k(t) + \Dot{\gamma}^i(t) W^j(t)\Gamma^k_{ij}\brac{\gamma(t)}\big)\partial_k|_{\gamma(t)}\\
        &= a D_t V(t) + b D_t W(t).
    \end{align*}
    For $f \in C^\infty(I)$,
    \begin{align*}
        D_t(fV) 
        &= \sum_{ijk} \big(\Dot{\brac{fV}}^k(t) + \Dot{\gamma}^i(t) fV^j(t)\Gamma^k_{ij}\brac{\gamma(t)}\big)\partial_k|_{\gamma(t)}\\
        &= \sum_{ijk} \big(\Dot{f} V^k(t) + f\Dot{V}^k(t) + \Dot{\gamma}^i(t) fV^j(t)\Gamma^k_{ij}\brac{\gamma(t)}\big)\partial_k|_{\gamma(t)}\\
        &= \sum_{ijk} \Dot{f} V^k(t)\partial_k|_{\gamma(t)} + \sum_{ijk}f\big(\Dot{V}^k(t) + \Dot{\gamma}^i(t) V^j(t)\Gamma^k_{ij}\brac{\gamma(t)}\big)\partial_k|_{\gamma(t)}\\
        &= \Dot{f}V(t) + fD_tV(t).
    \end{align*}
    Finally, for every extension $\Tilde{V}$ of $V$, we have
    \begin{align*}
        \partial_t V^k = \partial_t \brac{\Tilde{V}^k \circ \gamma} = \sum_\ell\Dot{\gamma}^\ell(t)\partial_\ell \Tilde{V}^k.
    \end{align*}
    Thus,
    \begin{align} 
    \begin{split}
        D_tV(t) 
         &= \sum_{ijk} \big(\Dot{V}^k(t) + \Dot{\gamma}^i(t) V^j(t)\Gamma^k_{ij}\brac{\gamma(t)}\big)\partial_k|_{\gamma(t)}\\
         &= \sum_{ijk} \big(\sum_\ell\Dot{\gamma}^\ell(t)\partial_\ell \Tilde{V}^k + \Dot{\gamma}^i(t) V^j(t)\Gamma^k_{ij}\brac{\gamma(t)}\big)\partial_k|_{\gamma(t)}
    \end{split}
    \end{align}
    and 
    \begin{align}
    \begin{split}
        \nabla_{\Dot{\gamma}(t)} \Tilde{V}
        &= \sum_{j} \nabla_{\Dot{\gamma}(t)} \brac{V^j(t)\partial_j|_{\gamma(t)}}\\
        &= \sum_{j} \Dot{\gamma}(t)(\Tilde{V}^j) + \sum_{j}V^j(t)\nabla_{\Dot{\gamma}(t)} \partial_j|_{\gamma(t)}\\
        &= \sum_{j} \sum_\ell\Dot{\gamma}^\ell(t)\partial_\ell \Tilde{V}^j + \sum_{ik}\Dot{\gamma}^i(t) V^j(t)\Gamma^k_{ij}\brac{\gamma(t)}\partial_k|_{\gamma(t)}\\
        &= \sum_{ijk} \big(\sum_\ell\Dot{\gamma}^\ell(t)\partial_\ell \Tilde{V}^k + \Dot{\gamma}^i(t) V^j(t)\Gamma^k_{ij}\brac{\gamma(t)}\big)\partial_k|_{\gamma(t)}.
    \end{split}
    \end{align}
    Thus, $D_tV(t) = \nabla_{\Dot{\gamma}(t)} \Tilde{V}$.
\end{proof}
\begin{lemma} \label{def_cov}
    Let $N$ be an embedded Riemannian submanifold of a pseudo-Riemannian manifold $M$. Let $E^1, \dots, E^n$ be a local orthonormal frame of $N$. Define $\nabla: \mathfrak{X}(N) \times \mathfrak{X}(N) \to \mathfrak{X}(N)$ by
    \begin{align*}
        \nabla_X Y = \langle X(Y), E^1 \rangle E^1 + \dots + \langle X(Y) , E^n \rangle E^n.
    \end{align*}
    This is a connection in $M$.
\end{lemma}
\begin{proof}
    We need to show that $\nabla$ satisfies the three properties in \cref{cov_def}. 
    First, for $f_1, f_2 \in C^\infty(N)$ and $X_1,X_2 \in \mathfrak{X}(N)$, we have
    \begin{align*}
        \nabla_{f_1X_1 + f_2X_2} Y 
        & = \langle (f_1X_1 + f_2X_2)(Y), E^j \rangle E^j\\
        & = f_1\langle X_1(Y), E^j \rangle E^j + f_2\langle X_2(Y), E^j \rangle E^j\\
        & = f_1\nabla_{X_1}Y + f_2\nabla_{X_2}Y.
    \end{align*}
    Second, for $a_1,a_2 \in \R$ and $Y_1,Y_2 \in \mathfrak{X}(N)$,
    \begin{align*}
        \nabla_X(a_1Y_1 + a_2Y_2) 
        &= \langle X(a_1Y_1 + b_2Y_2), E^j \rangle E^j\\
        &= a_1\langle X(Y_1), E^j \rangle E^j + a_2\langle X(Y_2), E^j \rangle E^j\\
        &= a_1\nabla_X Y_1 + a_2\nabla_X Y_2.
    \end{align*}
    Finally, $f \in C^\infty(M)$,
    \begin{align*}
        \nabla_X(fY) 
        &= \langle X(fY), E^j \rangle E^j\\
        &= \langle (Xf)Y, E^j \rangle E^j + \langle fX(Y), E^j \rangle E^j\\ 
        &= Xf\langle Y, E^j \rangle E^j + f\langle X(Y), E^j \rangle E^j\\
        &=  (Xf)Y +  f \nabla_X Y.
    \end{align*}
    Hence, $\nabla$ is a connection in $N$.
\end{proof}
\begin{lemma} \label{cov_der_tang}
    Let $N$ be an embedded Riemannian submanifold of a pseudo-Riemannian manifold $M$ and let $\gamma: I \to N$ be a smooth curve. Let $E^1, \dots, E^n$ be a local orthonormal frame of $N$. Take the connection $\nabla: \mathfrak{X}(N) \times \mathfrak{X}(N) \to \mathfrak{X}(N)$ defined in \cref{def_cov}. Then the operator $D^\alpha: \mathfrak{X}(\gamma) \to \mathfrak{X}(\gamma)$ defined by 
    \begin{align*}
        D_t V = \langle \partial_\alpha V, E^1 \rangle E^1 + \dots + \langle \partial_\alpha V, E^n \rangle E^n
    \end{align*}
    is the covariant derivative along $\gamma$ with respect to $\nabla$.
\end{lemma}
\begin{proof}
    We have to show that the above definition satisfies the 3 properties in \cref{unique_cov}.
    For $a,b \in \R$, we have
    \begin{align*}
        D_t(aV + aW) 
        &= \langle \partial_\alpha (aV + bW), E^j \rangle E^j\\
        &= a\langle \partial_\alpha V, E^j \rangle E^j + b\langle \partial_\alpha W, E^j \rangle E^j\\
        &= aD_tV + bD_tW.
    \end{align*}
    By the product rule of $\partial_t$, for $f \in C^\infty(I)$ we have that
    \begin{align*}
        D_t(fV) 
        &= \langle \partial_t (fV), E^j \rangle E^j\\
        &= \langle f'V + f \partial_t V, E^j \rangle E^j\\
        &= f' \langle V, E^j \rangle E^j + f \langle \partial_t V, E^j \rangle E^j\\
        &= f' V + f D_t V.
    \end{align*}
    Now let $V \in \mathfrak{X}(\gamma)$ and let $\Tilde{V}$ be an extension of $V$. Since $\Tilde{V}$ is an extension of $V$, we have
    \begin{align*}
        \partial_t V = \partial_t \brac{\Tilde{V} \circ \gamma} = \sum_\ell\Dot{\gamma}^\ell(t)\partial_\ell \Tilde{V}.
    \end{align*}
    Then we have
    \begin{align*}
        D_t V 
        &= \langle \partial_t V, E^j \rangle E^j\\
        &= \Big\langle \sum_\ell\Dot{\gamma(t)}^\ell(t)\partial_\ell \Tilde{V}, E^j \Big\rangle E^j\\
    \end{align*}
    and
     \begin{align*}
        \nabla_{\Dot{\gamma}(t)} \Tilde{V} 
        &= \langle \dot{\gamma(t)}(\Tilde{V}), E^j \rangle E^j\\
        &= \Big\langle \sum_\ell\Dot{\gamma}^\ell(t)\partial_\ell \Tilde{V}, E^j \Big \rangle E^j.
    \end{align*}
    Hence, $D_t V = \nabla_{\Dot{\gamma}(t)} \Tilde{V}$. Thus, $D_t$, as defined, is the covariant derivative along $\gamma$ determined by $\nabla$.
\end{proof}
\begin{lemma} \label{pullback_conn}
   Let $M$ and $\Tilde{M}$ be smooth manifolds. If $\Tilde{\nabla}$ is a connection in $T\Tilde{M}$ and $\varphi$ is a diffemorphism, then the map $\nabla: \mathfrak{X}(M) \times \mathfrak{X}(M) \to \mathfrak{X}(M)$ defined by 
   \begin{align*}
       \nabla_XY = d\varphi^{-1}\brac{\Tilde{\nabla}_{\Tilde{X}}\Tilde{Y}}
   \end{align*}
   is a connection in $TM$, where $\Tilde{X} = d\varphi X$, $\Tilde{Y} = d\varphi Y \in \mathfrak{X}(\Tilde{M})$.
\end{lemma} 
\begin{proof}
    Let $a,b \in R$ and $Y_1,Y_2 \in \mathfrak{X}(M)$. Then
    \begin{align*}
       \nabla_X \brac{a Y_1 + b Y_2}
       & = d\varphi^{-1} \brac{\Tilde{\nabla}_{\Tilde{X}} a Y_1 + b Y_2}\\
       & = a \ d\varphi^{-1} \brac{\Tilde{\nabla}_{\Tilde{X}} Y_1} + b \ d\varphi^{-1} \brac{\Tilde{\nabla}_{\Tilde{X}} Y_2}\\
       & = a\nabla_X \brac{Y_1} + b \nabla_X \brac{Y_2}.
    \end{align*}
    Let $f_1,f_2 \in C^\infty(M)$ and $X_1,X_2 \in \mathfrak{X}(M)$. Note that 
    \begin{align*}
        d\varphi_p\brac{f(p)X_p} 
        & = f(p) d\varphi_p X_p \\
        & =  (f \circ \varphi^{-1} \circ \varphi)(p) \Tilde{X}_{\varphi(p)}\\
        &= (f \circ \varphi^{-1})(\Tilde{p}) \Tilde{X}_{\Tilde{p}}.
    \end{align*}
    Let us define $\Tilde{f} = f \circ \varphi^{-1}$. Hence, $d\varphi\brac{fX} = \Tilde{f} \Tilde{X}$. Thus,
    \begin{align*}
        \nabla_{fX}Y 
        & = d\varphi^{-1}\Tilde{\nabla}_{\Tilde{f}\Tilde{X}} \Tilde{Y} \\
        & = f d\varphi^{-1}\Tilde{\Tilde{X}} \Tilde{Y}\\
        & = f \nabla_{X} Y.
    \end{align*}
    For $f \in C^\infty(M)$, we have
    \begin{align*}
        \nabla_X \brac{fY} 
        & = d\varphi^{-1}\brac{\Tilde{\nabla}_{\Tilde{X}}\Tilde{f}\Tilde{Y}}\\
        & = d\varphi^{-1}\brac{\Tilde{f}\Tilde{\nabla}_{\Tilde{X}}\Tilde{Y} + \Tilde{X}(\Tilde{f})\Tilde{Y}}\\
        & = f\nabla_{X}Y + X(f)Y.
    \end{align*}
    Hence, $\nabla$ is a connection on $TM$.
\end{proof}

\begin{lemma}\label{no_sec_der}
    Let $M$ and $\Tilde{M}$ be smooth manifolds, and let $\varphi: M \to \Tilde{M}$ is a diffeomorphism. Let $\Tilde{\nabla}$ be a connection in $T\Tilde{M}$ and let $\nabla$ be the connection defined in \cref{pullback_conn}. Let $\gamma: I \to M$ be a smooth curve. Then
    \begin{align*}
        d\varphi D_t V = \Tilde{D}_t\brac{d\varphi V},
   \end{align*}
   where $D_t$ is covariant differentiation along $\gamma$ with respect to $\nabla$, and $\Tilde{D}_t$ is covariant differentiation along $\varphi \circ \gamma$ with respect to $\Tilde{\nabla}$.
\end{lemma}
\begin{proof}
    we show that $D_t = d\varphi^{-1} \Tilde{D}_t d\varphi$ by uniqueness (\cref{unique_cov}), which will yield the desired equality $d\varphi D_t V = \Tilde{D}_t\brac{d\varphi V}$.

    First, we show that $d\varphi^{-1} \Tilde{D}_t d\varphi$ is linear over $\R$. Let $a,b \in R$. Then
    \begin{align*}
        d\varphi^{-1} \Tilde{D}_t d\varphi(aV + bW) 
        & =  a \ d\varphi^{-1} \Tilde{D}_t d\varphi V + b \ d\varphi^{-1} \Tilde{D}_t d\varphi W.
    \end{align*}
    Now for $f \in C^\infty(I)$, we have
    \begin{align*}
        d\varphi^{-1} \Tilde{D}_t d\varphi(fV)
        & = d\varphi^{-1} \Tilde{D}_t (f d\varphi V)\\
        & = d\varphi^{-1} (f' d\varphi V) + d\varphi^{-1} (f \Tilde{D}_t d\varphi V)\\
        & = f' V + f \ d\varphi^{-1} \Tilde{D}_t d\varphi V.
    \end{align*}
    Finally, for $V \in \mathfrak{X}(\gamma)$ extendable, we have that 
    $$
    d\varphi_\gamma V = d\varphi_\gamma (\Tilde{V} \circ \gamma) = (d\varphi_{\varphi^{-1} \Tilde{\gamma}} \Tilde{V} \circ \varphi^{-1}) \circ \Tilde{\gamma},
    $$
    for an extension $\Tilde{V}$ of $V$. So $\overline{V} = d\varphi_{\varphi^{-1}} \Tilde{V} \circ \varphi^{-1} \in \mathfrak{X}(\Tilde{M})$ is an extension of $d\varphi_{\varphi^{-1} \Tilde{\gamma}} V \in \mathfrak{X}(\Tilde{\gamma})$. Hence,
    \begin{align*}
        d\varphi^{-1}_{\varphi \circ \gamma} \Tilde{D}_t d\varphi_\gamma V
        & = d\varphi^{-1}_{\varphi \circ \gamma} \Tilde{D}_t d\varphi_{\varphi^{-1} \Tilde{\gamma}} V\\
        & = d\varphi^{-1}_{\varphi \circ \gamma} \Tilde{\nabla}_{\Tilde{\gamma}'} \overline{V}\\
        & = \nabla_{\Dot\gamma} \Tilde{V},
    \end{align*}
    where the last line follows from \cref{pullback_conn}. Thus, by uniqueness, $D_t = d\varphi^{-1} \Tilde{D}_t d\varphi$ and the result follows.
\end{proof}
\begin{lemma}\label{lem:cov der is tang der} \cite{Lee}
    (page 124, prop 5.12) Suppose $M$ is an embedded (pseudo-)Riemannian submanifold of a (pseudo-)Euclidean space. Then the Levi-Civita connection on $M$ is equal to the tangential connection $\nabla^\top$.
\end{lemma}

\begin{lemma}[\textbf{The Gauss Formula Along a Curve}] \label{lem:gauss formula}\cite{Lee}
    (page 229, corollary 8.3) Suppose $(M,g)$ is an embedded Riemannian submanifold of a Riemannian or pseudo-Riemannian manifold $(\Tilde{M},\Tilde{g})$, and $\gamma: I \to M$ is a smooth curve. If $X$ is a smooth vector field along $\gamma$ that is everywhere tangent to $M$, then 
    \begin{align*}
        \Tilde{D}_tX = D_tX + A(\gamma',X),
    \end{align*}
    where $A$ is the second fundamental form of $M$.
\end{lemma}
\newpage
\bibliographystyle{plain}
\bibliography{citation}
\end{document}